\documentclass[leqno,11pt]{amsart}
\usepackage{amssymb, amsmath}

\usepackage{color,ulem,textcomp}

\numberwithin{equation}{section}
\usepackage{enumerate}
\usepackage{graphicx}
\usepackage{cancel}
\theoremstyle{definition}
\theoremstyle{plain}
\newtheorem{thm}{Theorem}[section]%[chapter]%[section]%!!the counter really doesn't follow the chapters!!
\newtheorem{Prop}[thm]{Proposition}
\newtheorem{lem}[thm]{Lemma}

 \newtheorem{Thm}{Theorem}[section]
 \newtheorem{Rmk}[thm]{Remark}
 \newtheorem{Lem}[thm]{Lemma}
\def\N {\mathbf{N}}
\def\R {\mathbf{R}}
\def\D {\mathbf{D}}

\def\T {\mathbf{T}}

\def\Z {\mathbf{Z}}

\def\cA {\mathcal{A}}
\def\cB {\mathcal{B}}

\def\cD {\mathcal{D}}

\def\cG {\mathcal{G}}
\def\cL {\mathcal{L}}

\def\cM {\mathcal{M}}

\def\cZ {\mathcal{Z}}

\def\b {{\beta}}

\def\eps {{\varepsilon}}
\def\e {{\varepsilon}}

\def\k {{\kappa}}

\def\indc {{\bf 1}}

\def\d {{\partial}}

\newcommand{\Ker}{\operatorname{Ker}}

\newcommand{\ba}{\begin{aligned}}
\newcommand{\ea}{\end{aligned}}

\newcommand{\be}{\begin{equation}}
\newcommand{\ee}{\end{equation}}

\newcommand{\petit}{\xi}
\newcommand{\brown}{\Xi}

\let\dsp=\displaystyle
\numberwithin{equation}{section}
\begin{document}%%%%%%%%%%%%%%%%%%%%%%%%%%%%%%%%
%%%%%%%%%%%%%%%%%%%%%%%%%%%%%%%%%%%%%%%%%%%%%%%%

%\linenumbers

	\title[From hard-spheres to Brownian motion]{The Brownian motion  as the  limit of a   deterministic system of hard-spheres}
	\author{Thierry Bodineau, Isabelle Gallagher and Laure Saint-Raymond}
	\date{\today}
		
\begin{abstract}
We provide a rigorous derivation of the brownian motion  as the  limit of  a deterministic system of hard-spheres as the number of particles $N$ goes to infinity and their diameter $\varepsilon$ simultaneously goes to $0,$ in the fast relaxation limit $\alpha = N\varepsilon^{d-1}\to \infty $ (with a suitable 
 {diffusive} scaling of the observation time).

\noindent As suggested by Hilbert in his sixth problem, we rely on 
a kinetic formulation as an intermediate level of description between the microscopic and the fluid descriptions: we use indeed 
the linear Boltzmann equation to describe one tagged particle in a gas close to global equilibrium. 
 Our proof is based on the fundamental ideas of Lanford. 
The main novelty here is the detailed study of the branching process, leading to   explicit estimates on pathological collision trees.

\end{abstract}
	
	\maketitle
	\section{Introduction}

\subsection{From microscopic to macroscopic models}

We are interested here in describing the macroscopic behavior of a gas consisting of $N$ interacting particles of mass $m$ in a domain~$\D $ of~$ \R^d$, with positions and velocities $(x_i, v_i)_{1\leq i \leq N} \in (\D \times \R^d)^N$, the dynamics of which is given by
\begin{equation}\label{potentialcase}
\begin{aligned}
 \frac  { d x_i }{d t} = \displaystyle v_i \, , \qquad 
  m\frac  { d  v_i }{d  t} = -\frac1 \varepsilon  \sum_{ j \neq i} \nabla \Phi\Big(\frac{ x_i - x_j}\varepsilon\Big)\, ,
\end{aligned}
\end{equation}
for some compactly supported potential $\Phi$, meaning that the scale for the microscopic interactions is typically $\eps$. We shall actually mainly be interested in the case when the interactions are pointwise (hard-sphere interactions):  the presentation of that model is postponed  to Section~\ref{strategyandmainresults}, see~(\ref{hard-spheres1}),(\ref{hard-spheres2}).

\bigskip
\noindent
In the limit when $N\to \infty$, $\eps \to 0$ with $N\eps^d = O(1)$,  it is expected that the distribution of particles averages out to a local equilibrium.
The microscopic fluxes in the conservations of empirical density, momentum and energy should therefore converge to some macroscopic fluxes, and we should end up with a macroscopic system of equations (depending on the observation time and length scales).
 However the complexity of the problem is such that there is no complete derivation of any fluid model starting from the full deterministic Hamiltonian dynamics, regardless of the regime (we refer to \cite{OVY, EMY, QY}  for partial results obtained by adding a small noise in the microscopic dynamics).

\bigskip
\noindent
For rarefied gases, i.e. under the assumption that there is asymptotically no excluded volume~$N\eps^d \ll1$, Boltzmann   introduced an intermediate level of description, referred to as kinetic theory, in which the state of the gas is described by the statistical distribution $f$ of the position and velocity of a typical particle.
In the Boltzmann-Grad scaling~$\alpha \equiv N\eps^{d-1}  = O(1)$, we indeed expect 
the particles to undergo $\alpha$ collisions per unit time in average and all the correlations to be negligible.
Therefore, depending on the initial distribution of positions and velocities in the $2dN$-phase space, the 1-particle density $f$ should satisfy a closed evolution equation where the inverse mean free path $\alpha$ measures the collision rate.

\noindent In the fast relaxation limit $\alpha \to \infty$, we then expect the system to relax towards local thermodynamic equilibrium, and the dynamics to be described by some macroscopic  {equations} (depending on the observation time and length scales).

\begin{figure}[h] %  figure placement: here, top, bottom, or page
   \centering
  \includegraphics[width=3in]{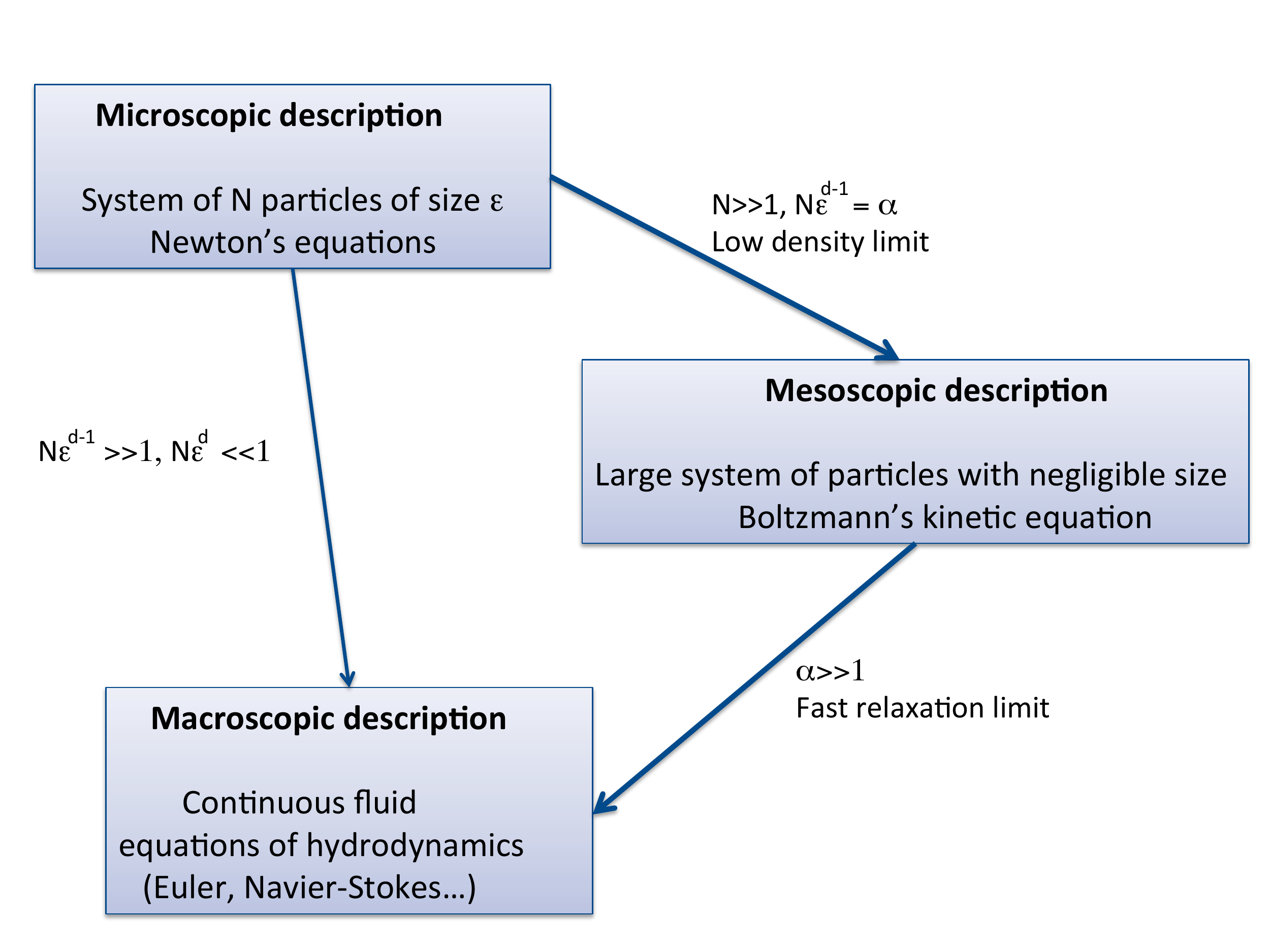} 
\caption{\label{hilbert}  {\small Fluid equations of hydrodynamics can be recovered directly from the microscopic system or in a two-step limit using Boltzmann's kinetic equation  as a mesoscopic description. Note that these two procedures may lead to limiting equations with different transport coefficients since the kinetic equation describes only  perfect gases (without excluded volume in the state relation).}  }
 \end{figure}

\bigskip
\noindent
One of the major difficulties to achieve this program using kinetic models as an intermediate description  is to justify the low density limit $\alpha \equiv N\eps^{d-1} $ on time intervals independent of~$\alpha$. Note that this step is also the most complicated one from the conceptual viewpoint  as it should explain the appearance of irreversibility, and dissipation mechanisms.

\bigskip
\noindent
The best result concerning the low density limit, which is due to Lanford in the case of hard-spheres~\cite{lanford} and King \cite{king} for more general potentials  (see also \cite{CIP,uchiyama,GSRT} for a complete proof)  is indeed valid only for short times, i.e. breaks down before any relaxation can be observed. The result may indeed be stated as follows \cite{GSRT} (see also \cite{PSS} for less restrictive assumptions on the potential $\Phi$). 
\begin{Thm}\label{lanford-thm}
Consider a system of $N$ particles interacting 
\begin{itemize}
\item either as hard-spheres of diameter $\eps$ 
\item or  as in~{\rm(\ref{potentialcase})} via a repulsive   potential $\Phi $, with compact support, radial and singular at 0, and such that the scattering of particles can be parametrized by their deflection angle. 
\end{itemize}

\noindent  Let $f_0:\R^{2d} \mapsto \R^+$ be a continuous density of probability such that
 $$\big\| f_0 \exp(\frac\beta 2|v|^2 ) \big\|_{L^\infty(\R^d_x \times \R^d_v)}\leq \exp (-\mu)$$
 for some $\beta>0, \mu \in \R$.
 
\noindent Assume that the $N$ particles are  initially distributed according to $f_0$ and ``independent". Then, there exists some $T^* >0$ (depending only on $\beta$ and $\mu$) such that, in the Boltzmann-Grad limit~$N \to \infty, \eps \to 0,  \, N\eps^{d-1} = \alpha$, the distribution function of the particles  converges uniformly on~$[0,T^*/\alpha] \times \R^{2d}$  to the solution of the Boltzmann equation
\begin{equation}
\label{boltzmann}
\begin{aligned}
& \d_t f +v\cdot \nabla_x f=\alpha Q(f,f),\\
& Q(f,f)(v):=\iint_{{\mathbf S}^{d-1} \times \R^d} [f(v^*)f(v^*_1)-f(v)f(v_1)]  \,b(v-v_1,\nu) \, dv_1 d\nu\,\\
& v^*=v +\nu \cdot (v_1-v)  \, \nu \, , 
	\quad v^*_1=v_1 - \nu \cdot(v_1-v) \, \nu \,,
 \end{aligned}
\end{equation}
 with a locally bounded cross-section $b$ depending on~$\Phi $ implicitly,  and with initial data $f_0$. In the case of a hard-sphere interaction, the cross section is given by
 $$
 b(v-v_1,\nu) = \big((v-v_1)  \cdot \nu\big)_+ \, .
 $$
\end{Thm}

\noindent Here, by ``independent", we mean   that the initial $N$-particle distribution satisfies a chaos property, namely that the correlations vanish asymptotically. Typically the distribution is obtained by factorization, and conditioning on energy surfaces (see~\cite{GSRT} and references therein). In the case of hard-spheres for instance, one would have
$$ f_{N}^0  = \cZ_N^{-1} f_0^{\otimes N}  \indc_{{\mathcal D}_\eps^N} \,,$$
%$$
% f_{N}^0 (Z_N)= N M_{N,\beta} (Z_N) +\Big (\sum_{j=1}^N \varphi_0(z_j) \Big)M_{N,\beta} (Z_N) 
%$$
with
$$
{\mathcal D}_\eps^N := \big\{ (x_1,v_1,\dots , x_N,v_N) \in {\mathbb T}^{dN} \times  {\mathbb R}^{dN}    \, / \,  \forall i \neq j \, ,\quad |x_i - x_j| > \eps \big\}
$$
and
$$
f_0^{\otimes N} (x_1,v_1,\dots , x_N,v_N):= \prod_{i=1}^N f_0 (x_i,v_i) \, , 
$$
while~$\cZ_N$ normalizes the integral of~$ f_{N}^0$ to 1.

%\bigskip
%\noindent The assumption on the potential (namely the fact that the scattering of particles can be parametrized by their deviation angle) holds for a large class of monotonic potentials, for more details we refer to~\cite{GSRT},~\cite{PSS}.

\bigskip
\noindent
The main difficulty to prove convergence  for longer time intervals consists in ruling out the possibility of spatial concentrations of the density leading to some pathological collision process.

\bigskip

\subsection{Linear regimes}

In this paper, we overcome this difficulty by considering a good notion of fluctuation around global equilibrium for the system of interacting particles. In this way we get a complete derivation of the diffusion limit from the hard-sphere system in a linear regime. Of course,  in this framework one cannot hope to retrieve a model for the full (nonlinear) gas dynamics, but -- as far as we know -- this is the very first  result describing the Brownian motion as the limit of a deterministic classical system of interacting particles.

\noindent
The main difficulty here is to justify the approximation by the linear Boltzmann equation  
\begin{equation}
\label{linear-boltzmann}
\begin{aligned}
& \d_t \varphi_\alpha  +v\cdot \nabla_x \varphi_\alpha =-\alpha \, \mathcal{L} \varphi_\alpha \\
& \mathcal{L} \varphi_\alpha (v) :=\iint [\varphi_\alpha(v) - \varphi_\alpha(v')] \, M_\beta(v_1) \,b(v-v_1,\nu) \, dv_1 d\nu \\
& M_\beta (v) := \left( \frac\beta{2\pi}\right)^{\frac d2}\exp \left( -\frac\beta2 |v|^2\right) \, , \quad \beta>0  \, ,
 \end{aligned}
\end{equation}
for times diverging as $\alpha$ when $\alpha\to \infty$. Indeed, in the diffusive regime, the convergence  of the 
Markov process associated to the 
linear Boltzmann operator $\mathcal{L}$  towards the Brownian motion is by now a classical result \cite{KLO}.

\section{Strategy and main results}\label{strategyandmainresults}

\noindent A good notion of fluctuation is obtained by considering the motion of a tagged particle (or possibly a finite set of tagged particles) in a gas of $N$ particles initially at equilibrium (or close to equilibrium), in the limit $N\to \infty$.

\subsection{The Lorentz gas} {If the background particles are infinitely heavier than the tagged particle then the dynamics can be approximated by a Lorentz gas, i.e. by the motion of the tagged particle in a frozen background.}
The linear Boltzmann equation has been derived (globally in time) from the dynamics of a tagged particle in a {low density} Lorentz   gas, meaning that
	\begin{itemize}
	\item the obstacles are distributed randomly according to some Poisson distribution;
	\item the obstacles have no dynamics, in particular they do not feel the effect of  collisions with the tagged particle.
	\end{itemize}
	
	\noindent This problem, suggested by Lorentz \cite{lorentz} at the beginning of the twentieth century  to study the motion of electrons in metals, is the {core}   of a number of works, and the corresponding literature includes a large variety of contributions. 
We do not intend to be exhaustive here 
and refer the reader to the book by Spohn \cite[Chapter 8]{Spohnbook} for a survey on this topic. We state one basic result due to Gallavotti \cite{gallavotti} in the low density limit
and   then   indicate some of the many important research directions.
		
	\begin{Thm} 
Consider randomly distributed scatterers with radius $\eps$ in $\R^d$ according to a  Poisson distribution   of parameter~$ \alpha \eps^{1-d}$. Let $T_\eps ^t$ be the flow of a point particle reflected at the boundary of these scatterers.
	For a given continuous initial datum $f_0 \in L^1 \cap L^\infty (\R^{2d})$, we define
	$$f_\eps (t,x,v) := {\mathbb E} [f_0( T_\eps ^{-t} (x,v))]\,.$$
	Then, for any time $T>0$, $f_\eps$ converges to the solution $f$ of the linear Boltzmann equation~{\rm(\ref{linear-boltzmann})},  with hard-sphere cross-section, in $L^\infty([0,T], L^1(\R^{2d}))$.
	\end{Thm}
\noindent 	
{A refinement of this result can be found for instance in \cite{S} in terms of convergence of path measures (and not only of the mean density), as well as in \cite{BBS} where the convergence is proven for typical scatterer configurations (and not only in average).}
%: instead of characterizing only the mean motion, Spohn proves indeed   the convergence of the process, as well as the fact that  the limiting process is Markovian if and only if the density of the scatterers converges in probability to its mean.
	
\medskip

\noindent These convergence statements lead naturally to various questions concerning
	\begin{itemize}
	\item the assumptions on the microscopic potential of interaction,
	\item the role of randomness for the distribution of scatterers,
	\item the long time behavior of the system, in particular the relaxation towards thermodynamic equilibrium and hydrodynamic limits.
	\end{itemize}

	\medskip

\noindent 	The first  point was addressed by Desvillettes, Pulvirenti and Ricci \cite{DP, DR}. Their goal was to derive ``singular " kinetic equations such as the linear Boltzmann equation without angular cut-off or the Fokker-Planck equation, from a system of particles with long-range interactions. They have obtained partial results in this direction, insofar as they can consider only asymptotically long-range interactions. Due to the fact that the range of the potential is infinite in the limit, the test particle interacts typically with infinitely many obstacles. Thus the set of bad configurations of the scatterers (such as the set of configurations yielding recollisions) preventing the Markov property of the limit must be estimated explicitly. Even though the long-range tails add  a very small contribution to the total force for each typical scatterer distribution, the non grazing collisions generate an exponential instability making the two trajectories (with and without cut-off) very different. 
The complete derivation of the linear Boltzmann equation for long-range interactions is therefore still open.
%, although the methods developed in this paper could be a first angle of attack, at least for sufficiently decaying potentials~\cite{GSRS}.  
\medskip

\noindent 	
{It is often   appropriate from a physical point of view to consider more general distributions of obstacles than the Poisson distribution. In particular, in the original problem of Lorentz, the atoms of metal are distributed on a periodic network.
For the two-dimensional periodic Lorentz gas with fixed scatterer size, Bunimovich and Sinai \cite{BS} have shown the convergence, after a suitable time rescaling, of the tagged particle to a brownian motion. Their method relies on techniques from ergodic theory~: it uses the fact that the mapping carrying a phase point on the boundary of a scatterer to the next phase point along its trajectory  can be represented by a symbolic dynamics on a countable alphabet which is an ergodic Markov chain on a finite state space. 
Another important research direction, initiated by Golse, is to consider the periodic Lorentz gas in the Boltzmann-Grad limit. In this case there can be infinitely long free flight paths and the 
linear Boltzmann equation is no longer a valid limit~\cite{golse, CG2, Marklof, MarklofStrombergsson}, 
but the convergence toward a Brownian motion can be recovered after an appropriate superdiffusive rescaling \cite{MarklofToth}.

 \medskip

\noindent 
In \cite{ESY, erdos}, Erd\"os, Salmhofer and Yau  obtained the counterpart of the long time behavior for random quantum systems.
Our approach is closer to their method than to the ones used for the periodic Lorentz gas
(even though the setting of \cite{ESY} deals with a fixed random distribution of obstacles and a slightly different regime, known as weak coupling limit).  Their proof is indeed  based on a careful analysis of Duhamel's formula in combination with a renormalization of the propagator and 
   stopping rules to control recollisions. We refer also to~\cite{roeck_frohlich} for further developments on the quantum case.

\subsection{Interacting gas of particles}
%{An alternative strategy based on the maximum principle} 

We adopt here a  different point of view, and consider a deterministic system of $N$ hard-spheres, meaning that the tagged particle is identical to the particles of the background, interacting according to the same collision laws.
 In this paper, we will focus on the case $d \geq 2$ (and refer to~\cite{Szasz_Toth} for results in the case~$d=1$).
}

 \noindent On the one hand,  the problem seems more difficult than the Lorentz gas insofar as
 the background has its own dynamics, which is coupled with the tagged particle.
But, on the other hand, pathological situations as described in \cite{golse,CG1,CG2} are not stable: because of the dynamics of the scatterers, we expect the situation to be better  since 
some ergodicity could be retrieved from the additional degrees of freedom. In particular, there are invariant measures for the whole system, i.e. the system  consisting in  both the background and the tagged particle.

\noindent Here we shall take advantage of {the latter property} to establish global uniform a priori bounds for the distribution of particles, and more generally for all marginals of the $N$-particle distribution (see Proposition \ref{apriori-est}). This will be the key to control the collision process, and to prove (like in  Kac's model~\cite{kac} for instance) that dynamics for which a very large number of collisions occur over a short time interval, are of vanishing probability.

\noindent 
 Note that a similar strategy, based on the existence of the invariant measure, was   already used by 
 van Beijeren, Lanford, Lebowitz and Spohn
  \cite {beijeren, LS2} to derive  the linear Boltzmann equation for long times.

%%%%%%%%% fin des modifications

\bigskip
\noindent
 Let us now give the precise framework of our study. As     explained above, the idea is to improve Lanford's result    by considering fluctuations around some global equilibrium. Locally the~$N$-particle   distribution $f_N$ should therefore look like a conditioned  tensorized Maxwellian.

\bigskip

\noindent In   the sequel, we shall focus on the case of   hard-sphere dynamics (with mass $m=1$) to avoid technicalities due to artificial boundaries and cluster estimates. We shall further restrict our attention to the case when the domain is periodic $\D = \T^d = {[0,1]^d}$ ($d\geq 2$).

\noindent The microscopic model is therefore given by the following system of ODEs: 
\begin{equation}
\label{hard-spheres1}
{dx_i\over dt} =  v_i \,,\quad {dv_i\over dt} =0 \quad \hbox{ as long as \ } |x_i(t)-x_j(t)|>\eps 
\quad \hbox{for \ } 1 \leq i \neq j \leq N
\, ,
\end{equation}
with specular reflection after a collision %for two indexes~$i$ and~$j$ satisfying~$1 \leq i \neq j \leq N$, we have
\begin{equation}
\label{hard-spheres2}
\begin{aligned}
\left. \begin{aligned}
& v_i(t^+) =  v_i(t^-) - \frac1{\eps^2} (v_i-v_j)\cdot (x_i-x_j) (x_i-x_j) (t^-) \\
& v_j(t^+) =  v_j(t^-) + \frac1{\eps^2} (v_i-v_j)\cdot (x_i-x_j) (x_i-x_j) (t^-)\end{aligned}\right\} 
\quad  \hbox{ if } |x_i(t)-x_j(t)|=\eps\,.
\end{aligned}
\end{equation}

\noindent In the following we denote, for~$1 \leq i \leq N$, $z_i:=(x_i,v_i)$ and~$Z_N:= (z_1,\dots,z_N)$.   With a slight abuse we say that~$Z_N $ belongs to~$ \T^{dN} \times \R^{dN}$ if~$ X_N:=(x_1,\dots,x_N)$ belongs to~$ \T^{dN}$ and~$V_N:=(v_1,\dots,v_N)$ to~$ \R^{dN}$.
Recall that the phase space is denoted by
$$
{\mathcal D}_\eps^N:= \big\{Z_N \in \T^{dN} \times \R^{dN} \, / \,  \forall i \neq j \, ,\quad |x_i - x_j| > \eps \big\}.
$$
We now distinguish   pre-collisional configurations  from   post-collisional  ones by defining 
 for indexes~$1 \leq i \neq j \leq N$
$$
\begin{aligned}
\d {\mathcal D}_{\eps}^{N\pm }(i,j) := \Big \{Z_N \in \T^{dN} \times \R^{dN} \, / \,  &|x_i-x_j| = \eps \, , \quad \pm (v_i-v_j) \cdot (x_i- x_j) >0 \\
& \mbox{and}  \quad\forall (k,\ell) \in\big\{ [1,N] \setminus \{i,j\}\big\}^2 \, ,  |x_k-x_\ell| > \eps\Big\} \, .
\end{aligned}
$$
Given~$Z_N$ on~$\d {\mathcal D}_{\eps}^{N+}(i,j)$,  we define~$Z_N^*\in \d {\mathcal D}_{\eps}^{N-}(i,j)$ as the configuration having the same positions $(x_k)_{1\leq k\leq N}$, the same velocities $(v_k)_{k\neq i,j}$ for non interacting particles, and the following pre-collisional velocities for particles $i$ and $j$
$$
\begin{aligned}
 v_i^*& := v_i - \frac1{\eps^2} (v_i-v_j)\cdot (x_i-x_j) (x_i-x_j)   \\
 v_j^*& := v_j + \frac1{\eps^2} (v_i-v_j)\cdot (x_i-x_j) (x_i-x_j) \, .
\end{aligned}
$$

	\medskip
	
	\noindent 
Defining the Hamiltonian  
$$
H_N (V_N):= \frac12  \sum_{i=1}^N |v_i|^2  \, ,
$$
we consider the Liouville equation in the $2Nd$-dimensional phase space~${\mathcal D}_\eps^N$ 
	\begin{equation}
	\label{Liouville}
	\d_t f_N +\{ H_N, f_N\} =0
	\end{equation}
	with specular reflection on the boundary, meaning that if~$Z_N$ belongs to~$\d {\mathcal D}_{\eps}^{N+}(i,j)$  then
\begin{equation}\label{tracecondition}
	 f_N (t,Z_N ) =  f_N (t,Z_N^*) \, .
\end{equation}  We recall, as shown in~\cite{alexanderthesis} for instance, that
the set of initial configurations leading to   ill-defined characteristics (due to clustering of collision times, or collisions involving more than two particles) is of measure zero in~${\mathcal D}_\eps^N$.

	\medskip
	\noindent 
Define the Maxwellian distribution by  
\begin{equation}
\label{eq: Maxwellian}
	M_\beta^{\otimes s}(V_s):=\prod_{i=1}^sM_\beta (v_i) 
	\quad \mbox{and} \quad 
		M_\beta (v) := \left( \frac\beta{2\pi}\right)^{\frac d2}\exp \left( -\frac\beta2 |v|^2\right)\,.
\end{equation}	
	An  obvious remark is that~$M_\beta$ is a stationary solution  of~(\ref{boltzmann}), and any function of the energy~$f_N \equiv F(H_N)$ is    a stationary solution of the Liouville equation~(\ref{Liouville}). 
In particular, for~$\beta>0$, the Gibbs measure with distribution in~${\mathbf T}^{dN} \times \R^{dN}$ defined by
%	are obtained by taking $F_\beta(x) = \exp (-\beta x)$: we define, for~$\beta>0$ given,
		\begin{equation}
	\label{Gibbs}
	M_{N,\beta} (Z_N) : =\frac1{\cZ_{N }}  \left(\frac\beta{2\pi}\right)^\frac{dN}2\exp (-\beta H_N(V_N)) \, \indc_{ {\mathcal D}_\eps^N  }(Z_N)
	= \frac1{\cZ_{N }}  \, \indc_{ {\mathcal D}_\eps^N  } (Z_N)
 M_\beta^{\otimes N}(V_N)  
\end{equation}
where the partition function $\cZ_{N }$ is the normalization factor
		\begin{equation}
	\label{Gibbsnormalization}
	\begin{aligned}
	\cZ_{N } & := \int_{\T^{dN} \times \R^{dN} } \indc_{ {\mathcal D}_\eps^N  } (Z_N)
 M_\beta^{\otimes N}(V_N)  
 \, dZ_N 
	=  \int_{\T^{dN}  }  \prod_{1 \leq i \neq j \leq N}\indc_{ |x_i-x_j| > \eps }  \, dX_N\,,
	\end{aligned}
	\end{equation}
is an invariant measure for the gas dynamics.  	
	
\medskip

	\noindent In order to obtain the convergence for long times, a natural idea is to ``weakly" perturb the   equilibrium state~$ M_{N,\beta}$, by modifying the distribution of  one particle.
	In other words, we shall  describe the dynamics of a tagged particle in a background initially  at equilibrium. 
Actually this is the reason for placing the study in a bounded domain, in order for~$M_{N,\beta}$ to be  integrable in the whole phase space. Moreover we have restricted our attention to the case of a torus in order to avoid pathologies related to boundary effects, and complicated free dynamics. 
	
\medskip

	\noindent 	The strategy of perturbating $ M_{N,\beta}$  is classical in probability theory; following this strategy
	\begin{itemize}
	\item 
	we lose asymptotically the nonlinear coupling: we thus expect to get a linear equation for the distribution of the tagged particle;
	
	\item we also lose the feedback of the tagged particles on the background: since this background is constituted of $N\gg 1$ indistinguishable particles, the  momentum and energy exchange with the tagged particle has  a very small effect on each one of these indistinguishable particles and thus does not modify on average the background distribution. As a consequence, the limiting equation  for the distribution of the tagged particle should be non conservative.
	\end{itemize}
	
		\noindent What we shall actually prove  is that the limiting dynamics is governed by the linear Boltzmann equation~(\ref{linear-boltzmann}) with hard-sphere cross-section.

%\begin{Rmk}
%	In order to retrieve the whole effect of the initial perturbation (including the feedback on the background), we should look at the density of a typical particle, which is a fluctuation of order $1/N$ around the equilibrium $M_\beta$
%	$$
% f_N^0 (Z_N) :=  M_{N,\beta} (Z_N) \big( 1 + \frac1N \sum_{i=1}^N\rho^0(z_i) \big) \, , \quad \int \rho^0(z_i) M_\beta (v_i)\, dz_i = 0 \, .
% 	$$
%	 In this case, we expect the dynamics to be well approximated by the linearized Boltzmann equation
%\begin{equation}
%\label{eq: linearise}
%\partial_t \varphi +v\cdot \nabla_x \varphi  = -\alpha  L (\varphi)
%\end{equation}
%with
%$$
%L (\varphi) :=\iint [\varphi(v) + \varphi (v_1) -\varphi(v^*)-\varphi (v^*_1) ] M_\beta (v_1)  \,b(v-v_1,\omega) \, dv_1 d\omega \, .
%$$
%However it is much more difficult to obtain the equation for the fluctuation because there is no easy counterpart of the uniform a priori bounds coming from the maximum principle \cite{BGSRprepa}.
%	\end{Rmk}

\subsection{Main results}

\noindent For the sake of simplicity, we consider only one tagged particle which will be labeled by 1 with coordinates $z_1 = (x_1,v_1)$. 
The initial data is a perturbation of the equilibrium density \eqref{Gibbs} only with respect to the position 
$x_1$ of the tagged particle.
Consider~$\rho^0$ a continuous density of probability   on $\mathbf T^d$ and define
\begin{equation}
\label{initial}
 f_N^0 (Z_N) :=  M_{N,\beta} (Z_N) \rho^0(x_1)  	 \, .	 
\end{equation}
Note that the distribution $f_N^0$ is normalized by 1 in~$L^1({\mathbf T}^{dN} \times {\mathbf R}^{dN} )$ thanks to the translation invariance of ${\mathbf T}^d$ and that $\displaystyle \int_{{\mathbf T}^d} \rho^0(x) dx = 1$.

\medskip
\noindent The main result of our study is the following statement.
\begin{Thm}\label{long-time}
Consider the initial distribution
$f_N^0$ defined in \eqref{initial}.
%\begin{equation}\label{initialdataparticletoheat}
%\begin{aligned}
% f_N^0 (Z_N) :=  {   {\mathcal Z}_{N}^{-1} }  { {\mathbf 1}_{ {\mathcal D}_\varepsilon^N } (X_N) } \varphi^0( x_1) 
% M_\beta^{\otimes N}(V_N)    \, , \\
% \mbox{with} \quad   {\mathcal Z}_{N} := \int  { {\mathbf 1}_{ {\mathcal D}_ \varepsilon ^N } (X_N) } \varphi^0( x_1) M_\beta^{\otimes N}(V_N)    \, dZ_N \, , 
%\end{aligned}
%\end{equation}
Then the distribution~$f_N^{(1)}(t, x,v)$ of the tagged particle 
is close to $M_\beta(v) \varphi_\alpha  ( t , x,v) $, where $\varphi_\alpha  ( t , x,v) $ is the solution of the linear Boltzmann equation~{\rm(\ref{linear-boltzmann})} with initial data~$\rho^0 (x_1)$ and hard-sphere cross section.
More precisely,  for all~$t>0$ and all~$\alpha > 1$,  in the limit~$N \to \infty$,~$N\varepsilon^{d-1} \alpha^{-1}= 1$, one has
\begin{equation}
\label{eq: approx temps gd}
\big\|  f_N^{(1)}(t,  x,v)- M_\beta(v) \varphi _\alpha ( t , x,v) 
\big\|_{L^\infty( {\mathbf  T}^d\times  {\mathbf R}^d)} \leq C   
\left[ \frac{ t \alpha}{ ({\log\log N})^{\frac{A-1}{A} } } \right]^{\frac{A^2}{A-1}} \, ,
\end{equation}
where $A\geq 2$ can be taken arbitrarily large, and~$C$ depends on~$A,\beta,d$ and~$\|\rho^0\|_{L^\infty
}$. 
\end{Thm}
\noindent In  \cite{beijeren, LS2}, the linear Boltzmann equation was derived for any time $t>0$ (independent of~$N$).
In comparison, our approach leads to quantitative estimates on the convergence up to times diverging when $N \to \infty$. As we shall see, this is the key to derive the diffusive limit in Theorem \ref{brownien}.
Theorem~\ref{long-time}  proves that the linear Boltzmann equation is a good asymptotics of the hard-sphere dynamics, even for large concentrations $\alpha$ and long times $t$.  It further provides a rather good estimate on the approximation error.
Up to a suitable rescaling of  time, we can therefore obtain diffusive limits. 

\medskip

\noindent  In the macroscopic limit, the   trajectory of the tagged particle is defined by
\begin{equation}
\label{eq: tagged}
\brown (\tau) := x_1 \big( \alpha \tau \big) \in \T^d . 
\end{equation}
The distribution of $\brown (\tau)$ is given by $ f_N^{(1)}(\alpha \tau , x,v)$.
In the following, $\tau$ represents the macroscopic time scale.

\begin{Thm}
\label{brownien}
Consider~$N$ hard spheres on the space~${\mathbf T}^d \times {\mathbf R}^d$, initially distributed 
according to  $f_N^0$ defined in \eqref{initial}. Assume that~$\rho^0 $ belongs to~$C^0(\T^d)$.
Then the distribution~$f_N^{(1)}(\alpha \tau , x,v)$ remains close for the~$L^\infty$-norm to~$\rho( \tau, x) M_\beta(v) $ where~$\rho( \tau, x)$ is the solution of the linear heat equation
\begin{equation}\label{heat}
\partial_\tau \rho - \kappa_\beta \Delta_x \rho = 0 \quad \mbox{in} \quad   {\mathbf T}^d \, , \quad \rho_{|\tau = 0} = \rho^0 \, ,
\end{equation}
and the diffusion coefficient $\kappa_\beta$ is given by
\begin{equation*}
%\label{kappa-def0}
\kappa_\beta: = \frac{1}{d}  \int_{\R^d}  v \mathcal{L}^{-1} v \; M_\beta(v) dv ,
%= \frac{1}{d} \int_{\R^d}  \gamma (|v|) |v|^2 \, M_\beta(v) dv \, ,
\end{equation*}
where $ \mathcal{L}$ is the linear Boltzmann operator \eqref{linear-boltzmann} and $\cL^{-1}$ is its pseudo-inverse defined on~$(\Ker \cL)^\perp$  (see also {\rm(\ref{kappa-def})}). 
More precisely,
\begin{equation}\label{limitfN1toheat}
\big \|  f_N^{(1)} (\alpha \tau, x , v) -  \rho (\tau, x)  M_\beta (v) \big\|_{ L^\infty (  [0,T]\times   {\mathbf T}^d \times {\mathbf R}^d)} 
\to 0
\end{equation}
in the limit~$N\to \infty$,   with $\alpha=N \varepsilon^{d-1}$  going to infinity much slower than~$\sqrt{\log \log N}$.  

\noindent In the same asymptotic regime,  the process
$
\displaystyle \Xi (\tau) = x_1 (\alpha \tau )
$  associated with the tagged particle
converges in law towards a Brownian motion of variance~$\kappa_\beta$, initially distributed under the measure~$\rho^0$.

\end{Thm}

\noindent The Boltzmann-Grad scaling $\alpha=N \varepsilon^{d-1}$ is chosen  such that the mean free path is of order~$1/\alpha$, i.e. that a particle has on average $\alpha$ collisions per unit time.  This explains why 
in~\eqref{eq: tagged}, the position of the particle is not rescaled. Indeed over a time scale $\alpha \tau$ a particle will encounter $\alpha^2 \tau$ collisions which   is the correct balance for a diffusive limit.
In other words, one can think of $\alpha$ as a parameter tuning the density of the background particles.
{The positions and velocities  are not rescaled with $\alpha$ and are always at the macroscopic scale.}

\subsection{Generalizations}

For the sake of clarity, Theorem \ref{brownien} has been stated in the simplest framework. We   mention below several extensions which can be deduced in a straightforward way from the  proof of Theorem \ref{brownien}.

\medskip

\noindent
{\it Several tagged particles :}

\noindent The dynamics of a finite number of tagged particles can be followed and one can show that asymptotically, they converge to independent Brownian motions. 
This gives an answer to a conjecture raised by Lebowitz and Spohn \cite{LS1} on the diffusion of colored particles in a fluid.

\medskip

\noindent
{\it Interaction potential :}

\noindent Following the arguments in \cite{GSRT,PSS}, the behavior of a tagged particle in a gas with an interaction potential  can also be treated. %, under the assumptions of Theorem~\ref{lanford-thm}.

\medskip

\noindent
{\it Initial data :}

\noindent The perturbation on the initial particle could depend on $z_1 = (x_1, v_1)$ instead of depending only on the position $x_1$. The comparison argument to the linear Boltzmann equation is identical, but the derivation of the  diffusive behavior in Section 
\ref{subsec: Convergence to the heat equation} should be modified to show the relaxation of the velocity to a Maxwellian at the initial stage
(see Remark~\ref{rem-truc}).
\smallskip

\noindent By considering an initial data of the form 
\begin{equation}
\label{dirac-approx}
\rho_\alpha^0(x_1) = \alpha^{d \zeta} \; \rho^0  \big( \alpha^\zeta x_1\big  ) 
\qquad \hbox{ with } \quad  \zeta \ll 1 
\end{equation}
the tagged particle localizes when $\alpha$ goes to infinity.
The analysis can be extended to this class of initial data and leads, in the macroscopic limit, 
to a Brownian motion starting initially from a Dirac mass.

\medskip

\noindent
{\it Scalings :}

\noindent We have chosen here to work with macroscopic variables $(x,v)$, i.e. to rescale the particle concentration of the background and to dilate the time with a factor $\alpha$. However, the diffusive limit can be obtained by many other equivalent scalings involving the space  variable. In particular, one could have considered a domain $[0,\lambda]^d$ with a size $\lambda$ growing and a Boltzmann-Grad scaling $(N/\lambda^d) \eps^{d-1} =1$. Rescaling   space by a factor $\lambda$ and   time by $\lambda^2 \ll \log \log N$ would have led to the same diffusive limit.
In fact, one only needs the Knudsen number   to be small and of the same order as the Strouhal number~\cite{BGL1,LSR}.

\subsection{Structure of the paper}   

\noindent 
 Theorem \ref{brownien} is a consequence of  Theorem~\ref{long-time}, as explained in Section~\ref{proofbrownien}.  The core of our study is therefore the proof of Theorem~\ref{long-time}, which  relies on a comparison of the particle 
system to a limit system known as Boltzmann hierarchy. This hierarchy is obtained formally in Section~\ref{formal} from the hierarchy of equations satisfied by the marginals of~$f_N$, known as the BBGKY hierarchy (which is   introduced in Section~\ref{formal}). Section~\ref{controlbranch} is devoted to the control of the branching process that can be associated with the hierarchies, and in particular with the elimination of super-exponential trees;  the specificity of the linear framework is crucial  in this step, as it makes it possible  to compare the solution with the invariant measure globally in time. The actual proof of the convergence of the 
BBGKY hierarchy towards the Boltzmann hierarchy, on times diverging with~$N$, can be found in Section~\ref{endproofsection}.  

\noindent Some more technical estimates are postponed to Appendix A and B.

 %%%%%%%%%%%%%%%%%%%%%%%%%%%%%%%%%%%%%%%%%%%%%%%%
%%%%%%%%%%%%%%%%%%%%%%%%%%%%%%%%%%%%%%%%%%%%%%%%
\section{Formal derivation of the low density limit}
\label{formal}

Our starting point to study the low density limit is the Liouville equation (\ref{Liouville}) and its projection on the first marginal
 		$$ 
	f_N^{(1)} (t,z_1): = \int f_N(t,Z_N) dz_{2}\dots dz_N\,.
	$$
	Since it does not satisfy a closed equation, we have to consider the whole BBGKY hierarchy (see Paragraph 3.1).
The main difference with the usual strategy to prove convergence is that the symmetry is partially broken due to the fact that one particle is  {distinguished} from the others. In other words~$	f_{N|t=0}$
	is symmetric with respect to $z_2,\dots z_N$ but not to~$z_1$, and this property is preserved by the dynamics.
	
	\noindent 
More precisely we shall see that the specific form of the initial data (see Paragraph 3.2) implies that asymptotically we  have the following closure 
	$$ 
	f_N^{(2)} (t,z_1, z_2): = \int f_N(t,Z_N) dz_{3}\dots dz_N \sim  f_N^{(1)} (t,z_1) M_\beta (v_2)
	\sim \varphi_\alpha (t,z_1)    M_\beta (v_1) M_\beta (v_2)  
	$$
where~$\varphi_\alpha$ satisfies the linear Boltzmann equation~(\ref{linear-boltzmann}) with initial data~$\rho^0$.  Thus the limiting hierarchy reduces to the linear Boltzmann equation (see Paragraph 3.3).

\subsection{The series expansion} 
	 	The quantities we shall consider are   the marginals 
			$$ 
	f_N^{(s)} (t,Z_s): = \int f_N(t,Z_N) dz_{s+1}\dots dz_N
	$$
	so  $f_N^{(1)}$ is exactly the distribution of the tagged particle, and $f_N^{(s)}$ is the correlation between this tagged particle and $(s-1)$ particles of the background.

\noindent A formal computation based on Green's formula  leads to the following  BBGKY hierarchy for~$s<N$
	\begin{equation}
\label{eq: BBGKY}
	(\d_t +\sum_{i=1}^s v_i\cdot \nabla_{x_i} ) f_N^{(s)} (t,Z_s) = \alpha \big( C_{s,s+1} f_N^{(s+1)}\big) (t,Z_s)
	\end{equation}
	on~$\cD_\eps^s$, 	with the boundary condition as in~(\ref{tracecondition})
	$$f_N^{(s)} (t,Z_s) = f_N^{(s)} (t,Z_s^*) \hbox{ on }\d D_\eps^{s +}(i,j)  \, .$$

\noindent 	The collision term is defined by
		\begin{equation}
	\label{BBGKYcollision}
		\begin{aligned}
	&\big( C_{s,s+1} f_N^{(s+1)}\big) (Z_s)	  :=   (N-s) \eps^{d-1} \alpha^{-1}\\
&\times \Big( \sum_{i=1}^s \int_{{\mathbf S}^{d-1} \times \R^ d}  f_N^{(s+1)}(\dots, x_i, v_i^*,\dots , x_i+\eps \nu, v^*_{s+1}) \Big((v_{s+1}-v_i) \cdot \nu\Big)_+ d\nu dv_{s+1}  \\
& \quad    - \sum_{i=1}^s \int_{{\mathbf S}^{d-1} \times \R^ d}    f_N^{(s+1)}(\dots, x_i, v_i,\dots , x_i+\eps \nu, v_{s+1}) \Big((v_{s+1}-v_i) \cdot \nu\Big)_- d\nu dv_{s+1} \Big)
\end{aligned}
	\end{equation}
	where~${\mathbf S}^{d-1} $ denotes the unit sphere in~$\R^d$.
	  Note that  the collision integral is split into two terms according to the sign of $(v_i-v_{s+1}) \cdot \nu $ and we used the trace condition on $\d {\mathcal D}_\eps^N$
		to  express all quantities in terms of pre-collisional configurations.
		
		\medskip

	\noindent The closure  for $s=N$ is given by the Liouville equation (\ref{Liouville}). 
\noindent Note that the classical symmetry arguments used to establish the BBGKY hierarchy, i.e. the evolution equations for the marginals $f_N^{(s)}(t,Z_s)$, only involve the particles we add by collisions to the sub-system~$Z_s$ under consideration. In particular, the equation in the BBGKY hierarchy will not be modified at all since - by convention - the tagged particle is labeled by~1 and always belongs to the sub-system under consideration.

	\medskip
	
	\noindent	Given the special role played by the initial data (which is the reference to determine the notion of pre-collisional and post-collisional configurations), it is then natural to express solutions of the BBGKY hierarchy  in terms of a series of operators applied to the initial marginals. The  starting point in Lanford's proof is therefore  the  iterated Duhamel formula
	\begin{equation}\label{Duhamel}
\begin{aligned}
 f^{(s)} _N(t) =\sum_{n=0}^{N-s}  \alpha^n  \int_0^t \int_0^{t_1}\dots  \int_0^{t_{n-1}}  {\bf S}_s(t-t_1) C_{s,s+1}  {\bf S}_{s+1}(t_1-t_2) C_{s+1,s+2}   \\
\dots  {\bf S}_{s+n}(t_n)     f^{(s+n)}_N(0) \: dt_{n} \dots dt_1 \, ,
\end{aligned}
\end{equation}
where  ${\bf S}_s$ denotes the group associated to free transport in $\cD_\eps^s$ with specular reflection on the boundary.

\noindent To simplify notations, we define the operators $Q_{s,s} (t) = {\bf S}_s (t)$
and for $n \geq 1$
\begin{equation}
\label{Q-def}
	\begin{aligned}
Q_{s,s+n} (t) :=    \int_0^t \int_0^{t_1}\dots  \int_0^{t_{n-1}}  {\bf S}_s(t-t_1) C_{s,s+1}  {\bf S}_{s+1}(t_1-t_2)  
  C_{s+1,s+2}\dots   {\bf S}_{s+n}(t_n)   \: dt_{n} \dots dt_1
	\end{aligned}
\end{equation}
so that
\begin{equation}
\label{iterated-Duhamel}
    f^{(s)} _N(t) =  \sum_{n=0}^{N-s} \alpha^n Q_{s,s+n} (t) f^{(s+n)}_N(0) \, .
\end{equation}

\begin{Rmk}\label{bbgky-def} 
\label{bbgky-def}
It is not  obvious that formula {\rm(\ref{iterated-Duhamel})} makes sense since
the transport operator~${\bf S}_{s+1}$  is defined only for almost all initial configurations, and 
the collision operator~$C_{s,s+1}$ is defined by some integrals on manifolds of codimension 1. This fact is analyzed in~\cite{simonella} and in the erratum of~\cite{GSRT}. 
In the following, we will rely on the estimates on the collision operator derived in~\cite{GSRT}. 
%The collision integral   is defined according to the following procedure:
%\begin{itemize}
%
%\item  We first note that $(Z_s, \nu, v_{s+1},t )$ is a system of coordinates on the neighborhood of any regular point of  the boundary $\d  \cD_\eps^{s+1,\pm} (i,s+1)$.
%Since the transpor~${\bf S}_{s+1}$ preserves the $L^\infty$ norm, for any function $\varphi_{s+1} \in L^\infty(  \cD_\eps^{s+1})$ with bounded support,
%up to some truncation $\chi_\delta^{i,\pm}$ avoiding pathological trajectories
%$$ \chi_\delta^{i,\pm} {\bf S}_{s+1} (t) \varphi_{s+1} \in L^\infty( \cD_\eps^s \times {\mathbf S}^{d-1} \times \R^d \times [0,T]).$$
%
%\item  Fubini's theorem and  Alexander's result~\cite{alexander} then show that, up to some truncation, each elementary term of the collision operator is well-defined (by partial integration with respect to the only variables $(\nu, v_{s+1})$)~: for any function $\varphi_{s+1} \in L^\infty( \cD_\eps^{s+1})$ with bounded support,
%$$C_{s,s+1}^{\delta, i,\pm} {\bf S}_{s+1} (t) \varphi_{s+1} \in L^\infty( \cD_\eps^s \times [0,T]) \, .$$
%
%
%\item We can finally prove that the contribution of pathological trajectories to the collision integral is small outside from a small measure subset of $\cD_\eps^s \times [0,T]$. This means that we can remove the truncation and get that for any function $\varphi_{s+1} \in L^\infty(  \cD_\eps^{s+1})$ with bounded support,
%$$C_{s,s+1}^{ i,\pm} {\bf S}_{s+1} (t) \varphi_{s+1} \in L^\infty( \cD_\eps^s \times [0,T]).$$
%
%\end{itemize}
%
%The assumption on the support is then removed by considering weighted $L^\infty$ norms.
\end{Rmk}

\subsection{Asymptotic factorization of the initial data}\label{sectioninitialdata}$ $ 
The effect of the exclusion in the equilibrium measure vanishes when $\eps$ goes to 0 and the particles become asymptotically independent in the following sense.
	\begin{Prop}\label{exclusion-prop}
	 Given $\beta>0$,  there is a constant~$C>0$ such that for any fixed $s\geq 1$, the marginal of order $s$
	\begin{equation}\label{defmarginal}
	M_{N,\beta}^{(s)}  (Z_s) :=\int 	M_{N,\beta} (Z_N) \, dz_{s+1} \dots dz_N
\end{equation}
	satisfies,  as $N\to \infty $ in the scaling $N \e^{d-1} \equiv \alpha \ll 1/ \eps$, 
\begin{equation}
\label{lambdadsmbetaN}
\Big| \left(   M_{N,\beta}^{(s)}  -  M_\beta^{\otimes s} \right) \indc_{{\mathcal D}_{\eps}^s}  \Big|
\leq C  ^s  \, \eps \alpha \,  M_\beta^{\otimes s} 	
\end{equation}
where  the Maxwellian distribution $M_\beta^{\otimes s}$ was introduced in \eqref{eq: Maxwellian}.
\end{Prop}
\noindent	The proof of Proposition \ref{exclusion-prop}, by now classical, is recalled in Appendix A for the sake of completeness.

\noindent As a consequence of Proposition \ref{exclusion-prop}, the initial data is asymptotically close to a product measure: the following result is a direct corollary of Proposition \ref{exclusion-prop}.
\begin{Prop}\label{exclusion-prop2}
For the initial data~$f^0_{N } $ given  in~{\rm(\ref{initial})}, define the marginal of order~$s$  $$
f_{N }^{0(s)}  (Z_s) :=\int 	f^0_{N } (Z_N) \, dz_{s+1} \dots dz_N
= \rho^0 (x_1) M_{N,\beta}^{(s)}  (Z_s) \, .
$$
There is a constant~$C>0$ such that  as $N\to \infty $ in the scaling $N \e^{d-1}= \alpha \ll 1 /\eps $
	$$
			\Big| \left(f_{N }^{0(s)}  -g^{0(s)}\right) \indc_{{\mathcal D}_{\eps}^s}
	\Big|  \leq C ^s \eps \alpha    M_\beta ^{\otimes s}  \|\rho^0\|_{L^\infty}  \,,
			$$
where   $g^{0(s)}$ is defined by
\begin{equation}\label{defmarginalsboltzmann}
g^{0(s)}  (Z_s):= \rho^0 (x_1)  M_\beta^{\otimes s} 	
(V_s)\,.
\end{equation}
\end{Prop}

\subsection{The limiting hierarchy and the linear Boltzmann equation}

To obtain the Boltzmann hierarchy we start with the expansion~(\ref{iterated-Duhamel}) and compute the formal  limit of the collision operator~$ Q_{s,s+n}$ when~$\eps$ goes to~$ 0$. Recalling that~$(N-s) \eps^{d-1} \alpha^{-1} \sim 1$, it is  given by
$$
	\begin{aligned}
Q^0_{s,s+n} (t) :=   \int_0^t \int_0^{t_1}\dots  \int_0^{t_{n-1}}  {\bf S}^0_s(t-t_1) C^{0}_{s,s+1}  {\bf S}^0_{s+1}(t_1-t_2)  
  C^{0}_{s+1,s+2}\dots   {\bf S}^0_{s+n}(t_n)   \: dt_{n} \dots dt_1
	\end{aligned}
$$
where~$ {\bf S}^0_s$ denotes the free flow of~$s$ particles on~${\mathbf T}^{ds} \times \R^{ds}$, and $C^{0}_{s,s+1}$ are   the limit collision operators defined by
\begin{equation}
	\label{Boltzmanncollisionoperators}
	\begin{aligned}
	\big( C_{s,s+1}^{0} g ^{(s+1)}\big) (Z_s)      &:=   \sum_{i=1}^s \int  g^{(s+1)}(\dots, x_i, v_i^*,\dots , x_i , v^*_{s+1}) \Big((v_{s+1}-v_i) \cdot \nu\Big)_+ d\nu dv_{s+1} \\
&-  \sum_{i=1}^s \int  g^{(s+1)}(\dots, x_i, v_i,\dots , x_i , v_{s+1}) \Big((v_{s+1}-v_i) \cdot \nu\Big)_- d\nu dv_{s+1} \, .
	\end{aligned}
	\end{equation}
	Then the iterated Duhamel formula for the Boltzmann hierarchy takes the form
\begin{equation}
\label{iterated-Duhamelboltz}
 \forall s \geq 1 \, ,   \quad  g_\alpha ^{(s)} (t) =  \sum_{n\geq0}\alpha^n Q^0_{s,s+n} (t) g^{0(s+n)} \, .
\end{equation}

\begin{Rmk}\label{boltz-def}
%In the Boltzmann hierarchy, the collision operators are defined by integrals on manifolds of codimension $d$, so   the arguments presented in Remark {\rm\ref{bbgky-def}} cannot be applied anymore.
%We shall therefore require that the functions $\big(g_\alpha ^{(s)} \big)_{s \geq 1}$ are continuous, which is possible since free transport preserves continuity on~$\T^d \times \R^d$. 
In the Boltzmann hierarchy, the collision operators are defined by integrals on manifolds of codimension $d$, so we shall require that the functions $\big(g_\alpha ^{(s)} \big)_{s \geq 1}$ are continuous, which is possible since free transport preserves continuity on~$\T^d \times \R^d$.
\end{Rmk}

\noindent Consider the initial data~(\ref{defmarginalsboltzmann}). Then the family $(g_\alpha^{(s)} )_{s\geq 1}$ defined  by 
\begin{equation}\label{solutionBoltzmannlinear}
g_\alpha ^{(s)} (t,Z_s) :=\varphi_\alpha (t,z_1)    M_\beta^{\otimes s} (V_s)  
\end{equation}
is a solution to the Boltzmann hierarchy with initial data~$g^{0(s)}  $ since~$\varphi_\alpha$ satisfies the linear Boltzmann equation~(\ref{linear-boltzmann}) with initial data~$\rho^0$.  

\noindent
We insist   that the $g_\alpha^{(s)} $ are not defined as the marginals of some $N$-particle density.

\begin{Rmk}\label{boltz-uniq}
Note  that the estimates established in the next section imply actually that~$(g_\alpha^{(s)} )_{s\geq 1}$ is the unique  solution to the Boltzmann hierarchy (see~\cite{GSRT}).

\noindent
Furthermore the maximum principle for the linear Boltzmann equation leads to the following estimate
$$\sup _{t \geq 0}  \varphi_\alpha  (t,z_1) \leq \| \rho^0\|_{L^\infty}    \,.$$
\end{Rmk}

\noindent
In the following for the sake of simplicity we write~$g_\alpha:=g_\alpha ^{(1)}.$

%%%%%%%%%%%%%%%%%%%%%%%%%%%%%%%%%%%%%%%%%%%%%%%%%%%%%%%%%%%%%%%%%%%%%%%%%%%%%%%%%%%%%%%%%%
%%%%%%%%%%%%%%%%%%%%%%%%%%%%%%%%%%%%%%%%%%%%%%%%%%%%%%%%%%%%%%%%%%%%%%%%%%%%%%%%%%%%%%%%%%

\section{Control of the branching process}\label{controlbranch}
The restriction on the time of validity~$T^*/\alpha$ of  Lanford's convergence proof   (determined by a weighted norm of the initial data) is based on the elimination of ``pathological" collision trees, defined by a too large number of branches created in the time interval $[0,T^*/\alpha]$ (typically greater than~$n_\eps =O(|\log \eps|)$, see \cite{GSRT} for a quantitative estimate of the truncation parameter).
Here the global bound coming from the maximum principle  will enable us to iterate this truncation process on any time  interval.

\subsection{A priori estimates coming from the maximum principle}
\label{sec:  A priori estimates}\ \\
	For initial data as~(\ref{initial}), uniform a priori bounds can be obtained using only  the maximum principle for the Liouville equation (\ref{Liouville}).

\begin{Prop}\label{apriori-est} 
	For any fixed $N$, denote by $f_N$ the solution to the Liouville equation~{\rm(\ref{Liouville})} with initial data~{\rm(\ref{initial})}, and by $f_N^{(s)}$ its marginal of order $s$
	\begin{equation}
	\label{s-marginal}
	f_N^{(s)} (t,Z_s) := \int f_N(t,Z_N) \, dz_{s+1}\dots dz_N \, .
	\end{equation}
Then, for any $s\geq 1$,  the following  bounds hold uniformly with respect to time 
	\begin{equation}
	\label{upper-bound}
	   \sup_{t }  f_N^{(s)}(t,Z_s) \leq     M_{N,\beta}^{(s)} (Z_s)  \|\rho^0\|_{L^\infty} \leq    C^s M_\beta^{\otimes s} (V_s)  \|\rho^0\|_{L^\infty}\, ,
	   \end{equation}
for some $C>0$, provided that $\alpha \eps\ll1$.
\end{Prop}
	
	 {\noindent 	Note here that although the variable $z_1$ does not play at all a symmetric role with respect to~$z_2,\dots z_N$,  the upper bound~(\ref{upper-bound}) does not see this asymmetry.}
		
	\begin{proof}
One has immediately from~(\ref{initial}) that 
 	$$
	   f_N^0(Z_N) =M_{N,\beta}(Z_N)  \rho^0(x_1)  \leq    M_{N,\beta}(Z_N)   \|\rho^0\|_{L^\infty}\,.
	$$
Since the maximum principle holds for the Liouville equation (\ref{Liouville}), and as the Gibbs measure~$M_{N,\beta}$ is a stationary solution, we get for all~$t \geq 0$
$$ 
f_N(t,Z_N) \leq    M_{N,\beta} (Z_N)   \|\rho^0\|_{L^\infty}\,.
$$

\noindent	The   inequalities for the marginals follow  by integration and Proposition \ref{exclusion-prop}.	\end{proof}

\subsection{Continuity estimates for the collision operators}$ $
  To get uniform estimates with respect to $N$, the usual strategy is to use some  Cauchy-Kowalewski argument.
  In the following we shall  denote by~$|Q|_{s,s+n} $ the operator obtained by summing the absolute values of all elementary contributions
  $$	\begin{aligned}
|Q|_{s,s+n} (t) :=    \int_0^t \int_0^{t_1}\dots  \int_0^{t_{n-1}}  {\bf S}_s(t-t_1)\, \,  |C_{s,s+1} | \, \, {\bf S}_{s+1}(t_1-t_2)  
 \, \,  |C_{s+1,s+2}|\, \dots   {\bf S}_{s+n}(t_n)   \: dt_{n} \dots dt_1
	\end{aligned}
	$$
	and similarly for~$|Q^0|_{s,s+n}$
  $$
	\begin{aligned}
|Q^0|_{s,s+n} (t) :=   \int_0^t \int_0^{t_1}\dots  \int_0^{t_{n-1}}  {\bf S}^0_s(t-t_1)\, \,   |C^{0}_{s,s+1}|\, \,  {\bf S}^0_{s+1}(t_1-t_2)  
\, \,  |C_{s+1,s+2}^{0}|\, \dots   {\bf S}^0_{s+n}(t_n)   \: dt_{n} \dots dt_1
	\end{aligned}
$$
where
$$\begin{aligned}
	&\big(|C_{s,s+1}| f_N^{(s+1)}\big) (Z_s)\\
	& :=   (N-s) \eps^{d-1} \alpha^{-1}\sum_{i=1}^s \int_{{\mathbf S}^{d-1} \times \R^ d}  f_N^{(s+1)}(\dots, x_i, v_i^*,\dots , x_i+\eps \nu, v^*_{s+1}) \Big((v_{s+1}-v_i) \cdot \nu\Big)_+ d\nu dv_{s+1}  \\
& \   +(N-s) \eps^{d-1} \alpha^{-1} \sum_{i=1}^s \int_{{\mathbf S}^{d-1} \times \R^ d}    f_N^{(s+1)}(\dots, x_i, v_i,\dots , x_i+\eps \nu, v_{s+1}) \Big((v_{s+1}-v_i) \cdot \nu\Big)_- d\nu dv_{s+1}
\end{aligned}
$$
and
$$	\begin{aligned}
	\big(|C^{0}_{s,s+1 }  | g^{(s+1) }\big)(Z_s)      &:=   \sum_{i=1}^s \int  g^{(s+1)}(\dots, x_i, v_i^*,\dots , x_i , v^*_{s+1}) \Big((v_{s+1}-v_i) \cdot \nu\Big)_+ d\nu dv_{s+1}  \\
&\quad +  \sum_{i=1}^s \int  g^{(s+1)}(\dots, x_i, v_i,\dots , x_i , v_{s+1}) \Big((v_{s+1}-v_i) \cdot \nu\Big)_- d\nu dv_{s+1} \, .
	\end{aligned}
	$$

\noindent  For $\lambda>0$ and~$k\in \N^*$, we define~$X_{\e,k,\lambda}$
the space of measurable functions~$f_k$ defined almost everywhere on~$ \cD_\eps^k$ such that
$$
 	\|  f_{k}	\| _{\e,k,\lambda}:= {\rm{supess}}_{Z_k \in  \cD_\eps^k}
	\Big| f_k(Z_k) \; \exp \big( \lambda H_k(Z_k) \big) \Big| < \infty\, ,
$$	
and similarly~$X_{0,k,\lambda}$ is
the space of continuous functions~$g_k$  defined on~${\mathbf T}^{dk} \times \R^{dk}$ such that
$$
 	\|g_k\|_{0, k,\lambda} := \sup_{Z_k \in {\mathbf T}^{dk} \times \R^{dk}}\Big| g_k(Z_k) \exp \, \big(\lambda H_k(Z_k)\big) \Big|< \infty\, .$$	
	
	\begin{Lem}\label{continuity} 
	There is a constant~$C_d$ depending only on~$d$ such that for all $s,n\in \N^*$ and all~$t\geq 0$, the operators $|Q|_{s,s+n} (t)$ and $|Q^0|_{s,s+n} (t)$  satisfy the following continuity estimates: for all~$f_{s+n} $ in~$ X_{\e,s+n,\lambda}$, $|Q|_{s,s+n}(t)f_{s+n}$ belongs to~$ X_{\e,s,\frac\lambda2}$ and
	\begin{equation}
	\label{estimatelemmacontinuity} 
	\Big\| |Q|_{s,s+n}(t)f_{s+n} \Big\|_{\e,s,\frac\lambda2} \leq    e^{s-1} \left( {C_d t\over \lambda^{\frac {d+1}2}}  \right) ^n 	\| f_{s+n}	\| _{\e,s+n,\lambda} \, .
	\end{equation}
	Similarly for all~$g_{s+n} $ in~$ X_{0,s+n,\lambda}$, $|Q^0|_{s,s+n}(t)g_{s+n}$ belongs to~$ X_{0,s,\frac\lambda2}$ and
	\begin{equation}
	\label{Q0-est}
	\Big\| |Q^0|_{s,s+n}(t)g_{s+n} \Big\|_{0,s,\frac\lambda2} \leq   e^{s-1} \left( {C_d t\over \lambda^{\frac {d+1}2}}  \right) ^n 	\| g_{s+n}	\| _{0,s+n,\lambda} \,.
	\end{equation}
\end{Lem}

\begin{proof} Estimate~(\ref{estimatelemmacontinuity}) is simply obtained from the fact that the transport operators preserve the weighted norms, along with the continuity of the elementary collision operators.
From the erratum of \cite{GSRT}, we get the following statements
\begin{itemize}
\item  the transport operators satisfy the identities
$$\begin{aligned}
	\| {\mathbf  S}_k (t) f_k 	\| _{\eps,k,\lambda} & = 	\| f_k	\| _{\eps,k,\lambda} \\
	\| {\mathbf  S}_k^0 (t) g_k 	\| _{0,k,\lambda} & = 	\| g_k	\| _{0,k,\lambda} \, .
	\end{aligned}
	$$
\item  the collision operators satisfy the following bounds in the Boltzmann-Grad scaling~$N \e^{d-1} \equiv \alpha $ 
$$
 \Big| {\mathbf S}_k (-t) \, |{C}_{k,k+1} | \, {\mathbf S}_{k+1} (t)   f_{k+1} (Z_k)\Big| 
  \leq   {C_{d}} \, {\lambda^{-\frac d2} } \Big( k\lambda^{-\frac12} + \sum_{1 \leq i \leq k} |v_i|\Big) \exp\left(- \lambda H_k (Z_k)\right)\|  f_{k+1}	\| _{\e,k+1,\lambda} 
$$
almost everywhere on $\R_t \times \cD_\eps^k$, for some $C_{d} > 0$ depending only on $d$, and
\begin{equation}
\label{est-col:1bisHS}
  \big| |{C}^0_{k,k+1} |\; g_{k+1} (Z_k)\big|    \leq    {C_{d}} \, {\lambda^{-\frac d2} } \Big( k\lambda^{-\frac12} + \sum_{1 \leq i \leq k} |v_i|\Big)  \exp\left(- \lambda H_k (Z_k)\right)	\|  g_{k+1}	\| _{0,k+1,\lambda}\, ,
\end{equation}
on $\T^{dk} \times \R^{dk}$.

\end{itemize}
The result then  follows from piling together those inequalities (distributing the  exponential weight evenly on each occurence of a collision term). We  notice that by the Cauchy-Schwarz inequality,
$$
\begin{aligned}
\sum_{1 \leq i \leq k} |v_i| \exp\Big(- \frac \lambda{4n} \sum_{1 \leq j \leq k} |v_j|^2\Big) &\leq \left( k\frac {2n} \lambda\right) ^{\frac12} \left( \sum_{1 \leq i \leq k} \frac\lambda {2n} |v_i|^2 \exp\Big(- \frac \lambda{2n} \sum_{1 \leq j \leq k} |v_j|^2\Big)\right)^{1/2} \\
&\leq \Big( \frac{2nk}{e\lambda}\Big)^{1/2} \leq \sqrt{ \frac{2}{e\lambda}} (s+n)  \, ,
\end{aligned}
$$
with $k \leq s+n$ in the last inequality. 
Each collision operator gives therefore a loss of $ C \lambda ^{-(d+1)/2}  (s+n) $ together with a loss on the exponential weight, while 
the integration with respect to time provides  a factor $t^n/n!$. By Stirling's formula, we have
$${ (s+n)^n\over n!} \leq  \exp \left(  n \log {n+s \over n }  + n\right) \leq \exp ( s+n) \,.$$
That proves the first statement in the lemma.
The same arguments give the counterpart for the Boltzmann collision operator.
\end{proof}

\subsection{Collision trees of controlled size} 
For general initial data (in particular, for chaotic initial data),
the proof of Lanford's convergence result then relies on two steps:
\begin{itemize}
\item[(i)] a short time bound for the series expansion~(\ref{iterated-Duhamel}) expressing the correlations of the system of $N$ particles  and a similar bound for  the corresponding quantities associated with the Boltzmann hierarchy;
\item[(ii)] the termwise convergence of each term of the series.
 \end{itemize}
However after a short time  (depending on the initial data), the question of the convergence of
the series~(\ref{iterated-Duhamel}) is still open. 
One of the difficulties to prove this convergence is to take into account the cancellations between the gain and loss terms of the collision operators. These cancellations are neglected in Lanford's strategy.

\smallskip

 	\noindent Here  we assume that the BBGKY initial data takes the form~(\ref{initial})
and the Boltzmann initial data takes the form~(\ref{defmarginalsboltzmann}), and we shall take advantage of the control by   stationary solutions (the existence of which is obviously related to these cancellations) given by Proposition~\ref{apriori-est} to obtain a lifespan which does not depend on the initial data.
Indeed, we have thanks to Propositions~\ref{exclusion-prop} and~\ref{apriori-est} provided that $\alpha \eps\ll1$
$$
\begin{aligned}
	\| f_N^{(k)}(t) 	\| _{\e,k,\b} & = {\rm{supess}}_{Z_k \in  \cD_\eps^k}
	\Big| f_N^{(k)}(t,Z_k) \; \exp \big( \b H_k(Z_k) \big) \Big| \\
	& \leq    \sup_{Z_k \in  \cD_\eps^k}\Big(
M_{N,\beta}^{(k)}(Z_k) \exp \big( \b H_k(Z_k) \big)\Big)   \|\rho^0\|_{L^\infty}\\
 & \leq  C^k \sup_{Z_k \in  \cD_\eps^k} \Big( M_\beta^{\otimes k} (V_k) \exp \big( \b H_k(Z_k) \big) \Big)  \|\rho^0\|_{L^\infty}\, .
\end{aligned}
$$
Thus for all~$t \in \R$,
\begin{equation}
\label{estimatefNktweight}
	\| f_N^{(k)}(t)	\| _{\e,k,\b} \leq   C^k \Big(\frac\beta{2\pi}\Big)^{kd/2}  \|\rho^0\|_{L^\infty}\, .
\end{equation}
Similarly for  the initial data for the Boltzmann hierarchy defined in~(\ref{defmarginalsboltzmann}),  by Remark \ref{boltz-uniq} the solution~(\ref{solutionBoltzmannlinear}) of the evolution is bounded by  
\begin{equation}
\label{estimategNktweight}
 	\| g_\alpha^{(k)}(t)	\| _{0,k,\b} \leq       \Big(\frac\beta{2\pi}\Big)^{kd/2}  \|\rho^0\|_{L^\infty} \, .
\end{equation}

\medskip
\noindent
 Moreover we shall use a truncated series expansion instead of (\ref{iterated-Duhamel}) and (\ref{iterated-Duhamelboltz}).
Let us fix a (small) parameter $h>0$  and a sequence~$\{n_k\}_{k \geq 1}$ of integers  to be tuned later.
We shall study the dynamics up to time  $t := K h$ for some large integer $K$, by splitting the time interval~$[0,t]$  into~$K$ intervals,  and controlling the number of collisions  on each interval. 
In order to discard  trajectories with a large number of collisions in the iterated Duhamel formula \eqref{iterated-Duhamel}, we define collision trees ``of controled size" by the condition that they have strictly less than $n_k$ branch points on the  interval~$[t-kh ,t-(k-1) h]$.
Note that by construction, the trees are actually  followed ``backwards", from time~$t$ (large) to time~$0$. 

\begin{figure}[h] %  figure placement: here, top, bottom, or page
   \centering
  \includegraphics[width=4.2in]{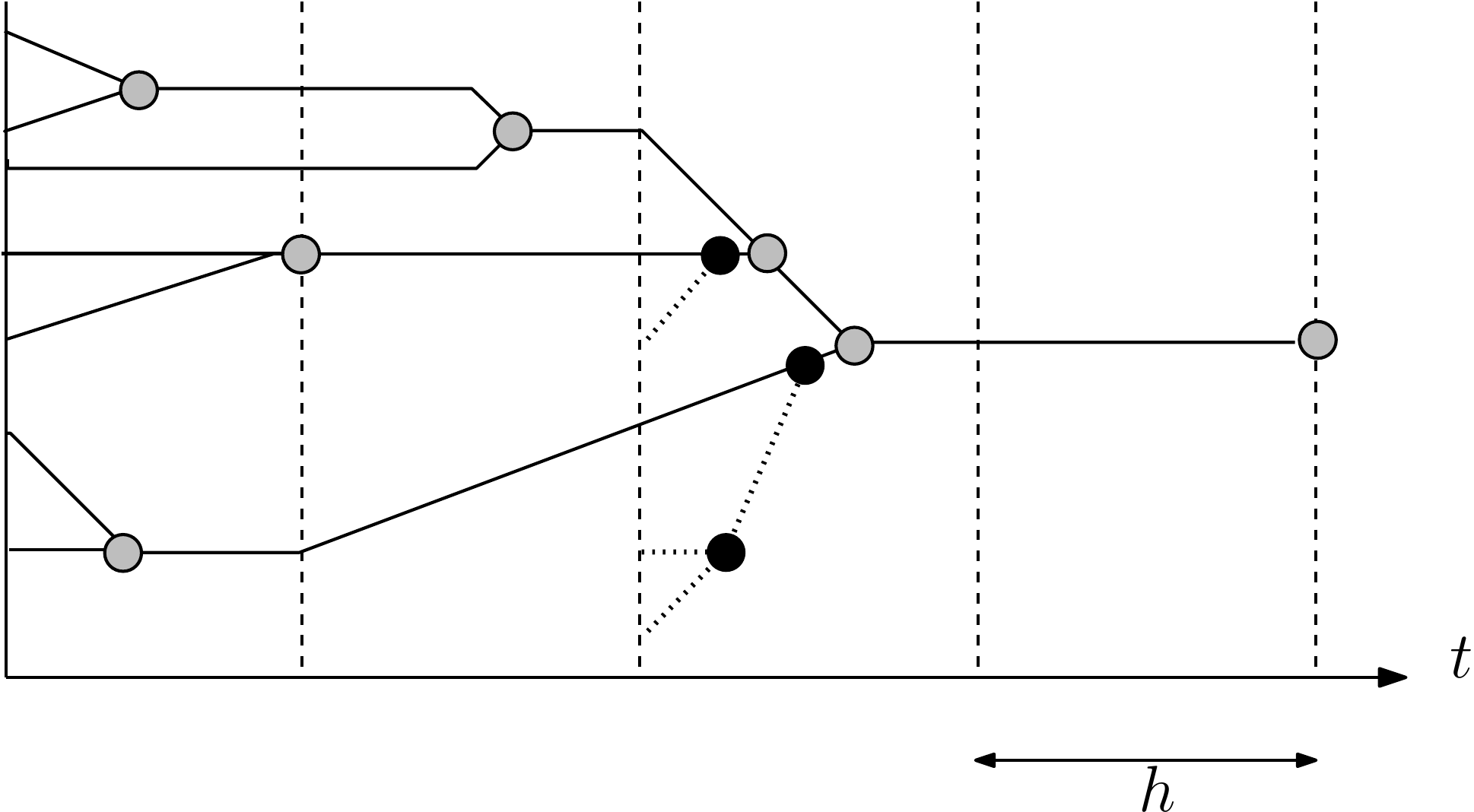} 
\caption{\label{schema} Suppose~$n_k= A^k$ with~$A=2$. 
Each collision is represented by a circle from which 2 trajectories emerge.
The tree including the three  extra collisions in dotted lines occurring during $[t-2 h,t-h]$ is not a good collision tree and in our procedure, it would be truncated at time $t-2 h$.
The tree without the dotted lines is a good collision tree with $t = 4 h$ :  the number of collisions during the $k^{\text{th}}$-time interval is less than $n_k - 1= A^k -1$. }
 \end{figure}

\noindent As we are interested only in the asymptotic behaviour of  the first marginal, we start by using~\eqref{Duhamel} with~$s=1$, during the time interval $[t -  h, t]$: iterating Duhamel's formula up to time~$t-h$ instead of time 0,  we have
\begin{equation}
\label{firststepfN1}
\begin{aligned}
f^{(1)}_N(t) =  \sum_{j_1=0}^{n_1 - 1}\alpha^{j_1-1}  Q_{1,1+j_1} (h ) f^{(j_1)}_N(t - h )
+ R_{1,n_1}(t - h,t) \, ,
\end{aligned}
\end{equation}
where~$R_{1,n_1}$ accounts for at least~$n_1$ collisions
\begin{align*}
R_{1,n_1}(t',t) :=   \sum_{p= n_1}^{N-1}\alpha^p Q_{1,p+1}(t-t')    f^{(p+1)}_N(t') \, . \nonumber
\end{align*}
More generally we define~$R_{k,n}$   as follows
\begin{align}
\label{eq: reste}
R_{k,n}(t',t) :=   \sum_{p= n }^{N-k}  \alpha^p Q_{k,k+p} (t-t')     f^{(k+p)}_N(t') \, . \nonumber
\end{align}
  The  term $R_{k,n}(t',t)$  accounts for trajectories originating at~$k$ points at time $t$, and involving at least~$n$ collisions during the time-span~$t-t'$. The idea is that
if $n$ is large then such a behaviour should be atypical and~$R_{k,n}(t ',t)$  should be negligible. 

\medskip
\noindent The first term on the right-hand side of~(\ref{firststepfN1}) can be broken up again by iterating the Duhamel formula on the time interval 
$[t -  2 h, t - h]$ and truncating the contributions with more than~$n_2$ collisions: this gives
\begin{align*}
f^{(1)}_N(t) 
 =&   \sum_{j_1=0}^{n_1-1}\sum_{j_2=0}^{n_2-1}\alpha^{j_1+j_2} Q_{1,1+j_1} (h )Q_{1+j_1,1+j_1 +j_2} (h ) \, f^{(1+j_1+ j_2)}_N(t-2h  ) \\
& + R_{1,n_1}(t-h,t) +  \sum_{j_1=0}^{n_1-1} \alpha^{j_1} Q_{1,j_1+1} (h ) R_{j_1+1,n_2} (t- 2h,t -h) \, .
\end{align*}
Iterating this procedure $K$ times and truncating the trajectories with at least~$n_k$ collisions during the time interval $[t-kh ,t-(k-1) h]$, leads to the following expansion
\begin{equation}
\label{eq: serie f1}
 f^{(1)}_N(t) 
 =  f^{(1,K)}_N(t)  + R_N^K(t) \, , 
 \end{equation}
where denoting~$J_0:=1$ and~$J_k :=1+ j_1 + \dots +j_k$,
\begin{equation}
\label{eq: serie f1developpee}
\begin{aligned}
  f^{(1,K)}_N(t)     : =   \sum_{j_1=0}^{n_1-1}\! \!   \dots \!  \! \sum_{j_K=0}^{n_K-1}\alpha^{J_K-1} Q_{1,J_1} (h )Q_{J_1,J_2} (h )
 \dots  Q_{J_{K-1},J_K} (h ) \, f^{0(J_K)}_N  
 \end{aligned}
\end{equation}
and
$$
  R_N^K(t)  := \sum_{k=1}^K \; \sum_{j_1=0}^{n_1-1} \! \! \dots \! \! \sum_{j_{k-1}=0}^{n_{k-1}-1} \; 
\alpha^{J_{k-1}-1} Q_{1,J_1} (h ) \dots  Q_{J_{k-2},J_{k-1}} (h ) \, R_{J_{k-1},n_k}(t-k h, t-(k-1)h )  \, . 
$$
By an appropriate choice of the sequence $\{n_k\}$, we are going to show that the main contribution to the density $f^{(1)}_N (t)$ is given by~$ f^{(1,K)}_N(t)$ and that~$R_N^K(t)  $ vanishes asymptotically.

 \bigskip
 \noindent
 Next as in~(\ref{eq: serie f1developpee}) we can write a truncated expansion for
 $g_\alpha$ (see \eqref{solutionBoltzmannlinear}) as follows:
\begin{equation}
\label{eq: serie f1boltz}
g_\alpha (t) 
 =  g_\alpha^{(1,K)} (t)  + R_\alpha ^{0,K}(t) \, , 
 \end{equation}
where with notation~(\ref{defmarginalsboltzmann}) and~(\ref{solutionBoltzmannlinear}),
\begin{equation}
\label{eq: serie f1developpeeboltz}
\begin{aligned}
  g _\alpha^{(1,K)} (t)    : =   \sum_{j_1=0}^{n_1-1} \! \! \dots\! \!  \sum_{j_K=0}^{n_K-1}\alpha ^{J_K-1} Q^0_{1,J_1} (h )Q^0_{J_1,J_2} (h )
 \dots  Q^0_{J_{K-1},J_K} (h ) \, g_\alpha^{0(J_K)}
 \end{aligned}
\end{equation}
and
$$
R_\alpha ^{0,K}(t) := \sum_{k=1}^K \; \sum_{j_1=0}^{n_1-1} \! \! \dots \! \! \sum_{j_{k-1}=0}^{n_{k-1}-1} \; 
\alpha^{ J_{k-1}-1} Q^0_{1,J_1} (h ) \dots  Q^0_{J_{k-2},J_{k-1}} (h ) \, R^0_{J_{k-1},n_k}(t-k h, t-(k-1)h )
$$
with
$$R^0_{k,n}(t',t) :=  \sum_{p\geq n } \alpha^p
 Q^0_{k,k+p}(t-t')     g_\alpha^{(k+p)}(t')  \, .
$$

\subsection{Estimates of the remainders}$ $
 \noindent Since we expect the particles to undergo on average one collision per unit of time, the growth of collision trees is typically exponential. Pathological trees are therefore those with super exponential growth. There are   two natural ways of defining such pathological trees
\begin{itemize}
\item either by choosing some fixed $h$ (given for instance by Lanford's proof) and $\log n_k \gg k$;
\item or by fixing $n_k = A^k$ and letting the elementary time interval $h \to 0$.
\end{itemize}
We shall choose the latter option.

\begin{Prop}
\label{prop: remainders}
Under the assumptions of Theorem~{\rm\ref{long-time}}, the following holds. Let~$A\geq 2$ be given and define~$ n_k: = A^k, $ for~$k \geq 1$. Then there exist  $c, C, \gamma_0>0$  depending   on  $d$, $A$ and~$\beta$ such that for any $t>1$  and any $\gamma\leq \gamma_0 $,
 choosing 
\begin{equation}
\label{eq: tau nk}
  h \leq  { c\gamma \over  \alpha^{A/(A-1)}t^{1/(A-1)}} 
  \quad \text{and} \quad
  K = t/h \ \text{integer}
\end{equation}
we get
\begin{equation}
\label{eq: controle erreur}
\left\|
  R_N^K(t)\right\|_{L^\infty({\mathbf T}^{d} \times \R^{d})} +  \left\|
 R_\alpha ^{0,K}(t)\right\|_{L^\infty({\mathbf T}^{d} \times \R^{d})}  \leq  C   \gamma^A    \|\rho^0\|_{L^\infty} \, .\end{equation}
\end{Prop}

\begin{proof}
We are going to bound 
$$
\big\| Q_{1,J_1} (h ) \dots  Q_{J_{k-2},J_{k-1}} (h ) \, R_{J_{k-1},n_k}(t-k h, t-(k-1)h ) \big \|_{L^\infty({\mathbf T}^{d} \times \R^{d})} 
$$ 
for each term in the remainder~$R_N^K$.
The exact distribution  of collisions in the last $k-1$ intervals is not needed and it is enough to estimate directly
$$
\big \| |Q|_{1,J_{k-1}} ((k-1)h ) \, R_{J_{k-1},n_k }(t-k h, t-(k-1)h )  \big\|_{L^\infty({\mathbf T}^{d} \times \R^{d})}  \, .
$$

\noindent Applying 
Lemma \ref{continuity}, one has  (denoting generically by~$C_d$ any constant depending only on~$d$)
\begin{align*}
& \big\| |Q|_{1,J_{k-1}} ((k-1)h ) \,  R_{J_{k-1},n_k}(t-k h, t-(k-1)h )  \big\|_{L^\infty({\mathbf T}^{d} \times \R^{d})} \\
& \qquad 
\leq  \left( {C_d  \; (k-1) h \over \beta^{(d+1)/2 }} \right)^{J_{k-1}-1}   \|  R_{J_{k-1},n_k}(t-k h, t-(k-1)h ) \|_{\e,J_{k-1},\b/2}
\,.
\end{align*}
Then arguing as in the proof of Lemma \ref{continuity}, one can write
$$
\begin{aligned}
& \alpha^{J_{k-1} - 1} \big \| |Q|_{1,J_{k-1}} ((k-1)h ) \,  R_{J_{k-1},n_k}(t-k h, t-(k-1)h ) \big \|_{L^\infty({\mathbf T}^{d} \times \R^{d})} \\
& \qquad\leq  \sum_{p=n_k}^{N-J_{k-1}}
 \left( {C_d  \alpha  (k-1) h \over  \beta ^{(d+1)/2} } \right)^{J_{k-1}-1}  
\left( {C_d \alpha  h \over \beta^{(d+1)/2} } \right)^{p}
 \sup_{t \geq 0} \| f_N^{(J_{k-1} +p)}(t) \|_{\e,J_{k-1} +p,\b} 
  \\
& \qquad 
\leq     \|\rho^0\|_{L^\infty} \beta^{\frac d2} (\alpha t)^{J_{k-1}-1} \sum_{p=n_k}^{N-J_{k-1}}
 \left( {  C_d  \over \sqrt {\beta} }  \right)^{J_{k-1} +p-1} 
  \;(\alpha  h)^{p} 
\, ,
\end{aligned}
$$
thanks to~(\ref{estimatefNktweight}) and recalling that~$(k-1)h \leq t$. Assuming from now on that
\begin{equation}\label{hsmall}
{  C_d  \alpha  h\over \sqrt {\beta} } <\frac12
\end{equation}
we find
\begin{equation}\label{eq: controle erreur bis}
\begin{aligned}
& \alpha^{J_{k-1} - 1} \big\| |Q|_{1,J_{k-1}} ((k-1)h ) \,  R_{J_{k-1},n_k}(t-k h, t-(k-1)h )  \big\|_{L^\infty({\mathbf T}^{d} \times \R^{d})}  \\
& \qquad \leq     \|\rho^0\|_{L^\infty} \beta^{\frac d2} (\alpha t)^{J_{k-1}-1}  \left( {  C_d  \over \sqrt {\beta} }  \right)^{J_{k-1} +n_k-1}  \;(\alpha  h)^{n_k}  \, .
\end{aligned}
\end{equation}

\noindent Note that  $ \mathcal{N}_{j}: = 1+n_1 + \dots + n_j = {A^{j+1} -1\over A-1} \leq {1\over A-1} n_{j+1}$. Then, since $J_{k-1} \leq  \mathcal{N}_{k-1}$,  one has, for some appropriate  constant~$C(d,\beta)  $, 
\begin{align*}
%\label{eq: erreur 2}
&\alpha^{J_{k-1} - 1}  \big\| |Q|_{1,J_{k-1}} ((k-1)h ) \,   R_{J_{k-1},n_k}(t-k h, t-(k-1)h ) \ \big\|_{L^\infty({\mathbf T}^{d} \times \R^{d})} \\
& \qquad 
\leq      \beta^{d/2} 
\exp \Big( A^k \Big(\log { C(d,\beta) +{1\over A-1} \log(\alpha t )+\log(\alpha  h)  \Big)\Big)}
  \|\rho^0\|_{L^\infty}\, . \nonumber
\end{align*}
Therefore, choosing  
$$h \leq  \frac\gamma { { C(d,\beta)  } \; \alpha ^{A/(A-1)} t^{1/(A-1)}}  \,,$$
 which is compatible with~(\ref{hsmall}) as soon as~$\gamma$ is small enough  one has
\begin{equation}
\label{eq: erreur 2}
\begin{aligned}
 \alpha^{J_{k-1} - 1}\big\| |Q|_{1,J_{k-1}} ((k-1)h ) \,  R_{J_{k-1},n_k}(t-k h, t-(k-1)h ) \big\|_{L^\infty({\mathbf T}^{d} \times \R^{d})} 
\\
  \leq  {    \beta^{d/2}}  \exp \Big(     A^k \log \gamma \Big)   \|\rho^0\|_{L^\infty}\, . 
\end{aligned}
\end{equation}
This implies 
\begin{align*}
& \left\|R_N^K\right\|_{L^\infty({\mathbf T}^{d} \times \R^{d})} \\
& \leq   {    \beta^{d/2}} \sum_{k=1}^K  \Big(  \prod_{i=1}^k n_i \Big) \; \exp \Big(   A^k \log \gamma  \Big)  \|\rho^0\|_{L^\infty}
\leq      \beta^{d/2} \sum_{k=1}^K \;  \exp \Big(  k(k+1) \log(A) + A^k \log \gamma \Big)  \|\rho^0\|_{L^\infty} \\
& \leq C_A   \beta^{d/2} \sum_{k=1}^K \;  \exp \Big(  Ak \log \gamma \Big)   \|\rho^0\|_{L^\infty}\leq C_A  \beta^{d/2} \gamma^A  \|\rho^0\|_{L^\infty} \end{align*}
for $\gamma$ sufficiently small, where $C_A$ is a constant depending on $A$.
Thus, we get the first part of \eqref{eq: controle erreur}
$$ \| R_N^K\|_{L^\infty( \T^d \times \R^d)} \leq C \gamma^A \| \rho^0\|_{L^\infty}.$$
The argument is identical in the case of the Boltzmann hierarchy: 
%estimate~(\ref{eq: controle erreur}) becomes
 $$
 \begin{aligned}
&\alpha^{J_{k-1} - 1} \left \| |Q|^0_{1,J_{k-1}} ((k-1)h ) \,  R^0_{J_{k-1},n_k}(t-k h, t-(k-1)h ) \right\|_{L^\infty({\mathbf T}^{d} \times \R^{d})} \\
& \qquad 
\leq  \left( {C_d   (k-1) \alpha h \over  \beta ^{(d+1)/2}}  \right)^{J_{k-1}-1}  
\left( {C_d \alpha   h \over \beta^{(d+1)/2} } \right)^{n_k}  \sup_{t \geq 0} \| g_\alpha^{(J_{k-1} +n _k)}(t) \|_{0,J_{k-1} +n _k,\b}   \\
& \qquad 
\leq     \beta^{\frac d2} \left( {  C_d  \over \sqrt{\beta} }  \right)^{J_{k-1} +n _k-1} 
\; (\alpha t)^{J_{k-1}-1} \; (\alpha h)^{n_k}  \|\rho^0\|_{L^\infty}
\, ,
\end{aligned}
$$
hence finally
$$   \left\| R_\alpha ^{0,K}\right\|_{L^\infty({\mathbf T}^{d} \times \R^{d})}  \leq C_A   \beta^{d/2} \gamma^A  \|\rho^0\|_{L^\infty} \, , 
$$
and the proposition is proved.
\end{proof}

%%%%%%%%%%%%%%%%%%%%%%%%%%%%%%%%%%%%%%%%%%%%%%%%
%%%%%%%%%%%%%%%%%%%%%%%%%%%%%%%%%%%%%%%%%%%%%%%%

\section{Proof of the convergence}
\label{endproofsection}

\noindent In this section, we conclude the proof of Theorem \ref{long-time}.
%prove that  in the scaling~$N \eps^{d-1} \alpha^{-1} \equiv 1$
%$$
%\| f_N^{(1)} - g_\alpha  \|_{L^\infty ([0,T/\alpha] \times {\mathbf T}^d\times \R^d) }\to 0 \, , \quad \mbox{as} \, \, N \to \infty \, .
%$$
Thanks to Proposition~\ref{prop: remainders}, we are reduced to studying $ f_N^{(1,K)} - g_\alpha^{(1,K)}$
(introduced in~(\ref{eq: serie f1}),~(\ref{eq: serie f1boltz})) and to proving that the matching terms in the series $f_N^{(1,K)}$ and $g_\alpha^{(1,K)}$ are close to each other. 

\medskip

Throughout this section, the parameters are chosen such that (with the notation of Proposition
\ref{prop: remainders}) 
\begin{equation}
\label{eq: parametre conditions}
N \eps^{d-1} = \alpha \ll \frac{1}{\eps} \, ,  \quad A \geq 2 \, ,  
\quad t>1 \, ,   \quad K = \frac{t}{h} \, \cdotp
\end{equation}

\noindent
Each elementary term in the series $f_N^{(1,K)}$ and $g_\alpha^{(1,K)}$ has a geometric interpretation as an integral over some  {\it pseudo-trajectories}.
As explained in~\cite{lanford,CIP,GSRT}, in this formulation the characteristics associated with the operators~$ {\bf S}_{i} (t_{i-1}-t_{i}) $ and~$ {\bf S}^0_{i} (t_{i-1}-t_{i}) $ are followed {backwards} in time between two consecutive times~$t_i$ and~$t_{i-1}$, and the collision terms (associated with~${C}_{i,i+1}$ and~${C}_{i,i+1}^0$) are seen as source terms in which  ``additional particles" are ``adjoined" to the system.  
The main heuristic idea is that the pseudo-trajectories associated to both hierarchies can be coupled precisely if no recollisions occur in the BBGKY hierarchy. The core of the proof will be to obtain an upper bound on the occurrence of recollisions and to show that their contribution is negligible.
 
\smallskip

\noindent In order to prevent recollisions in the time interval~$[t_{i+1},t_{i}]$, some bad sets in   phase space must be removed. Following the approach developed  in~\cite{GSRT}, a geometrical  control of the  trajectories in the torus (stated in Lemma~\ref{geometric-lem1}) enables us  to define bad sets, outside of which the flow ${\bf S}$ between two collision times is  the free flow ${\bf S}^0$ (see Proposition~\ref{geometric-prop}).
Finally, the geometric controls are used  in Section
\ref{subsec: Estimates} to obtain quantitative estimates on the collision integrals 
where those bad sets have been removed.

\subsection{Reformulation in terms of pseudo-trajectories} 
\label{subsec: pseudo-trajectories}

We consider one term of the sum~$f^{(1,K)}_N(t)$ in~\eqref{eq: serie f1developpee} and show how it can be interpreted in terms of pseudo-trajectories.
Given the indices $J = (j_1, \dots, j_K)$, we set 
\begin{align}
\label{eq: 1 terme de la serie}
 F^{(1,K)}_N (J) \; (t,z_1) & :=  Q_{1,J_1} (h )Q_{J_1,J_2} (h )
 \dots  Q_{J_{K-1},J_K} (h ) \, f^{0(J_K)}_N  \\
& =  \int_{{\mathcal T}_J (h)}  \, dT \;   {\bf S}_1(t-t_1) C_{1,2}  {\bf S}_{2}(t_1-t_2) C_{2,3}   
\dots  {\bf S}_{J_K}(t_{J_K-1})     f^{0(J_K)}_N
\nonumber
\end{align}
where the time integral is over the collision times  $T=(t_1, \dots, t_{J_{K-1}})$  taking values in 
\begin{equation}
\label{eq: collision times}
{\mathcal T}_J (h) :=   \Big\{ T = (t_1,\dots,t_{J_{K-1}}) \,  \Big| \,   t_i < t_{i-1} \,  \,  \mbox{and} 
\,  (t_{J_k}, \dots, t_{J_{k-1}+1}) \in [t-k h, t - (k-1)h]    \Big\} \, .
\end{equation}
\noindent In the following we denote by~${\bf \Psi}_s$ the~$s$-particle flow.
Given~$z_1 = (x_1,v_1) \in \T^d \times \R^d$ and a time~$u \in [t_1, t]$, we call 
$z_1(u) = {\bf \Psi}_1(u) z_1$ the coordinates following the backward flow ${\bf \Psi}_1$ of one particle. 
The first collision operator $C_{1,2}$ is interpreted as the adjunction at time $t_1$ of a new particle at $x_1(t_1) + \eps \nu_{2}$ for a deflection angle~$\nu_{2} \in {\mathbf S}^{d-1}$
and with a velocity~$v_{2} \in \R^d$. 
The new pair of particles $Z_{2}$ will be evolving according to the backward 2-particle flow 
${\bf \Psi}_2$ during the time interval $[t_2,t_1]$ starting at $t_1$ from
\begin{align}
\label{eq: 2 cas}
\begin{cases}
Z_{2}(t_1) = \big( (x_1(t_1), v_1), (x_1(t_1) + \eps \nu_{2}, v_2) \big)
\text{ in the pre-collisional case~$(v_2-v_1)\cdot \nu_2  <0$} \\
Z_{2}(t_1) = \big( (x_1(t_1), v_1^*), (x_1(t_1) + \eps \nu_{2}, v_2^* ) \big)  
\text{ in the post-collisional case~$(v_2-v_1)\cdot \nu_2  >0$  , }
\end{cases}
\end{align}
the latter case corresponding to the scattering.

\medskip
\noindent Iterating this procedure,
a branching process is built inductively by adding a particle labelled~$i+1$ at time $t_i$ to the particle $z_{m_i}(t_i)$ where $m_i \leq i$ is chosen randomly among the first~$i$ particles. 
Given a deflection angle $\nu_{i+1}$ and a velocity $v_{i+1}$, the velocity of the particles~$z_{m_i}$ and $z_{i+1}$ at time $t_i$ are updated according to the pre-collisional or post-collisional rule as in~\eqref{eq: 2 cas}
\begin{align*}
%\label{eq: 2 cas}
\begin{cases}
Z_{i+1}(t_i) = \big( \{ z_j (t_i) \}_{j \not = m_i}, (x_{m_i}(t_i), v_{m_i}(t_i)), (x_{m_i}(t_i) + \eps \nu_{{i+1}}, v_{i+1}) \big)\\
\hskip3cm \text{ in the pre-collisional case~$(v_{i+1} (t_i) -v_{m_i} )\cdot \nu_{i+1}  <0$} \\
Z_{i+1}(t_i) = \big(  \{ z_j (t_i) \}_{j \not = m_i}, (x_{m_i}(t_i), v_{m_i}^*(t_i)), (x_{m_i}(t_i) + \eps \nu_{i+1}, v_{i+1}^* ) \big)  \\
\hskip3cm 
\text{ in the post-collisional case~$(v_{i+1}(t_i) -v_{m_i} )\cdot \nu_{i+1}  >0$  .}
\end{cases}
\end{align*}
Let $Z_{i+1}$ denote the $i+1$ components after the $i^{\text{th}}$-collision.
The evolution of~$Z_{i+1}$ follows the flow of the backward transport ${\bf \Psi}_{i+1}$ during the time interval $[t_{i+1},t_i]$. From \cite{simonella} (see also Remark \ref{bbgky-def}), one can check that 
 ${\bf \Psi}_{i+1}$ is well defined up to a set of measure 0.
 In the following, we shall use the name {\it collision} to describe the creation of a particle and {\it recollision} if two  particles collide in the flow ${\bf \Psi}_{i+1}$.
\begin{figure}[h] %  figure placement: here, top, bottom, or page
\centering
\includegraphics[width=3in]{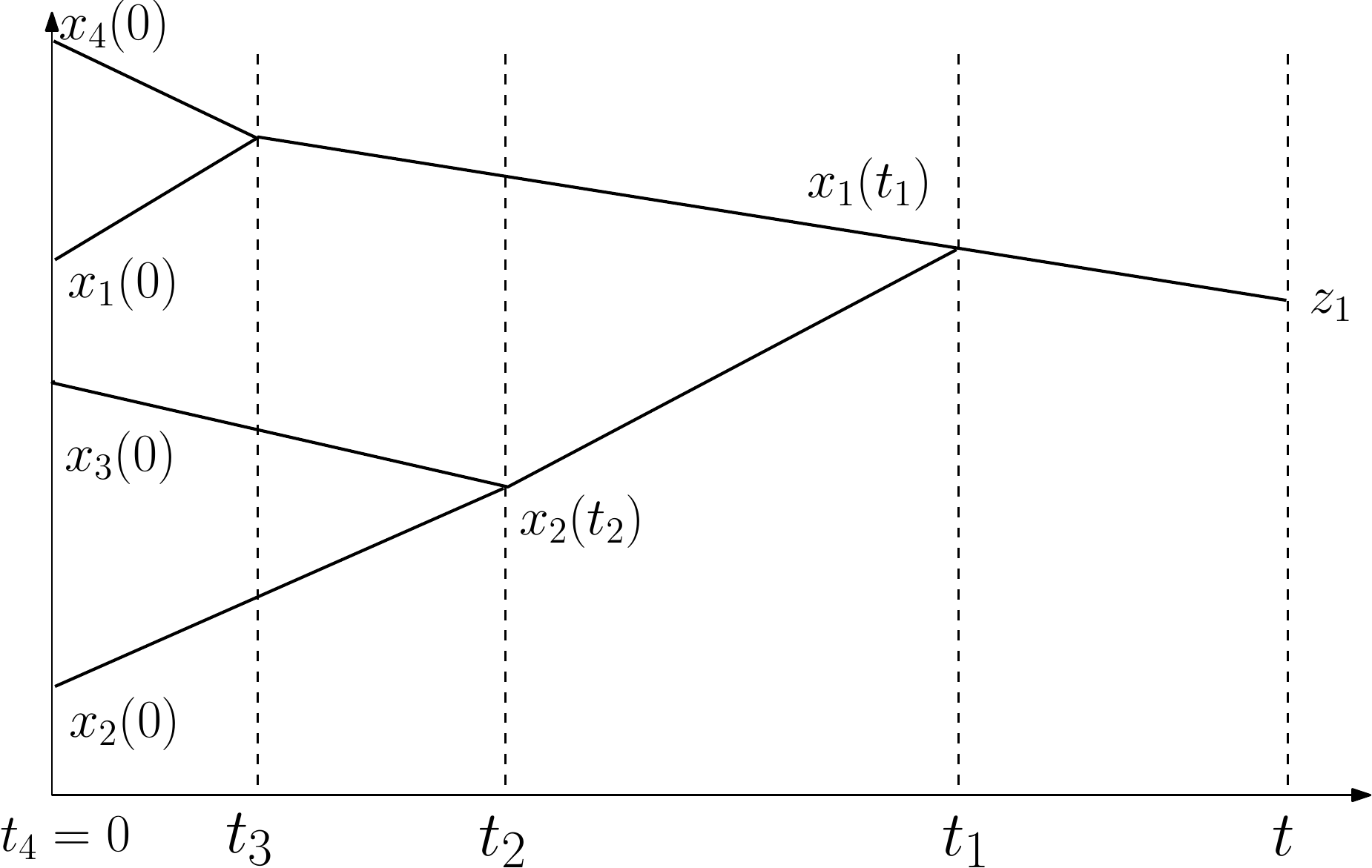} 
\caption{\small 
A collision tree is represented with 3 collisions. The velocities $(v_1,v_2)$ at time $t_1$ are pre-collisional and the first particle keeps its velocity $v_1$ after the collision.  At time $t_3$, the first particle is selected $m_3 =1$ and the velocity $v_1$ is modified into $v_1^*$ according to the post-collisional rule.
}
\label{fig: branchement}
 \end{figure}

\medskip
\noindent To summarize, pseudo-trajectories do not involve physical particles. They are a geometric interpretation of the iterated  Duhamel  formula in terms of a branching process 
flowing backward in time and determined by 
\begin{itemize}
\item 
the collision times~$T = (t_{1 },\dots,t_{J_K-1})$ which are interpreted as branching times
\item the labels of the collision particles~$m= (m_1, \dots , m_{J_K-1})$ from which branching occurs 
and which take values in the set
$$
\mathcal{M}_J := \big\{ m= (m_1, \dots , m_{J_K-1})  \, ,  \quad 1\leq m_i \leq i \big\}
$$
\item the coordinates of the initial particle $z_1$ at time $t$
\item the velocities~$v_2,\dots,v_{J_K}$ in $\R^d$ and deflection angles~$\nu_2,\dots,\nu_{J_K}$ in
${\mathbf S}_1^{d-1}$ for each additional particle.
\end{itemize}
The integral \eqref{eq: 1 terme de la serie} can be evaluated by integrating $f^{0(J_K)}_N$ on the value of the pseudo-trajectories $Z_{J_K} (0)$ at time 0 
\begin{equation*}
F^{(1,K)}_N (J)  =  \sum_{m \in \mathcal{M}_J} 
\left( \frac{\eps ^{d-1}}{\alpha} \right)^{J_K-1} \frac{(N-1)!}{ (N-J_K )!} \; F^{(1,K)}_N (J,m)
\end{equation*}
where
\begin{align}
\label{BBGKYfunctionalapprox 0}
F^{(1,K)}_N (J,m) \, (t,z_1) := 
      \int_{{\mathcal T}_J(h) }  dT
    \int_{({\mathbf S}^{d-1} \times \R^d)^{J_K-1}} \!  d \bar \nu \, d \bar V 
    \; \cA(T, z_1, \bar \nu, \bar V) \;    f_N^{0(J_K)}    (Z_{J_K} (0)) 
\end{align}
and with 
\begin{equation}
\label{eq: product velocities}
\cA(T, z_1,\bar \nu, \bar V) := \prod_{i=1}^{J_K - 1} ((v_{i+1} - v_{m_i} (t_i))\cdot \nu_{i+1} )
\quad \text{and} \quad
\begin{cases}
\bar \nu = \{ \nu_2, \dots, \nu_{J_K} \}\\
\bar V = \{ v_2, \dots, v_{J_K} \}
\end{cases}
.
\end{equation}
The definition of $\cA$ requires to compute the whole pseudo-trajectory on the time interval $[0,t]$ starting at $z_1$ in order to be able to sample the velocities at the different times $T = (t_1, \dots, t_{J_K -1})$.
Note that the contributions of the gain and loss terms in the collision operator $C_{k,k+1}$ are taken into account by the sign of $\big((v_{k+1} - v_{m_{k}}(t_{k}))\cdot \nu_{k+1} \big)$.

\medskip

\noindent In the same way, a branching process associated with the Boltzmann hierarchy can be constructed: given an initial particle $z_1^0 = (x_1^0,v_1^0)$ at time $t$, a collection of collision times~$T = (t_{1 },\dots,t_{J_K-1})$ and   labels of the collision   particles~$m= (m_1, \dots , m_{J_K-1}) \in {\mathcal M}_J$
as well as a collection of velocities~$v_2,\dots,v_{J_K}$ and deflection angles~$\nu_2,\dots,\nu_{J_K}$, 
the $(k+1)^{\text{th}}$ particle~$z^0_{k+1}$ is added at time $t_k$ at the position~$x_{m_k}^0(t_k)$
of  the  particle $m_k$ and their velocities are adjusted according to the type of the collision
$$
\begin{aligned}
\begin{cases}
z_{m_k}^0(t_k) =  \big( x_{m_k}^0(t_k), v_{m_k}(t_k) \big), \; 
z^0_{k+1}(t_k) = \big( x^0_{m_k}(t_k), v_{k+1} \big)
\quad \text{if~$(v_{k+1}  - v_{m_k} (t_k) )\cdot \nu_{k+1}  <0$} \\
z_{m_k}^0(t_k) =  \big( x_{m_k}^0(t_k^+), v_{m_k}^*(t_k^+)  \big), \; 
z^0_{k+1}(t_k) = \big( x^0_{m_k}(t_k^+), v_{k+1}^* \big)
\quad \text{if~$(v_{k+1}-v_{m_k} (t_k^+) )\cdot \nu_{k+1}  >0$ .}
\end{cases}
\end{aligned}
$$
Then, the corresponding pseudo-trajectory~$Z_{k+1}^0$ evolves according to the backward free flow denoted by ${\bf \Psi}^0_{k+1}$ during the time interval $[t_{k+1},t_k]$ until the next particle creation. As the particles are points, no recollision occurs in this branching process. 
Notice that~$u \mapsto~Z_{k+1}^0(u )$ is pointwise left-continuous on~$[0,t_k]$.
 
\noindent The counterpart of the integral \eqref{eq: 1 terme de la serie} in the series 
$g^{(1,K)}_\alpha(t)$ in \eqref{eq: serie f1developpeeboltz}
can be formally rewritten as follows
\begin{align}
\label{eq: 1 terme de la serie 0}
G^{(1,K)} (J) \; (t,z_1) & =  \int_{{\mathcal T}_J (h)}  \, dT \;   {\bf S}^0_1(t-t_1) C^0_{1,2}   {\bf S}^0_{2}(t_1-t_2) C^0_{2,3}   
\dots  {\bf S}^0_{J_K}(t_{J_K-1})     g^{0(J_K)}     \nonumber\\
&=  \sum_{m \in \mathcal{M}_J}   G^{(1,K)} (J,m)
\end{align}
where the integral is over the pseudo-trajectories
\begin{equation}
\label{botzfunctionalapprox 0}
\begin{aligned}
& G^{(1,K)} (J,m) \, (t,z_1) := 
     \int_{{\mathcal T}_J(h) }  dT     \int_{({\mathbf S}^{d-1} \times \R^d)^{J_K-1}} \!  d \bar \nu \, d \bar V 
    \; \hat \cA(T, z_1,\bar \nu, \bar V) \;  
%    \int_{{\mathbf S}^{d-1} \times \R^d} \!  d\nu_{2} dv_{2}
%  ((v_{2} - v_{m_1} (t_1))\cdot \nu_{2} ) \\
%&  \qquad       \dots    \int_{{\mathbf S}^{d-1} \times \R^d}   \!  d\nu_{J_K} dv_{J_K}
%  ((v_{J_K} - v_{m_{J_K - 1}}(t_{J_K -1}))\cdot \nu_{J_K} ) \;    
  g^{0(J_K)}    (Z^0_{J_K} (0)) ,
\end{aligned}
\end{equation}
with $\hat \cA$ defined as in \eqref{eq: product velocities} but with respect to the Boltzmann hierarchy pseudo-trajectories.

\medskip 
 
\begin{figure}[h] %  figure placement: here, top, bottom, or page
   \centering
  \includegraphics[width=4in]{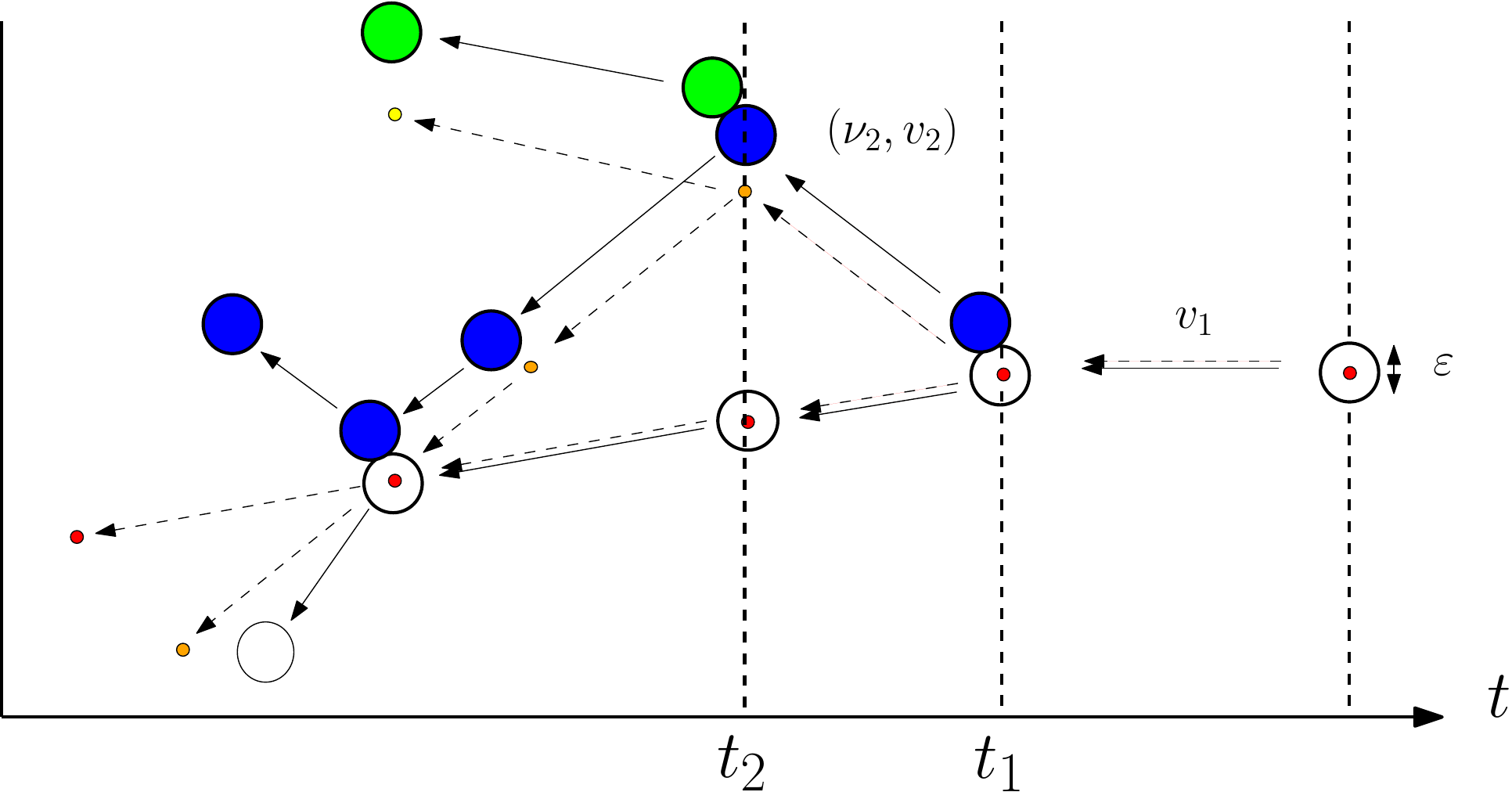} 
\caption{\small  The first stages of both pseudo-trajectories are depicted up to the occurence of a recollision. 
The BBGKY pseudo-trajectories are represented with plain arrows, whereas the Boltzmann
pseudo-trajectories correspond to the dashed arrows. 
At time $t$, the particle with label $1$ in the BBGKY hierarchy is  a ball of radius $\eps$ centered at position $x_1$ and the particle in the Boltzmann hierarchy is depicted as a point located at $x_1^0 = x_1$.
At time $t_1$ the second particle is added and at time $t_2$ the third.
Both hierarchies are coupled, but a small error in the particle positions of order $\eps$ can occur at each collision. 
In this figure, a recollision between the first and the second particle of the BBGKY pseudo-trajectories occurs and after this recollision the Boltzmann and the BBGKY pseudo-trajectories are no longer close to each other.
Indeed the BBGKY trajectories are deflected after the recollision, instead  the ideal particles do not collide and follow a straight line (see the dashed arrows).
Note that before the recollisions the trajectories of $z_1$ and $z_1^0$ are identical and therefore the plain and the dashed arrows overlap.
}
\label{fig: recollision}
 \end{figure}

\noindent To show that $F^{(1,K)}_N (J,m)$ and $G^{(1,K)} (J,m)$  are close to each other when $N$ diverges, we shall prove that the pseudo-trajectories $Z$ and $Z^0$ can be coupled in order to remain very close  to each other  up to a small error  (see Figure \ref{fig: recollision}) 
\begin{itemize}
\item due to the micro-translations $\eps \nu_{k+1}$ of the added particle at each collision time $t_k$
\item excluding the possible recollisions on the interval $]t_k, t_{k-1}[$ along the flow ${\bf S}_k$,  which do not occur for the free flow ${\bf S}^0_k$.
\end{itemize}
The proof of  the convergence   follows the arguments of \cite{GSRT}. 
This will be achieved by constructing in \eqref{eq: bad set}, a set of deflection angles and velocities $(\cB (J, T, m))^c \subset  \big( {\mathbf S}^{d-1} \times \R^d \big)^{J_k -1}$ such that the pseudo-trajectories $Z$ induced by this set have no recollisions and therefore remain very close to the pseudo-trajectories $Z^0$ associated to the free flow. 
Furthermore, the measure of $\cB (J, T, m)$ tends to 0 when $N$ goes to infinity. Finally, in Section \ref{subsec: Estimates}, all the estimates will be combined to derive a quantitative bound on $F^{(1,K)}_N (J,m) - G^{(1,K)} (J,m)$.

\subsection{Reduction to non-pathological trajectories}

\subsubsection{The elementary step}
\label{subsec: The elementary step}

The set of good configurations with $k$ particles will be such that the particles remain at a distance $\eps_0 \gg \eps$ for a time $t$, i.e. that they belong to the set  
$$
\cG_k(\eps_0) := \left\{ Z_k \in \T^{dk} \times \R^{dk}\; 
\Big| \; \forall u \in [0,t]  ,\quad \forall i\neq j ,\quad  d(x_i -u \, v_i,  x_j- u \, v_j) \geq \eps_0 \right\}
$$
where $d$ denotes the distance on the torus $\T^d$. 
For particles in $\cG_k(\eps_0)$, the transport ${\bf \Psi}_k$ coincides with the free flow. 
Fix $\bar a \ll \eps_0$.
Thus, if at time $t$ the configurations $Z_k$, $Z_k^0$ are  such that 
\begin{equation}
\label{eq: example}
\forall i \leq k \, , \qquad |x_i - x_i^0| \leq \bar a \, ,
\qquad %\text{and} \quad
v_i = v_i^0
\end{equation}
and that $Z_k^0$ belongs to $\cG_k(\eps_0)$, then the configurations ${\bf \Psi}_k (u) Z_k$, ${\bf \Psi}^0_k (u) Z_k^0$ will remain at distance less than $\bar a$ for $u \in [0,t]$.

\medskip

\noindent We are going to show that the good configurations are stable by adjunction of a 
$(k+1)^{\text{th}}$-particle next to the particle labelled by $m_k \leq k$. 
More precisely, let $Z_k^0 = (X_k^0, V_k)$ be in~$\cG_k(\eps_0)$ and  
$Z_k = (X_k, V_k)$ with positions close to $X_k^0$ and same velocities (cf. \eqref{eq: example}). Then,
by  choosing the velocity $v_{k+1}$ and the deflection angle 
$\nu_{k+1}$ of the new particle $k+1$ outside a bad set $\cB^{m_k}_k (Z_k^0)$, both configurations
$Z_k$ and $Z_k^0$ will remain close to each other.
Of course, immediately after the adjunction, the particles $m_k$ and $k+1$ will not be at distance~$\eps_0$, but~$v_{k+1}, \nu_{k+1}$ will be chosen such that the particles drift rapidly far apart and
after a short time $\delta>0$ the configurations $Z_{k+1}$ and $Z^0_{k+1}$
will be again in the good sets $\cG_{k+1} (\eps_0/2)$ and~$\cG_{k+1} (\eps_0)$.

\medskip

\noindent This stability result was obtained in \cite{GSRT} and is stated below.
We shall restrict to bounded velocities taking values in the ball $B_E := \big \{ v \in \R^d \, , \,  |v|  \leq E  \big\}$ for a given large parameter~$E>0$ to be tuned later on.

\begin{Prop}[\cite{GSRT}]
\label{geometric-prop}
We   fix   parameters~$ \bar a,\eps_0, \delta$     such that
\begin{equation}\label{sizeofparameters}
%\cancel{  \bar  a \ll \eps_0 \ll \delta E }
A^{K+1} \eps \ll \bar  a \ll \eps_0 \ll \min(\delta E,1)    \, .
\end{equation}
Given~$Z_k^0 = (X_k^0, V_k) \in \cG_k(\eps_0)$ and  $m_k \leq k$, there is a 
subset~$\cB^{m_k}_k (Z_k^0)$ 
of~${\mathbf S}^{d-1} \times B_E$  of small measure
\begin{equation}
\label{pathological-size}
\big |\cB^{m_k}_k ( Z_k^0) \big |  \leq  Ck\left( E^d  \left( {\bar a \over \eps_0}\right)^{d-1} +E^d  (E t  )^d  \eps_0  ^{d-1}+E\left({ \eps_0 \over \delta}\right)^{d-1}  \right)
\end{equation}
such that good configurations close to $Z_k^0 $  are stable by adjunction of a collisional particle close to the particle~$x^0_{m_k}$ %and not belonging to $\cB_k( Z_k^0)$, 
in the following sense. 

\smallskip
\noindent 	 
Let~$Z_k = (X_k, V_k)$ be a configuration of $k$ particles satisfying \eqref{eq: example}, i.e.
$| X_k- X_k^0  | \leq \bar  a$. 
Given~$(\nu_{k+1},v_{k+1}) \in ({\mathbf S}^{d-1} \times  B_E) \setminus  \cB^{m_k}_k( Z_k^0)$,
a new particle with velocity $v_{k+1}$ is added at $x_{m_k} + \eps \nu_{k+1}$ to $Z_k$ and at $x_{m_k}^0$ to $Z_k^0$.
Two possibilities may arise %depending on the velocity $v_{k+1}$

$ \bullet $ For a pre-collisional configuration~$\nu_{k+1} \cdot (v_{k+1} - v_{m_k})< 0$ then 
\begin{equation}
\label{taugeq0precoll}
\forall u \in ]0,t]  \, , \quad \left\{ \begin{aligned}
& \forall i \neq j \in [1,k] \, , \quad  d (x_i -u \, v_i, x_j - u \, v_j )  > \eps \, ,\\
& \forall j \in [1,k] \, , \quad  d (x_{m_k} + \eps \nu_{k+1} - u \,v_{k+1}, x_j - u\,  v_j )  > \eps \, .
\end{aligned}
\right.
\end{equation}
Moreover after the time~$\delta$, the~$k+1$ particles are in a good configuration
\begin{equation}
\label{taugeqdelta2precoll}
\forall u \in [\delta ,t]  \, , \quad \left\{ \begin{aligned}
& (X_k - u V_k   , V_k , x_{m_k} + \eps \nu_{k+1} - u \, v_{k+1} , v_{k+1}) \in \cG_{k+1} (\eps_0/2)\\
& (X^0_k - u V_k   , V_k ,  x^0_{m_k} - u \, v_{k+1} , v_{k+1}) \in \cG_{k+1} (\eps_0)\, .
\end{aligned}
\right.
\end{equation}

$ \bullet$ For a post-collisional configuration~$\nu_{k+1} \cdot (v_{k+1} - v_{m_k}) > 0$ then 
the velocities are updated 
\begin{equation}\label{taugeq0precoll}
\forall u \in ]0,t]  \, , \quad \left\{ 
\begin{aligned}
&\forall i \neq j \in [1,k] \setminus \{m_k\} \, , \quad  d (x_i - u \, v_i, x_j - u \, v_j )  > \eps \, ,\\
& \forall j \in [1,k] \setminus \{m_k\} \, , \quad d (x_{m_k} + \eps \nu_{k+1} - u \, v_{k+1}^*, x_j - u \, v_j )  > \eps \, ,\\
&\forall j \in [1,k] \setminus \{m_k\} \, , \quad d (x_{m_k}  - u \, v_{m_k}^*, x_j - u \,  v_j )  > \eps \, ,\\
&
d (x_{m_k}  - u \, v_{m_k}^*, x_{m_k} + \eps \nu_{k+1} - u \, v_{k+1}^*) > \eps 
\, .
\end{aligned}
\right.
\end{equation}
Moreover after the time~$\delta$, the~$k+1$ particles are in a good configuration
\begin{align}\label{taugeqdelta2precoll}
&\forall u\in [\delta, t], \nonumber \\
&\left\{ \begin{aligned}
& \big( \{ x_{j} - u \, v_{j}, v_j  \}_{j \not = m_k}, x_{m_k} - u \, v_{m_k}^* ,v_{m_k}^*,  x_{m_k} + \eps \nu_{k+1} - u \, v_{k+1}^* , v_{k+1}^* \big) \in \cG_{k+1} (\eps_0/2),\\
& \big( \{  x^0_{j} - u \, v_{j}, v_j  \}_{j \not = m_k},  x^0_{m_k}- u \, v_{m_k}^* ,v_{m_k}^*,  
 x^0_{m_k}  - u \, v_{k+1}^* , v_{k+1}^*\big) 
\in \cG_{k+1} (\eps_0)\, .
\end{aligned}
\right.
\end{align}
\end{Prop} 
\noindent Proposition \ref{geometric-prop} is the elementary step for adding a new particle. 
In Section \ref{traj-def}, we are going to show how this step can be iterated 
in order to build inductively good pseudo-trajectories~$Z$ and $Z^0$.
Note that after adding a new particle, the velocities remain identical at each time in both configurations, but their positions differ due the exclusion condition in the BBGKY hierarchy which induces a shift of $\eps$ at each creation of a new particle (see Figure \ref{fig: recollision}).
%In this construction only the positions of $x_1$ and $x_1^0$ coincide.

\medskip

\noindent We refer to \cite{GSRT} for a complete proof of Proposition \ref{geometric-prop} and simply recall that it can be obtained from  the following control on free trajectories.  
\begin{Lem}\label{geometric-lem1}
  Given~$t>0$,   and $ \bar a>0$ satisfying~$A^{K+1} \eps \ll \bar  a \ll \eps_0\ll  \min(\delta E, 1)$,
   consider two points~$x^0_1,x^0_2$ in~$\T^d$ such that $d(x^0_1, x^0_2)\geq \eps_0$,  and a velocity~$ v_1 \in B_E$.
 Then there exists a subset~$K(x^0_1- x^0_2, \eps_0, \bar a) $ of~$\R^d$   
 with  measure bounded by
 $$ |K( x^0_1- x^0_2, \eps_0, \bar a) | \leq  CE^d \left(  \left(\frac{\bar a}{\eps_0}\right)^{d-1} +  (E t  )^d \;  {\bar a} ^{d-1}\right) $$
 and a subset $K_\delta( x^0_1- x^0_2, \eps_0, \bar a) $ of~$\R^d$, the measure of which satisfies
 $$  
 |K_\delta(x^0_1- x^0_2, \eps_0, \bar a) | \leq CE
  \left(  
  \left(\frac{\eps_0 }{\delta}\right)^{d-1} 
  +  (E t  )^d E^{d-1} \eps_0  ^{d-1}
 \right) 
 $$
 such that for any~$ v_2\in B_E$ and~$x_1, x_2$ such that $|x_1 - x^0_1| \leq \bar a$, $|x_2 - x^0_2| \leq \bar a$, the following results hold :

\noindent
 $ \bullet $ $ $ If   $v_1 -v_2 \not \in K( x^0_1- x^0_2, \eps_0, \bar a) $, then 
  $$
   \forall u \in [0,t]  \,, \quad d(x_1-  u \, v_1, x _2 - u \, v_2)>\eps 
   $$
 $ \bullet $ $ $ If  $v_1 -v_2 \notin K_\delta( x^0_1- x^0_2, \eps_0, \bar a) $
   $$
   \forall u \in  [\delta, t]   \,, \quad d(x_1- u \, v_1, x _2- u \, v_2) >\eps_0 \, .
 $$
 \end{Lem}

\noindent  The proof of this lemma is a simple adaptation of  Lemma 12.2.1 in \cite{GSRT}, and is given in Appendix B. Note that this is the only point of the convergence proof which differs in the case of the torus $\T^d$ from    the case of the whole space $\R^d$.
 In the case of the torus, there are indeed no longer  dispersion properties so waiting for a sufficiently long time, we expect trajectories to go back $\eps$-close to their initial positions.

\subsubsection{Induction procedure for the pseudo-trajectories} \label{traj-def}

Using the elementary step of Section~\ref{subsec: The elementary step}, we are going to construct in Proposition 
\ref{translation-lem} a coupling between the BBGKY and Boltzmann pseudo-trajectories, defined in Section \ref{subsec: pseudo-trajectories}, such that both trajectories remain close for all times up to a small error. In particular, this proof shows that recollisions may occur for the BBGKY pseudo-trajectories only for a set 
of configurations at time 0 in~${\mathcal D}^{J_K}_\eps$ with small measure.

\medskip

\noindent
As the stability of the good configurations (proved in Proposition \ref{geometric-prop}) requires a delay $\delta >0$ in between 2 collisions, we introduce  a modified set of collision times
\begin{align}
\label{eq: delay time}
{\mathcal T}_{J, \delta}(h) := &  \Big\{ T = (t_1,\dots,t_{J_K-1}) \,  /\,   t_i < t_{i-1} - \delta \, , \, 
(t_{J_k},\dots,t_{{J_{k-1} + 1} }) \in [t - k h, t- (k-1) h] \Big\} .
\end{align}
The following statement is analogous to Lemma~14.1.1 of~\cite{GSRT}.
 \begin{Prop}\label{translation-lem}
Fix $J=(j_1, \dots, j_K)$, $m = (m_1, \dots, m_{J_K-1}) \in \mathcal M_J$ and $T\in {\mathcal T}_{J ,\delta}(h)$. 
Let the pseudo-trajectories $Z_i = (X_i,V_i)$, $Z^0_i = (X^0_i,V_i)$ be defined inductively by choosing at each collision time $t_i$ a deflection angle~$\nu_{i+1}$ and a velocity~$v_{i+1}$ such that~
$$
(\nu_{i+1},v_{i+1}) \in  \big( {\mathbf S}^{d-1} \times  B_E \big) \setminus {\mathcal B}^{m_i}_{i}(Z_{i}^0(t_i))
\quad \text{and} \quad
\sum_{k=1}^{i+1} v_k^2 < E^2.
$$
The velocities of both pseudo-trajectories coincide as well as the positions $x_1 (u) = x_1^0(u)$ for~$u \in [0,t]$.
Furthermore, for~$\eps$ sufficiently small 
\begin{equation}\label{theresult}
\forall i  \leq J_K - 1,  \;
\forall \ell\leq i+1,\qquad | x_ \ell(t_{i+1}) - x_ \ell ^0(t_{i+1}) |   \leq \eps i \, .
%\quad \mbox{and}\quad v_ \ell(t_{i+1 }) = v_ \ell ^0(t_{i +1})  
\end{equation}
\end{Prop}

\noindent As a consequence of this proposition, we define  a bad set of velocities and deflection angles for 
the pathological pseudo-trajectories
\begin{align}
\label{eq: bad set}
\cB (z_1,J, T, m)  & := \Big \{ ( \nu_i, v_i)_{2 \leq i \leq J_K} \in \big( {\mathbf S}^{d-1} \times  B_E \big)^{J_K-1}
\; \Big| \; \sum_{k=1}^{J_K} v_k^2 < E^2 \quad \text{and} \quad \exists i_0 \leq J_K-1 \nonumber \\
&
\qquad \qquad \text{such that} \quad  \forall i < i_0, \quad
   ( \nu_{i+1}, v_{i+1}) \in \big( {\mathcal B}^{m_i}_{i}(Z_{i}^0(t_i)) \big)^c \\
& \qquad \qquad \text{and} \quad   ( \nu_{i_0+1}, v_{i_0+1}) \in {\mathcal B}^{m_{i_0}}_{i_0}(Z_{i_0}^0(t_{i_0}))
\Big\}. \nonumber
\end{align}
%This set is defined in terms of the configuration of $k$ particles at time $t_k$ (which will be a good configuration by construction), and therefore depends on the sequence of collision times $(t_1, t_2,\dots t_k)$, as well as on the initial state $z_1$ and of the velocities, deflection angles and parameters $(m_i, v_{i+1},\nu_{i+1})_{i<k}$ describing the previous collisions.

\begin{proof}
We proceed by induction on~$i$, the index of the time variables~$t_i$ for~$1\leq i \leq J_K -1$. 
The recursion hypothesis at step $i$ is
\begin{equation}
\label{recurrence hyp}
Z^0_i (t_{i}) \in \cG_i (\eps_0) 
\quad \mbox{and} \quad 
\forall \ell\leq i \, ,\qquad | x_ \ell(t_{i}) - x_ \ell ^0(t_{i}) |   \leq \eps (i-1) \, , 
\quad v_ \ell(t_{i}) = v_ \ell ^0(t_{i})  \, .
\end{equation}

\medskip

\noindent
We first notice that by construction, $ z_1( t_1) = z_1^0(t_1)$, 
so~(\ref{recurrence hyp}) holds for~$i=1$. The initial configuration containing only one particle, there is no possible  recollision!

\medskip

\noindent
Assume that (\ref{recurrence hyp}) holds up to some ~$i \leq J_K-1$ and let us prove that~(\ref{recurrence hyp}) holds for~$i+1$.
% and that $Z^0_{i+1} (t_{i+1}) \in \cG_{i+1} (\eps_0)$.
We shall consider two cases depending on whether the particle adjoined at time~$t_{i}$ is pre-collisional or post-collisional. 

\medskip
\noindent
$\bullet$
Let us start with the case of pre-collisional velocities $(v_{i+1}, v_{m_i}(t_i))$ at time $t_i$. 
We recall that the particle is adjoined in such a way that~$(\nu_{i+1}, v_{i+1})$  belongs 
to~$ \big( {\mathbf S}^{d-1} \times  B_E \big) \setminus {\mathcal B}_{i} ^{m_i} (  Z_{i}^0(t_i))$.
The new configuration $Z^0_{i+1}$ satisfies for all~$u\in ]t_{i+1},t_i]$
$$
\begin{aligned}
\forall \ell\leq i \,,\quad &x^0_ \ell(u ) =x^0_ \ell (t_i) +(u-t_i) v_ \ell(t_i)\,,\qquad &v_ \ell(u) =v_ \ell(t_i)\,,\\
&x^0_{i+1}(u) = x^0_{m_i} (t_i)  + (u -t_i) v_{i+1}\,, &v_{i+1}(u) =v_{i+1}\,.
\end{aligned}
$$
Since $t_i - t_{i+1} > \delta$, Proposition~\ref{geometric-prop} implies that $Z_{i+1}^0(t_{i+1})$ will be in~$\cG_{i+1}(\eps_0)$.

\medskip

\noindent
Now let us study $Z_{i+1}$ the BBGKY pseudo-trajectory.
 Provided that $\eps$ is sufficiently small, by the induction assumption~(\ref{recurrence hyp}) and the fact that 
$A^{K+1} \eps \leq \bar a$ (see  \eqref{sizeofparameters}), we have
$$
\forall \ell\leq i,\quad  | x_{\ell}  (t_i) -x_{\ell}^0 (t_i) |\leq \eps (i-1) \leq  \bar a \,.
$$

\noindent
Since~$Z_{i}^0(t_i)$ belongs to~$\cG_{i}(\eps_0)$, Proposition~\ref{geometric-prop}
 implies  that backwards in time, there is free flow for~$Z _{i+1}$. In particular,
\begin{equation}\label{elli+1}
\begin{aligned}
\forall \ell < i+1\,,\quad &x_ \ell(u) =x_ \ell (t_i) +(u -t_i) v_ \ell(t_i)\,,\qquad &v_ \ell(u) =v_ \ell(t_i)\,,\\
&x_{i+1}(u) = x_{m_i} (t_i) + \eps \nu_{i+1}  + (u -t_i) v_{i+1}\,,\qquad &v_{i+1}(u) =v_{i+1}\,.
\end{aligned}
\end{equation}
Therefore, the velocities of both configurations coincide and  by the induction assumption~(\ref{recurrence hyp}) 
$$\forall \ell\leq i+1\,,\quad  \forall u  \in ]t_{i+1},t_i] \, , \quad  
 |x_ \ell(u)- x_ \ell ^0(u)| \leq \eps (i-1) +\eps \leq \eps i 
$$
where we used that in \eqref{elli+1} there is a shift by at most $\eps$.
 
 \medskip
\noindent
$\bullet$
The case of post-collisional velocities  $(v_{i+1}, v_{m_i}(t_i))$ at time $t_i$ is identical up to a scattering of the 
velocities $ v_{i+1}, v_{m_i}$ in $v_{i+1}^*, v^*_{m_i}$.
Note that the constraint $\sum_{k=1}^{i+1} |v_k^2| < \frac{E^2}{2}$ implies that both velocities $v_{i+1}^*, v^*_{m_i}$ remain in $B_E$.
%We indeed have for $u  \in [t_{i+1},t_i[$
%$$ 
%\begin{aligned}
%\forall \ell < i+1, \quad \ell\neq m_i \,,\quad 
%&x^0_ \ell(u ) =x^0_ \ell (t_i) +(u -t_i) v^{*}_ \ell(t_i)\,,\quad &v_ \ell(u) =v_ \ell(t_i)\,,\\
%&x^0_{m_i}(u) =x^0_{m_i}  (t_i) +(u -t_i) v^*_{m_i}(t_i)\,,\quad &v_{m_i}(u) =v^{*}_{m_i}(t_i)\,,\\
%&x^0_{i+1}(u) = x^0_{m_i} (t_i)  + (u -t_i) v_{i+1}^*\,,\quad &v_{i+1}(u ) =v_{i+1}^*\,.
%\end{aligned}
%$$
%Now let us study the BBGKY pseudo-trajectory.
%We recall that the particle is adjoined in such a way that~$(\nu_{i+1}, v_{i+1})$  belongs 
%to~$\big( {\mathbf S}^{d-1} \times  B_E \big) \setminus {\mathcal B}_{i} ^{m_i} (  Z_{i}^0(t_i))$. Provided that $\eps$ is sufficiently small, by the induction assumption~(\ref{theresult}), we have
%$$\forall \ell\leq i,\quad  | x_{\ell} (t_i) -x_ \ell ^0 (t_i) |\leq \eps (i-1) \,.$$
%Then as in~(\ref{elli+1}), we can apply Proposition~\ref{geometric-prop} which implies that
%%$$
%%\forall \ell\leq i+1\,,\quad \forall s  \in [t_{i+1},t_i[ \, , \quad v_ \ell(s  ) -v_ \ell ^0(s ) = v_ \ell(t_i ^-) -v_ \ell ^0(t_i^-) = 0\, ,
%%$$
%$$\forall \ell\leq i+1\,,\quad  \forall s  \in [t_{i+1},t_i[\, , \quad  
% |x_ \ell(s )- x_ \ell ^0(s )| \leq \eps (i-1) +\eps \leq i\eps \,,
%$$
 This concludes the proof of Proposition~\ref{translation-lem}.
    \end{proof}

\subsection{Estimate of the error term}
\label{subsec: Estimates}

We turn now to the main goal of this section and use the coupling of Proposition \ref{translation-lem} between the hierachies to show that for $K \ll \log \log N$ and~$\alpha t \ll (\log \log N)^{(A-1)/A}$ then 
\begin{equation}
\label{eq: error estimate (1,K)}
\| f_N^{(1,K)} - g_\alpha^{(1,K)} \|_{L^\infty ([0,t] \times {\mathbf T}^d\times \R^d) }\to 0 %\, , \quad \mbox{as} \, \, N \to \infty \, .
\end{equation}
with an explicit rate of convergence when $N$ diverges.
The coupling of Proposition~\ref{translation-lem} can be implemented only for a reduced set of velocities taking values in $B_E$ and for collision times separated at least by $\delta$. Thus the first step will be to estimate the cost of cutting-off the large velocities and the collision time separation in 
\eqref{BBGKYfunctionalapprox 0} and \eqref{botzfunctionalapprox 0}.
Then in Section \ref{subsec: Estimate of the main term}, the parameters $\delta, E$ and $K$ will be tuned and the error term evaluated.

\subsubsection{Energy truncation}\label{energytruncation}

Given $E>0$, define the large velocity cut-off 
for $f^{(1,K)}_N$ introduced in  \eqref{eq: serie f1developpee}  as
\begin{align*}
f^{(1,K)}_{N,E} := \sum_J   \eps ^{(d-1)(J_K-1 )}    \frac{(N-1)!}{ (N-J_K )!} \; 
\sum_{m \in \mathcal{M}_J}  F^{(1,K)}_{N,E} (J,m)
\end{align*}
where $\sum_J$ stands for $ \sum_{j_1= 0}^{n_1-1} \dots     \sum_{j_K= 0}^{n_K-1} $ and
the velocities in the integral \eqref{BBGKYfunctionalapprox 0} are truncated
\begin{align}
& F^{(1,K)}_{N,E} (J,m) \, (t,z_1) \\
& \qquad := 
     \int_{{\mathcal T}_J(h) }  dT
    \int_{({\mathbf S}^{d-1} \times B_E)^{J_K-1}} \!  d \bar \nu \, d \bar V 
    \; \cA(T, z_1, \bar \nu, \bar V)   & \indc_{\{ H_{J_K} ( Z_{J_K}(0) ) \leq \frac{E^2}{2} \}} 
    f_N^{0(J_K)} \big( Z_{J_K} (0) \big) \nonumber
\end{align}
where $\cA$ was defined in \eqref{eq: product velocities} and 
$H_k ( Z_k ) = \frac{1}{2} \sum_{i = 1}^k |v_i|^2$.

%\begin{equation*}
%%\label{eq: product velocities}
%\cA(\bar \nu, \bar V) := \prod_{i=1}^{J_K - 1} ((v_{i+1} - v_{m_i} (t_i))\cdot \nu_{i+1} )
%\quad \text{and} \quad
%\begin{cases}
%\bar \nu = \{ \nu_2, \dots, \nu_{J_K} \}\\
%\bar V = \{ V_2, \dots, V_{J_K} \}
%\end{cases}
%\end{equation*}
In the same way, for $g _\alpha^{(1,K)}$ in~\eqref{eq: serie f1developpeeboltz}, the large velocity cut-off is
defined as 
\begin{align*}
g^{(1,K)}_{\alpha,E} := \sum_J \alpha^{J_K-1} \; \sum_{m \in \mathcal{M}_J}  G^{(1,K)}_{E} (J,m)
\end{align*}
where the velocities in the integral \eqref{botzfunctionalapprox 0} are truncated 
\begin{align}
& G^{(1,K)}_{E} (J,m) \, (t,z_1) \\
& \qquad :=  \int_{{\mathcal T}_J(h) }  dT
    \int_{({\mathbf S}^{d-1} \times B_E)^{J_K-1}} \!  d \bar \nu \, d \bar V 
    \; \hat \cA(T, z_1, \bar \nu, \bar V)  & \indc_{\{ H_{J_K} ( Z^0_{J_K}(0) ) \leq \frac{E^2}{2} \}} 
g^{0(J_K)}    (Z^0_{J_K} (0)) \, . \nonumber
\end{align}
%
% \begin{equation}
%\label{gs1-fs-R}
%\begin{aligned}
%f_{N,E}^{(1,K)}  : = \sum_{j_1=1}^{n_1-1} \dots \sum_{j_K=0}^{n_K-1}Q_{1,J_1} (h )Q_{J_1,J_2} (h )
% \dots  Q_{J_{K-1},J_K} (h )   \indc_{H_{J_K}(Z_{J_K})\leq E^2}  \, f^{0(J_K)}_N\, , \\
%g^{(1,K)}_{N,E}   := \sum_{j_1=1}^{n_1-1} \dots \sum_{j_K=0}^{n_K-1}Q^0_{1,J_1} (h )Q^0_{J_1,J_2} (h )
% \dots  Q^0_{J_{K-1},J_K} (h )   \indc_{H_{J_K}(Z_{J_K})\leq E^2}  \,  g_N^{0(J_K)} \, .
% \end{aligned}
%\end{equation}
Then, we have the following error estimate.
        \begin{Prop}
\label{truncatedenergy}
There is  a constant~$C$ depending only on $\beta$ and $d$ such that, as~$N$ goes to infinity in the scaling~$N \eps^{d-1} \alpha^{-1} \equiv 1$, the following bounds hold:
$$
\begin{aligned}
\| f_{N}^{(1,K)}- f_{N,E}^{(1,K)} \|_{L^\infty ([0,t]\times {\mathbf T}^{d}\times \R^{d})  }+
\| g_\alpha^{(1,K)} - g^{(1,K)}_{\alpha,E}\|_{L^\infty ([0,t]\times{\mathbf T}^{d}\times \R^{d})  } \\
\leq   A^{K(K+1)} (C \alpha  t)^{A^{K+1} } e^{-\frac \beta 4  E^2}  \|\rho^0\|_{L^\infty}\, ,
\end{aligned}
$$
 with $A, K$ as in Proposition {\rm\ref{prop: remainders}}.
 \end{Prop}

 \begin{proof}
We first consider the BBGKY hierarchy.
Since the kinetic energy is preserved by the transport ${\bf S}_k$, the difference $(f_{N}^{(1,K)}- f_{N,E}^{(1,K)} )$ can be bounded from above by estimating the contribution of the pseudo-trajectories such that $\{ H_{J_K}(Z_{J_K}(0))\geq \frac{E^2}{2} \}$ at time 0.
Note that from \eqref{estimatefNktweight}
\begin{equation}\label{estimateexp}
\|  \indc_{\{ H_{J_K}(Z_{J_K})\geq \frac{E^2}{2}  \}}  \, f^{0 (J_K)}_N \|_{\e, J_K, \beta/2} 
\leq \|   f^{0 (J_K)}_N   \|_{\e, J_K, \beta} \,  e^{-\frac \beta 4  E^2}
\leq C^{J_K} \, e^{-\frac \beta 4  E^2}   \|\rho^0\|_{L^\infty}\, .
\end{equation}
By Lemma \ref{continuity},   we get
$$
\begin{aligned}
&\|F^{(1,K)}_N (J,m) - F^{(1,K)}_{N,E} (J,m) \|_{L^\infty ([0,t]\times {\mathbf T}^{d}\times \R^{d})  }
% & \Big\| Q_{1,J_1} (h )Q_{J_1,J_2} (h )
% \dots  Q_{J_{K-1},J_K} (h )   \indc_{H_{J_K}(Z_{J_K})\geq E^2}  \, f^{(J_K)}_N(0)\Big\|_{L^\infty ({\mathbf T}^{d}\times \R^{d})  } 
\\
& \qquad \qquad \qquad  \leq  \Big\| |Q|_{1,J_K} (t)   
\indc_{\{ H_{J_K}(Z_{J_K}) \geq \frac{E^2}{2}  \} }  
 \, f^{0 (J_K)}_N\Big\|_{L^\infty ([0,t]\times {\mathbf T}^{d}\times \R^{d})  } \\
& \qquad \qquad \qquad  
\leq \left( {C  t \over (\beta/2) ^{(d+1)/2} } \right) ^{J_K-1 }   	
 \| \indc_{\{ H_{J_K}(Z_{J_K})\geq \frac{E^2}{2}  \} }    f^{0 (J_K)}_N\|_{\e, J_K, \beta/2} \, .
  \end{aligned}
 $$
It follows that 
  $$
 \begin{aligned}
\|F^{(1,K)}_N (J,m) - F^{(1,K)}_{N,E} (J,m) \|_{L^\infty ([0,t]\times {\mathbf T}^{d}\times \R^{d})  }
% & \Big\| Q_{1,J_1} (h )Q_{J_1,J_2} (h )
% \dots  Q_{J_{K-1},J_K} (h )   \indc_{H_{J_K}(Z_{J_K})\geq E^2}  \, f^{(J_K)}_N(0)\Big\|_{L^\infty ({\mathbf T}^{d}\times \R^{d})  }\\
 \leq   (  C     t)^{A^{K+1} } e^{-\frac \beta 4  E^2}  \|\rho^0\|_{L^\infty}
 \end{aligned}
 $$
 thanks to~(\ref{estimateexp}) and to the fact that~$J_K \leq {\mathcal N}_K \leq A^{K+1}$.
A similar estimate holds for the Boltzmann hierarchy. Summing over all possible choices of $j_k$  proves the proposition, recalling that in the Boltzmann-Grad scaling 
$$
( \eps ^{d-1} )^{J_K-1} \frac{(N-1)!}{ (N-J_K )!} \leq \alpha^{J_K -1} \, .
$$
Proposition~\ref{truncatedenergy} is proved.
 \end{proof}

\subsubsection{Time separation}
\label{timetruncation}

We choose  a small parameter~$\delta>0$ such that~$A^K \delta \ll h$  and estimate the error for separating the collision times by at least $\delta$.
The time cut-off of the pseudo-trajectories is defined as
\begin{align}
f^{(1,K)}_{N,E, \delta}
:= \sum_J  \eps ^{(d-1)(J_K-1) }  \frac{(N-1)!}{ (N-J_K )!} \;
\sum_{m \in \mathcal{M}_J}  F^{(1,K)}_{N,E, \delta} (J,m)
\end{align}
where the time integrals are restricted to the set 
${\mathcal T}_{J, \delta}(h)$ defined in \eqref{eq: delay time}
\begin{align*}
& F^{(1,K)}_{N,E, \delta} (J,m) \, (t,z_1) \\
& \qquad := 
     \int_{{\mathcal T}_{J, \delta} (h) }  dT
    \int_{({\mathbf S}^{d-1} \times B_E)^{J_K-1}} \!  d \bar \nu \, d \bar V 
    \; \cA(T, z_1, \bar \nu, \bar V) \;    1_{\{ H_{J_K} ( Z_{J_K}(0) ) \leq \frac{E^2}{2} \}} f_N^{0(J_K)}    (Z_{J_K} (0))  \nonumber
\end{align*}
with $\cA(t, z_1,\bar \nu, \bar V)$ as in \eqref{eq: product velocities}.
In the same way, for the Boltzmann hierarchy, we set
\begin{align*}
g^{(1,K)}_{\alpha, E, \delta} := \sum_J  \alpha^{J_K-1}\, 
\sum_{m \in \mathcal{M}_J} G^{(1,K)}_{E, \delta} (J,m)
\end{align*}
where the separation time cut-off is defined as
\begin{align*}
& G^{(1,K)}_{E, \delta} (J,m) \, (t,z_1) \\
& \qquad :=  \int_{{\mathcal T}_{J, \delta} (h) }  dT
    \int_{({\mathbf S}^{d-1} \times B_E)^{J_K-1}} \!  d \bar \nu \, d \bar V 
    \; \hat \cA(T, z_1, \bar \nu, \bar V) \;  1_{\{ H_{J_K} ( Z^0_{J_K}(0) ) \leq \frac{E^2}{2} \}}
     g^{0(J_K)}    (Z^0_{J_K} (0)) .
\end{align*}
Then the following holds.
 \begin{Prop}\label{delta2smallHS}
There is  a constant~$C$ depending only on $\beta$ and $d$ such that, as~$N$ goes to infinity in the scaling~$N \eps^{d-1} \alpha^{-1} \equiv 1$, the following holds
\begin{align}
& \| f_{N,E}^{(1,K)}- f_{N,E,\delta}^{(1,K)} \|_{L^\infty ([0,t] \times{\mathbf T}^{d}\times \R^{d})  } +\| g^{(1,K)}_{\alpha,E} - g^{(1,K)}_{\alpha,E,\delta}\|_{L^\infty ([0,t] \times{\mathbf T}^{d}\times \R^{d})  } 
\nonumber \\  
& \qquad \qquad \qquad 
\leq    A^{(K+2)(K+1)} (C \alpha t)^{A^{K+1} }  \,  { \delta \over t}   \|\rho^0\|_{L^\infty}\, ,
\label{eq: time cut-off}
\end{align}
with $A,K$ as in Proposition~{\rm\ref{prop: remainders}}.
\end{Prop}

 \begin{proof}
Given $J,m$ the difference $(F_{N,E}^{(1,K)} - F_{N,E,\delta}^{(1,K)}) (J, m)$ involves the integration over two consecutive times such that $|t_{i+1} - t_i | \leq \delta$. This leads to a contribution $\delta t^{J_K-2}/ (J_K-2)!$ instead of  $t^{J_K-1}/ (J_K-1)!$ and there are $J_K-1$ possible choices for the collision with a short time separation.
Modifying accordingly the estimates of Lemma \ref{continuity}, we get for a given $J$
\begin{align*}
& \big\| 
\sum_{m \in \cM_J} \, \big( F_{N,E}^{(1,K)} - F_{N,E,\delta}^{(1,K)} \big) (J, m)
\big\|_{L^\infty ([0,t] \times{\mathbf T}^{d}\times \R^{d})  }  
%& \qquad \qquad 
\leq   (C \alpha t)^{A^{K+1} } \,   \,  { ( A^{K+1} )^2 \,  \delta\over t}   \|\rho^0\|_{L^\infty}\, ,
\end{align*}
where we used that $J_K \leq A^{K+1}$.
Summing over all possible choices of $j_k$ leads to an extra factor $A^{K (K+1)}$ as in 
\eqref{eq: time cut-off}.

%The point is to see that the following continuity estimate holds:
% $$
% \Big\| \big ( Q_{1,n} (t) -Q_{1, n}^\delta (t)\big)    \, f_n \Big\|_{L^\infty ([0,t] \times {\mathbf T}^{d}\times \R^{d})  } \\
% \leq \left( {C  t\over \beta  ^{\frac d2+1}}  \right) ^{n-1 }    	n^2{\delta\over t} \| f_n\|_{\e, n, \beta}\\
% $$
%since the integration with respect to time provides $\delta t^{n-1}/ (n-1)!$ instead of  $t^n/n!$, and there are $n$ possible choices of the integral to be modified.
\medskip

\noindent A similar estimate holds  in the Boltzmann case and completes the proof.
\end{proof}

\subsubsection{Neglecting the pathological pseudo-trajectories}

We now reduce the domain of integration of the velocities and deflection angles
outside  the set $\cB (J, T, m)$ defined in \eqref{eq: bad set} in order to remove the pathological pseudo-trajectories.
We set
\begin{align}
\label{eq: f tronquee}
\tilde f^{(1,K)}_{N,E, \delta}   = \sum_J  \alpha^{J_K-1} 
\; \sum_{m \in \mathcal{M}_J} 
\left( \frac{\eps ^{d-1}}{\alpha} \right)^{J_K-1} \frac{(N-1)!}{ (N-J_K )!} \; \tilde F^{(1,K)}_{N,E, \delta} (J,m)
\end{align}
where %$\sum_J$ stands for $ \sum_{j_1= 1}^{n_1-1} \dots     \sum_{j_K= 0}^{n_K-1} $ and
\begin{align}
\label{eq: F tronquee}
& \tilde F^{(1,K)}_{N,E, \delta} (J,m) \, (t,z_1)\\
& \qquad  := 
     \int_{{\mathcal T}_{J, \delta} (h) }  dT
    \int_{\cB (J, T, m)^c} \!  d \bar \nu \, d \bar V 
    \; \cA(T, z_1, \bar \nu, \bar V) \;  1_{\{ H_{J_K} ( Z_{J_K}(0) ) \leq \frac{E^2}{2} \}}   f_N^{0(J_K)}    (Z_{J_K} (0))  \nonumber
\end{align}
with $\cA(t, z_1,\bar \nu, \bar V)$ as in \eqref{eq: product velocities}.
In the same way, we define
\begin{align}
\label{eq: g tronquee}
\tilde g^{(1,K)}_{\alpha, E, \delta}  = \sum_J \alpha^{J_K-1} \; \sum_{m \in \mathcal{M}_J}  
\tilde G^{(1,K)}_{E, \delta} (J,m)
\end{align}
where the domain of integration is restricted to the complement of $\cB (J, T, m)$
\begin{align}
\label{eq: G tronquee}
&\tilde G^{(1,K)}_{E, \delta} (J,m) \, (t,z_1) \\
& :=  \int_{{\mathcal T}_{J, \delta} (h) }  dT
    \int_{\cB (J, T, m)^c} \!  d \bar \nu \, d \bar V 
    \; \hat \cA(T, z_1,\bar \nu, \bar V) \;  1_{\{ H_{J_K} ( Z^0_{J_K}(0) ) \leq \frac{E^2}{2} \}} \;  
       g^{0(J_K)}    (Z^0_{J_K} (0)) . \nonumber
\end{align}

 \noindent As a consequence of Proposition~\ref{geometric-prop} and of the continuity estimates in Lemma \ref{continuity}, the error induced by neglecting the pathological pseudo-trajectories can be estimated from above.
 \begin{Prop}\label{lastapprox}
  Let $\bar a,\eps_0, \delta$  satisfying {\rm(\ref{sizeofparameters})}. 
There is  a constant~$C$ depending only on~$\beta$ and $d$ such that, as~$N$ goes to infinity in the scaling~$N \eps^{d-1} \alpha^{-1} \equiv 1$, the following holds
$$
\begin{aligned}
&\Big\|     g^{(1,K)}_{\alpha, E,\delta}  - \widetilde g^{(1,K)}_{\alpha, E,\delta}  \Big\|_{L^\infty(   [0,t] \times {\mathbf T}^{d} \times \R^{d})} +\Big\|  f^{(1,K)}_{N,E,\delta}  - \widetilde f^{(1,K)}_{N,E,\delta}  \Big\|_{L^\infty(  [0,t] \times{\mathbf T}^{d} \times \R^{d})}   \\
& \qquad \leq  
  A^{(K+2)(K+1)} (C \alpha t)^{A^{K+1} } 
\left( E^d  \left( {\bar a \over \eps_0}\right)^{d-1} +E^d  (Et  )^d  \eps_0  ^{d-1}+E\left({ \eps_0 \over \delta}\right)^{d-1}  \right)  \|\rho^0\|_{L^\infty} \,  .
\end{aligned}
$$
\end{Prop}

\begin{proof}
The proof follows  the same lines as the proofs of Propositions \ref{truncatedenergy} and \ref{delta2smallHS}. 
In the usual continuity estimate for the elementary collision operator, the integration with respect to velocity brings a factor $(2\pi /\beta)^{d/2}$, while removing the integration over the pathological set~${\mathcal B}_{k}^{m_{k}}$ gives an error
\begin{equation}
\label{eq: erreur Bm}
C k \left( E^d  \left( {\bar a \over \eps_0}\right)^{d-1} +E^d \Big(Et \Big)^d \, \eps_0^{d-1}+E\left({ \eps_0 \over \delta}\right)^{d-1}  \right)
\end{equation}
according to Proposition~\ref{geometric-prop}.

\noindent

For a given $J$, there are $J_K - 1\leq A^{K+1}$ possible choices of the integral to be modified. 
Therefore, the estimate on the collision operator leads to 
$$
\begin{aligned}
&\big\| 
\sum_{m \in \cM_J} \, \big( \widetilde F_{N,E,\delta}^{(1,K)} - F_{N,E,\delta}^{(1,K)} \big) (J, m)
\big\|_{L^\infty ([0,t] \times{\mathbf T}^{d}\times \R^{d})  } \\
%&\Big\|     g^{(1,K)}_{\alpha,E,\delta}  - \widetilde g^{(1,K)}_{\alpha, E,\delta}  \Big\|_{L^\infty(  [0,t] \times{\mathbf T}^{d} \times \R^{d})} 
%+
%\Big\| f^{(1,K)}_{N,E,\delta}  - \widetilde f^{(1,K)}_{N,E,\delta} \Big\|_{L^\infty(  [0,t] \times{\mathbf T}^{d} \times \R^{d})}    \\
& \qquad  \leq     (C t)^{A^{K+1} } A^{2 (K+1)}\left( E^d  \left( {\bar a \over \eps_0}\right)^{d-1} +E^d (Et  )^d  \eps_0  ^{d-1}+E\left({ \eps_0 \over \delta}\right)^{d-1}  \right)  \|\rho^0\|_{L^\infty}
\end{aligned}
$$
where as previously $C$ depends only on $d$ and $\beta$.
The term $A^{2 (K+1)}$ comes from the $A^{K+1}$ possible choices of the integral to be modified and from
the additional factor $k \leq A^{K+1}$ in \eqref{eq: erreur Bm}.

\noindent Finally summing over all the possible choices of $J = (j_1, \dots, j_K)$ provides the additional factor~$A^{K(K+1)}$ in the estimate. Similar bounds hold also for the Boltzmann hierarchy. This completes the Proposition.
\end{proof}

\noindent 
Once the pathological pseudo-trajectories have been removed, the integrals \eqref{eq: F tronquee} and \eqref{eq: G tronquee} differ only by the small error on the positions $Z_{J_K} (0), Z^0_{J_K} (0)$ and by the initial data  $f_N^{0(J_K)}$ and~$ g^{0(J_K)}$.
Thus, one gets
\begin{Prop}
\label{lastapprox250000}
There is a constant~$C$ depending only on $\beta$ and $d$ 
 such that, as~$N$ goes to infinity in the scaling~$N \eps^{d-1}= \alpha$, 
 the following holds
\begin{align*}
\Big\|  \tilde f^{(1,K)}_{N,E, \delta} - \tilde g^{(1,K)}_{\alpha, E, \delta} 
%\widetilde g^{(1,K)}_{N,E,\delta}  - \widetilde g^{(1,K)}_{ N,E,\delta,\eps}  
\Big\|_{L^\infty(   [0,t] \times {\mathbf T}^{d} \times \R^{d})}   
\leq   A^{K(K+1)} (C \alpha t)^{A^{K+1} }  \left( \frac{A^{2(K+1)}}{N} + \alpha \eps \right)  \|\rho^0\|_{L^\infty} \,  .
\end{align*}
%where $| \nabla \rho^0 |$ stands for the Lipschitz norm of $\rho^0$ introduced in \eqref{initial}.
\end{Prop}

\begin{proof}  
There are 2 sources of discrepancies between  \eqref{eq: f tronquee} and \eqref{eq: g tronquee}.

\smallskip

\noindent
$\bullet$
{\it The prefactors in the collision operators} :
In \eqref{eq: f tronquee}, the elementary collision operators have prefactors of the type $ (N-k) \eps^{d-1}/\alpha$ that can be replaced in the limit by $1$.
For fixed $J_K$, the corresponding error is 
$$ \Big( 1- {(N-1) \dots (N- J_K+1) \over N^{J_K+1}}\Big)  \leq C {J_K^2 \over N}$$
which, combined with the bound on the collision operators,  leads to an error of the form 
$$ A^{K(K+1)} (C \alpha t)^{A^{K+1} } \frac{A^{2(K+1)} }{N}   \|\rho^0\|_{L^\infty}\,\cdotp$$
%We therefore conclude that 
%$$\Big\|   \widetilde f^{(1,K)}_{N,E,\delta}-\widetilde   g^{(1,K)}_{N,E,\delta
% } \Big\|_\infty \leq C\mu_N^2 2^{K(K+1)} (C t)^{2^{K+1} }\Big(  2^{2K}\mu_N \Lambda_N \eps +2^{2K} \eps^{d-1}   \Big) \,,$$
%which concludes the proof.

\medskip

\noindent
$\bullet$ {\it Discrepancy between $f_N^{0(J_K)}(Z_{J_K} (0))$  and $g^{0(J_K)}(Z^0_{J_K} (0))$} : 
First of all, we note that for the coupled pseudo-trajectories
$$
g^{0(J_K)}   (Z_{J_K} (0)) = g^{0(J_K)}   (Z_{J_K} ^0(0)). 
$$ 
Indeed, by construction both pseudo-trajectories have the same velocities and $x_1 = x^0_1$. 
The differences between the two configurations are only on the positions of the particles added and $g^{0(J_K)}$ is independent of these positions.

\smallskip

\noindent By Proposition~\ref{translation-lem}, the initial data satisfies $Z_{J_K}(0)\in \cG_{J_K}(\eps_0/2)$.
According to Proposition~\ref{exclusion-prop2}, we have
$$
\Big\| \indc_{  \cG_{J_K}(\eps_0/2)}(  f_N^{0(J_K)} - g^{0(J_K)}) \Big\|_{0,J_K,\beta } \leq   \|\rho^0\|_{L^\infty} \, 
C^{J_K} \,  \alpha \eps \,.
$$ 
% It follows that XXXXXXXX
%$$
%|\lambda^{d } f_N^{0(J_K)} -  |\lambda^{d(1-J_K)} f_N^{0(J_K)} - M_\beta^{\otimes J_K} \varphi^0M_\beta^{\otimes J_K} \varphi^0
%$$
%$$\indc_{Z_k \in {\mathcal D}_k}f^{(k)}_{0} 
% -  \lambda^{dk}  f_{0,N}^{(k)} = \Big(1 - \lambda^{dk} {\mathcal Z}_{N}^{-1} {\mathcal Z}_{N-k}\Big) \indc_{Z_k \in {\mathcal D}_k}f^{(k)}_{0}  
%+  \lambda^{dk} {\mathcal Z}_N^{-1} {\mathcal Z}^\flat_{(k+1,N)} \indc_{Z_k \in {\mathcal D}_k}f^{(k)}_{0}   $$
%with 
%$$ 
% \Big|1 - \lambda^{dk} {\mathcal Z}_{N}^{-1} {\mathcal Z}_{N-k}\Big| \leq (1-\eps \kappa_d )^{-k} - 1\leq \e k \k_d $$
% and
% $$ \lambda^{dk}  {\mathcal Z}_N^{-1} {\mathcal Z}^\flat_{(k+1,N)}  \leq \e k \k_d \big(1 - \e \k_d \big)^{-(k+1)} \,.$$
 Using the continuity estimate in Lemma \ref{continuity}, we then deduce that 
the error due to the initial data can be controlled by
$$   \|\rho^0\|_{L^\infty} \, A^{K(K+1)} (C \alpha t)^{A^{K+1} }   \, \alpha \eps .$$
This concludes Proposition \ref{lastapprox250000}.
\end{proof}

\subsubsection{Estimate of the main term}
\label{subsec: Estimate of the main term}

Finally combining the previous estimates, we get 
\begin{Prop}
\label{main}
For  parameters satisfying \eqref{eq: parametre conditions} and such that 
\begin{equation}
\label{eq: parameters}
\alpha t  \ll  \big( \log \log N \big)^{\frac{A-1}{A}}
\quad \text{and} \quad 
K \leq \frac{\log \log N }{2 \log A}
\end{equation}
then as $N$ goes to infinity
\begin{equation}
\label{eq: scaling partiel}
\big\|  f_N^{(1,K)} - g_\alpha^{(1,K)}  \big \|_{L^\infty([0,t] \times {\mathbf T}^d\times \R^d)}
\leq     \|\rho^0\|_{L^\infty} \,  \eps^{\frac{d-1}{d+1}} \, \exp \left( C  \,  (\log N)^{1/2}
\, \log \log N  \right).
\end{equation}
\end{Prop}
\noindent In particular, Estimate \eqref{eq: error estimate (1,K)} follows from Proposition \ref{main}.

\begin{proof}
We write 
\begin{equation*}
\big\|  f_N^{(1,K)} - g^{(1,K)}_\alpha \big \|_{L^\infty} 
\leq \big\| f_N^{(1,K)} -  \tilde f^{(1,K)}_{N,E, \delta}  \big \|_{L^\infty}
+\big\|  g^{(1,K)}_\alpha - \tilde g^{(1,K)}_{\alpha, E, \delta} \big \|_{L^\infty}
+
\big\|   \tilde f^{(1,K)}_{N,E, \delta} - \tilde g^{(1,K)}_{\alpha, E, \delta} \big \|_{L^\infty} .
\end{equation*}
Let $\bar a,\eps_0,\delta, E$  satisfying {\rm(\ref{sizeofparameters})}. 
By gathering together the estimates in Propositions  
 \ref{truncatedenergy}, \ref{delta2smallHS}, \ref{lastapprox} and \ref{lastapprox250000}, we see that 
there exists $C$ depending only on $\beta$ and $d$ such that, as~$N$ goes to infinity in the scaling~$N \eps^{d-1} \alpha^{-1} \equiv 1$, the following holds 
$$
\begin{aligned}
& \big\|  f_N^{(1,K)} - g^{(1,K)}_\alpha  \big \|_{L^\infty([0,t] \times {\mathbf T}^d\times \R^d)} \\
 &\quad \leq %\ C_0 \,  \gamma^2 + 
   A^{K(K+1)} (C  \alpha t)^{A^{K+1} } \left( e^{-\frac \beta 4  E^2} + {A^{2 (K+1)} \delta \over t}  \right)  \|\rho^0\|_{L^\infty} \\
&\qquad +     A^{ (K+2) (K+1)} (C  \alpha  t)^{A^{K+1} } 
\left( E^d  \left( {\bar a \over \eps_0}\right)^{d-1} +E^d \Big(Et \Big)^d \eps_0^{d-1}+E\left({ \eps_0 \over \delta}\right)^{d-1} \right) \,  \|\rho^0\|_{L^\infty} \\
&\qquad +  A^{K(K+1)} (C  \alpha t)^{A^{K+1} }   
\left( \frac{A^{2(K+1)}}{N} + \eps \alpha \right) \, \|\rho^0\|_{L^\infty} 
\end{aligned}
$$
with $A, K$ introduced in Proposition~{\rm{\ref{prop: remainders}}}.

\medskip
 
\noindent To derive the upper bound \eqref{eq: scaling partiel},  we choose for
the parameters   the following orders of magnitude
$$
\delta \sim \eps^{\frac{d-1}{d+1}} \, , \quad \eps_0 \sim  \eps^{\frac{d }{d+1}} \, ,  \quad E \sim \sqrt {|\log \eps|} , \quad \bar a = A^{K+1} \eps \, .
% \, ,  \quad K \leq \frac{\log |\log \eps|}{2 \log A}
$$
This leads to 
\begin{align*}
\big\|  f_N^{(1,K)} - g_\alpha^{(1,K)}  \big \|_{L^\infty([0,t] \times {\mathbf T}^d\times \R^d)}
&\leq  ( C \, \alpha t)^{A^{K+1} }  A^{2 K (K+1) } \, \big( \eps ^{\frac{d-1}{d+1}} |\log \eps|^d  + \eps^{d-1}\big) 
\| \rho^0 \|_\infty
\end{align*}
from which \eqref{eq: scaling partiel} can be deduced in the scaling \eqref{eq: parameters}
since $A^K \leq \sqrt{\log N}$. 
\end{proof}

\bigskip

\noindent Equipped with all these estimates, we  prove now Theorem \ref{long-time}.
\begin{proof}[Proof of Theorem {\rm\ref{long-time}}]
  Propositions  \ref{prop: remainders} and \ref{main} imply  that
with the scaling \eqref{eq: parameters}
$$
\begin{aligned}
  \big\|  f_N^{(1)} - g_\alpha  \big \|_{L^\infty([0,t] \times {\mathbf T}^d\times \R^d)}
&\leq  C \,   \, \left(  \gamma^A +   C_0 \,   \eps^{\frac{d-1}{d+1}} \, \exp \left( C  \,  (\log N)^{1/2}
\, \log \log N  \right) \right)   \|\rho^0\|_{L^\infty}
\\
&
\leq  C \,   \, \left(  \frac{(\alpha t)^{A/(A-1)}}{\log \log N} \right)^A  \|\rho^0\|_{L^\infty}
,
\end{aligned}
$$
where we have used the relation $\gamma = \frac{ (\alpha t)^{A/(A-1)} }{C K}$ of \eqref{eq: tau nk}
with the choice $K = \lfloor \frac{\log \log N }{2 \log A} \rfloor$.
This conludes the proof of Theorem {\rm\ref{long-time}}.

\noindent Note that the relevant scaling for this upper bound is $\alpha t = o \Big(  (\log \log N)^{(A-1)/A} \Big)$.
\end{proof}

%%%%%%%%%%%%%%%%%%%%%%%%%%%%%%%%%%%%%%%%%%%%%%%%
%%%%%%%%%%%%%%%%%%%%%%%%%%%%%%%%%%%%%%%%%%%%%%%%

\section{Proof of the diffusive limit : proof of Theorem \ref{brownien}}\label{proofbrownien}\noindent 
In Theorem \ref{long-time}, we have shown that the tagged particle distribution $ f_N^{(1)}(t,x,v)$ remains close to~$M_\beta(v) \varphi_\alpha(t,x,v)$ where~$\varphi_\alpha$ solves  the linear-Boltzmann equation~(\ref{linear-boltzmann}) 
 on $\T ^d\times \R^d$,  with initial data~$\rho^0 (x)$.
More generally, our proof implies that the whole trajectory of the tagged particle $\{ x_1 (s) \}_{s \leq T }$ can be approximated with high probability by the trajectory of~$\{   x_1^0 (s) \}_{s \leq T }$ (see Lemma~\ref{translation-lem}).
The latter process is much simpler to study as its velocities are given by a  Markov process.

\smallskip
\noindent These two points of view lead to two strategies to prove the diffusive limit. 
We   first present an analytic approach to show that $  \varphi_\alpha(\alpha\tau,x,v)$ can be approximated by the diffusion~\eqref{heat}.
Then we turn to an alternative method to show the convergence of the   trajectory 
to a brownian  motion which will rely on probabilistic estimates for $\{   x^0_1 (\alpha \tau) \}_{\tau \leq T}$.

\medskip

\noindent In the following  the macroscopic time variable will be denoted   by $\tau\in [0,T]$.

\subsection{Convergence to the heat equation}
\label{subsec: Convergence to the heat equation}
In this section we  prove the result~(\ref{limitfN1toheat}) stating the convergence of $ f_N^{(1)}(\alpha\tau,x,v)$ to~$M_\beta(v) \rho( \tau,x)$ where~$\rho$ solves  the heat equation~(\ref{heat})  on~$\T ^d $,  with initial data~$\rho^0 (x)$. We show in Paragraph~\ref{boilsdown} that this can be reduced    to proving that~$\varphi_\alpha(\alpha\tau,x,v)$ 
can be approximated by a diffusion, which is a standard procedure (see~\cite{BSS}). For the sake of completeness, we recall the salient features of the proof in Paragraphs~\ref{hilbert} and~\ref{proodoftheconvergence}.

 \subsubsection{Approximation by the linear Boltzmann equation}
 \label{boilsdown}

The explicit convergence rate provided in Theorem \ref{long-time} implies in particular that for any~$\tau>0$ and any~$\alpha > 1$,  in the limit~$N \to \infty$,~$N\varepsilon^{d-1} \alpha^{-1}= 1$, one has
\begin{equation}
\label{eq: approx temps gd2}
\big\|  f_N^{(1)}(\alpha \tau,  x,v)- M_\beta(v) \varphi _\alpha (\alpha \tau , x,v) 
\big\|_{L^\infty( {\mathbf  T}^d\times  {\mathbf R}^d)} \leq C   
\left[ \frac{ \alpha^2  \tau}{ ({\log\log N})^{\frac{A-1}{A} } } \right]^{\frac{A^2}{A-1}} \, ,
\end{equation}
where $A\geq 2$ can be taken arbitrarily large. It is therefore possible to   take the limit~$\alpha \to \infty$ while conserving a small right-hand side in~(\ref{eq: approx temps gd2}), as soon as~$\alpha \ll    ({\log \log N})^{\frac{A-1}{2A}}$. 

\smallskip
\noindent Let us define
$$
\widetilde \varphi_\alpha (\tau,x,v):=   \varphi_\alpha (\alpha\tau,x,v) \, , 
$$
which satisfies
\begin{equation}
\label{eq: widetildephialpha}
\partial_\tau \widetilde \varphi_\alpha + \alpha \, v \cdot \nabla_x  \widetilde \varphi_\alpha + \alpha^2 {\mathcal L}   \widetilde \varphi_\alpha  = 0 \, .
\end{equation}
Then~(\ref{limitfN1toheat}) follows directly from the following result
\begin{equation}\label{limitfN1toheat2}
\sup_{\tau  \in [0,T]} \sup_{(x,v) \in  {\mathbf T}^d \times {\mathbf R}^d} 
\Big |M_\beta(v) \big( \widetilde \varphi_\alpha (\tau,x,v) -  \rho (\tau, x)  \big)  \Big| 
\to 0
\end{equation}
in the limit~$\alpha\to \infty$.  The rest of this paragraph is devoted to the proof of~(\ref{limitfN1toheat2}). Notice that by the maximum principle  on the heat equation, we may assume without loss of generality 
(up to regularizing $\rho^0$) that~$\rho^0$ belongs to~$C^4(\T^d)$, which will be useful at the end of the proof.  

 \subsubsection{Hilbert's expansion} \label{hilbert}The formal Hilbert expansion consists in writing an asymptotic expansion of~$ \widetilde \varphi_\alpha $ in terms of powers of~$\alpha^{-1}$
 $$
 \widetilde \varphi_\alpha (\tau,x,v)  =  \widetilde\rho_{0} (\tau, x,v)  + \frac1\alpha\widetilde  \rho_{1} (\tau, x,v) +   \frac1{\alpha^2} \widetilde \rho_{2} (\tau, x,v) 
+ \dots \, ,
 $$
 in plugging that expansion in Equation~(\ref{eq: widetildephialpha}),
 and  in canceling successively all the powers of~$\alpha$. This gives formally the following set of equations (where we have considered only the~$O(1)$,~$O(\alpha)$ and~$O(\alpha^2)$  terms)
 \begin{equation}\label{eqhilbert}
 \begin{aligned}
 {\mathcal L} \widetilde  \rho _{  0}= 0  \, , \\
  v \cdot \nabla_x\widetilde\rho_{  0} +  {\mathcal L}\widetilde  \rho_{ 1} = 0 \, , \\
 \partial_\tau \widetilde\rho_{ 0} +    v \cdot \nabla_x\widetilde\rho_{ 1} +  {\mathcal L} \widetilde \rho_{ 2} = 0  \, .\end{aligned}
\end{equation}
 In order to find the expressions for~$\widetilde\rho_{ 1}$ and~$\widetilde\rho_{ 2}$, as well as the equation on~$\widetilde\rho_{  0}$ (which we expect to be the heat equation),  it is necessary   to be able to invert the operator~$\mathcal{L}$. This is made possible by the following result, whose proof can be found in~\cite{hilbert2} (in the case of the linearized Boltzmann equation, but it can easily be  adapted to our situation).
 In the following, we define
 $$
a_\beta (v) := \int_{{\mathbf S}^{d-1} \times \R^ d}  M_\beta(v_1) \, \big((v-v_1) \cdot \nu\big)_+  \, d\nu dv_1 \, .
$$
The proof of the next result   consists in  noticing the decomposition $\mathcal{L} = a _\beta(v) \, \text{Id} - \mathcal{K}$, where~$\text{Id}$ stands for the identity and  $ \mathcal{K}$ is a compact operator.
\begin{lem}
\label{lem : operateur L}
The operator $\mathcal{L}$ is a Fredholm operator of domain $L^2(\R^d, a_\beta M_\beta dv)$ and its kernel reduces to the constant functions.   
In particular, $\mathcal{L}$ is invertible on the set of functions
$$\displaystyle\Big \{  g \in L^2(\R^d, a_\beta M_\beta dv), \,  \int_{\R^d} g(v) \, M_\beta(v) dv = 0 \Big\}\, .$$
\end{lem} 
\noindent Note that the first equation in~(\ref{eqhilbert}) therefore reflects the fact that~$\widetilde\rho_{  0} $ does not depend on~$v$.

\smallskip
\noindent 
We define the vector $b(v) = \big( b_k (v) \big)_{k \leq d}$ with~$ \displaystyle \int_{\R^d} b(v) \, M_\beta(v) dv = 0$, by 
\begin{equation}
\label{eq: vecteur b}
\mathcal{L} b (v) = v .
%\mathcal{L} b_k (v) = v_k \, .
\end{equation}
Returning to~(\ref{eqhilbert}),  we have
$$
\widetilde\rho_{ 1} (\tau,x, v)  =  \rho_{ 1} (\tau,x, v)  + \overline \rho_{ 1} (\tau,x) \, , 
$$
with
$$
 \rho_{ 1} (\tau,x, v) := - b(v) \cdot \nabla_x \widetilde\rho_{  0}  (\tau,x) \quad \mbox{and} \quad   \overline \rho_{ 1}  \in \mbox{Ker} \,  {\mathcal L} \, .
$$
Next we consider the last equation in~(\ref{eqhilbert}) and we notice that for~$\widetilde\rho_{ 2} $ to exist it is necessary for~$ \partial_\tau \widetilde\rho_{  0}  +    v \cdot \nabla_x\widetilde\rho_{ 1} $ to belong to the range of~$\mathcal L$. Since~$\widetilde\rho_{  0} $ does not depend on~$v$, this means that
\begin{equation}\label{eqonrho}
\partial_\tau\widetilde \rho_{  0}  + \int_{\R^d} v \cdot \nabla_x\widetilde \rho_{ 1} (\tau,x, v) M_\beta (v) \, dv = 0Â \, .
\end{equation}
We then define  the diffusion matrix $D(v) = \big( D_{k,\ell} (v) \big)_{k, \ell \leq d}$, again with~$\displaystyle  \int_{\R^d}   D_{k,\ell} (v) \, M_\beta(v) dv =0$,
 by
\begin{equation}
\label{eq: matrix D}
\mathcal{L} D (v) := v \otimes  b (v) - \int_{\R^d}  v \otimes  b (v) \, M_\beta(v) dv  \, .
\end{equation}
 From the symmetry of the model, one can check (see \cite{desvillettes-golse} for instance) that there is a function~$\gamma$ such that 
$$
b (v)= \gamma (|v|) v \, .
$$ 
Then   an easy computation  shows that~$\widetilde\rho_{  0} = \rho $ where
$$
\partial_\tau \rho - \kappa_\beta \Delta_x \rho  = 0 \, , 
$$
while the diffusion coefficient is given by 
\begin{equation}
\label{kappa-def}
\kappa_\beta: = \frac{1}{d}  \int_{\R^d}  v \mathcal{L}^{-1} v \; M_\beta(v) dv 
= \frac{1}{d} \int_{\R^d}  \gamma (|v|) |v|^2 \, M_\beta(v) dv \, ,
\end{equation}
and where we used the symmetry of $b$ to derive the last equality. Finally we have
$$
\widetilde\rho_{ 2} (\tau,x, v)  =  \rho_{ 2} (\tau,x, v)  + \overline \rho_{ 2} (\tau,x) - b(v) \cdot \nabla_x \overline \rho_1(\tau,x)\, , 
$$
with
$$
 \rho_{ 2} (\tau,x, v) := D(v) :  {\rm Hess} \,  \rho (\tau,x)  \quad \mbox{and}Â \quad  \overline \rho_{ 2}   \in \mbox{Ker} \,  {\mathcal L} \, .
$$
 
\subsubsection{Proof of the convergence} \label{proodoftheconvergence}
Now let us prove~(\ref{limitfN1toheat2}).
With the notation introduced in the previous paragraph, let us define
\begin{equation}\label{defPsialpha}
\Psi_\alpha (\tau , x, v):= \rho(\tau,x) + \frac1\alpha  \rho_{  1} (\tau, x,v) +   \frac1{\alpha^2}  \rho_{  2} (\tau, x,v)  \, .
\end{equation}
Then~$\Psi_\alpha$ is almost a solution of \eqref{eq: widetildephialpha}: by construction one has
\begin{equation*}
\d_\tau \Psi_\alpha +\alpha \;  v\cdot \nabla_x \Psi_\alpha + \alpha^2 \; \mathcal{L} \Psi_\alpha = S_\alpha \, ,
\end{equation*}
where the error term $S_\alpha$ is given by
\begin{equation}
\label{eq: erreur}
S_\alpha(\tau,x, v)
: = \frac{1}{\alpha} \big(  \partial_{\tau}\rho_{ 1} (\tau,x, v)
+  v\cdot \nabla_x \rho_{ 2} (\tau,y, v)
+ \frac{1}{\alpha} \partial_{\tau} \rho_{ 2} (\tau,y, v)
 \big)
\, .
\end{equation}
 Defining
 $$
R_\alpha (\tau,x,v):= \Psi_\alpha (\tau,x,v) -  \widetilde \varphi_\alpha (\tau, x,v) 
 $$
 we have thanks to~(\ref{eq: widetildephialpha})
 $$
 \d_\tau R_\alpha +\alpha \;  v\cdot \nabla_x R_\alpha + \alpha^2 \; \mathcal{L} R_\alpha = S_\alpha
 $$
 and the result~(\ref{limitfN1toheat2}) then follows  from the maximum principle which states that
 $$
 \|M_\beta R_\alpha  \|_{L^\infty( [0,T] \times  \T^d \times \R^d)} \leq C(T) \big( \|M_\beta R_\alpha (0 )\|_{L^\infty(    \T^d \times \R^d)}  +   \|M_\beta S_\alpha \|_{L^\infty( [0,T] \times  \T^d \times \R^d)}\big) \, .
 $$
 We note that $S_\alpha$ involves spatial derivatives of $\rho$ of order at most 4,
thus from the maximum principle for the heat equation, each term of $M_\beta S_\alpha$ is bounded in $L^\infty$ norm by~$\alpha^{-1}$.  The same clearly holds for the initial data~$M_\beta R_\alpha (0,x,v)$ since
$$
\begin{aligned}
 R_\alpha (0,x,v) &=   \Psi_\alpha (0,x,v) -  \widetilde \varphi_\alpha (0, x,v) 
 =  \frac1\alpha  \rho_{  1} (0, x,v) +   \frac1{\alpha^2}  \rho_{  2} (0, x,v)  \, .
\end{aligned}
$$
It follows that
 $$
  \|M_\beta (\Psi_\alpha -   \widetilde \varphi_\alpha )\|_{L^\infty( [0,T] \times  \T^d \times \R^d)} \leq \frac{C(T)}{ \alpha}
 $$
and thanks to~(\ref{defPsialpha}),  the convergence result~(\ref{limitfN1toheat2})  is proved.

\begin{Rmk}\label{rem-truc}
We have considered here the case when~$\rho^0 = \rho^0 (x)$. In the case of ill-prepared initial data, namely if~$\rho^0 = \varphi^0 (x,v)$, then the same analysis works provided the following ansatz is used
$$
\Psi_\alpha (\tau , x, v):= \rho(\tau,x) + \frac1\alpha  \rho_{  1} (\tau, x,v) +   \frac1{\alpha^2}  \rho_{  2} (\tau, x,v) + \Pi^\perp \left( e^{-\alpha\tau {\mathcal L}} \varphi^0\right) \, ,
$$
where~$ \Pi^\perp $ is the orthogonal projector onto~$(\rm{Ker}\, {\mathcal L} ) ^\perp$.
\end{Rmk}

\subsection{Convergence to the brownian motion}

Let us denote the tagged particle by
$$
\brown(\tau):= x_1(\alpha \tau) \, .
$$
In the following, $\mathbb{E}_N, \mathbb{P}_N$ will refer to its expectation and probability with respect to the initial data sampled from the density $f_N^0$.
To prove the convergence of the  tagged particle 
to a brownian motion, one needs to check (see  \cite{billingsley}, Chapter 2)
\begin{itemize}
\item the convergence of the marginals of the tagged particle sampled at different times 
\begin{equation}
\label{eq: finite dim x1}
\lim_{N \to \infty} \mathbb{E}_N  \Big(  h_1 \big( \brown ( \tau_1) \big) \dots  
h_\ell \big(  \brown (  \tau_\ell) \big)
 \Big) 
=
\mathbb{E}  \Big(  h_1 \big(  B( \tau_1) \big) \dots  h_\ell \big(  B ( \tau_\ell) \big)  \Big) \, ,
\end{equation}
where $\{ h_1, \dots, h_\ell \}$ is a collection of continuous functions in $\T^d$.
Notice that these marginals refer to time averages and not to the number of particles.
\item the tightness of the sequence, that is for any $\tau \in [0,T]$ 
\begin{equation}
\label{eq: tightness criterion}
\forall \petit >0, \qquad 
\lim_{\eta \to 0} \lim_{N \to \infty}
\mathbb P_N \left( \sup_{ \tau< \sigma< \tau+\eta}  \big| \brown(\sigma) - \brown(\tau)  \big| \geq \petit 
\right) = 0 \, .
\end{equation}
\end{itemize}	
Note that \eqref{eq: finite dim x1} requires to understand time correlations and thus 
we are going to adapt Theorem \ref{long-time} to this new framework.

\medskip

\noindent
{\bf Step 1. Finite dimensional marginals.} $ $  
First, we are going to rewrite the time correlations in terms of collision trees.
A similar approach was devised in Lebowitz, Spohn \cite{LS2} to derive an information on  the true particle trajectories (in the physical space) from the Duhamel series.
Let $t_1 < \dots < t_\ell$ be an increasing collection of times and~$H_\ell = \{ h_1,\dots,h_\ell \}$ a collection of $\ell$ smooth functions.
Define the biased distribution  at time $t >t_\ell$ as follows 
\begin{align}
\label{eq: tilted measure}
\int_{\T ^{Nd} \times \R^{N d}} d Z_N
f_{N,H_\ell} (t, Z_N) & \Phi (Z_N)
: = \mathbb{E}_N  \Big(  h_1 \big( x_1(t_1) \big) \dots  h_\ell \big( x_1(t_\ell) \big)
\, \Phi \big( Z_N (t) \big)  \Big) \\
&= 
\int_{\T ^{Nd} \times \R^{N d}} dZ_N \, f^0_N (Z_N) \; h_1 \big( x_1(t_1) \big) \dots  h_\ell \big( x_1(t_\ell) \big)
\, \Phi \big( Z_N (t) \big)  ,
\nonumber
\end{align}
for any test function $\Phi$.
We stress that by construction the biased distribution $f_{N,H_\ell} (t, Z_N)$
\begin{itemize}
\item  is in general no longer normalized by 1
\item   is symmetric with respect to the $N-1$ last variables.
\end{itemize}
The corresponding marginals are
\begin{equation}
\label{eq: modified distribution}
f^{(s)}_{N,H_\ell} (t, Z_s) := \int f_{N,H_\ell} (t, Z_N) \, dz_{s+1} \dots dz_N \, .
\end{equation}
By construction $f_{N,H_\ell}$ satisfies the Liouville equation for $t >t_\ell$ and the marginals
$f^{(s)}_{N,H_\ell}$ obey the BBGKY hierarchy \eqref{eq: BBGKY} for $s< N$.
Applying the iterated Duhamel formula \eqref{iterated-Duhamel}, we get
\begin{equation}
\label{eq: fHk}
f^{(s)}_{N,H_\ell} (t) =  \sum_{m=0}^{N-s} Q_{s,s+m} (t - t_\ell) f^{(s+m)}_{N,H_\ell} (t_\ell) \, .
\end{equation}
By construction $f_{N,H_\ell} (t_\ell,Z_N) = f_{N,H_{\ell -1}} (t_\ell,Z_N) h_\ell(z_1)$, where the new distribution is now modified by the the first $\ell-1$ functions. 
This procedure can  be iterated up to the initial time.
The backward dynamics can be understood in terms of collision trees which are now weighted by the factor 
$h_1 \big( x_1(t_1) \big) \dots  h_\ell \big( x_1(t_\ell) \big)$ associated with the motion of the tagged particle
\begin{equation}\label{eq: fHk 2}
\begin{aligned}
f^{(1)}_{N,H_\ell} (t) & =  \sum_{m_1 + \dots+ m_\ell =0}^{N-1} Q_{1,1+ m_1} (t - t_\ell) 
\Big( h_\ell Q_{1+ m_1,1+m_2} (t_\ell - t_{\ell-1} ) \Big( h_{\ell -1} \dots    \\
& \qquad \qquad 
Q_{1+ m_1+\dots + m_{\ell-1} ,1+ m_1+ \dots + m_\ell} (t_1) \Big) f^{(1 + m_1+ \dots + m_\ell)}_N (0) \, .
\end{aligned}
\end{equation}
This identity holds for any $N$ and any time.

\medskip

\noindent In order to check \eqref{eq: tightness criterion}, we need also to generalize the identity to consider correlations of the form 
\begin{equation}
\label{eq: temps differents}
\mathbb{E}_N  \big(  h\big( x_1(t_1) - x_1(s) \big) \dots  h \big( x_1(t_\ell) - x_1(s) \big)   \big)
\end{equation}
for a smooth function $h$ with $s < t_1 < \dots < t_\ell$. Using a partition of unity $\{\Gamma_i^\petit \}$ centered at points $\gamma_i \in \T^d$ with mesh $\petit$, one can approximate $h$
$$
h(x -y) = \sum_{i,j} h(\gamma_i - \gamma_j) \Gamma_j^\petit (x)  \Gamma_i^\petit (y) + O(\petit ) \, .
$$
This allows us to use the identity \eqref{eq: fHk 2}
for any accuracy $ \petit >0$ of the approximation. 
Thus~\eqref{eq: temps differents} can be computed in terms of collision trees which are now weighted by the factor~$h\big( x_1(t_1) - x_1(s) \big) \dots  h \big( x_1(t_\ell) - x_1(s) \big)$.

\medskip

\noindent
{\bf Step 2. The limit process.} $ $ 

\noindent
In the Boltzmann Grad limit, the memory of the system is lost and the tagged particle behavior becomes equivalent to a Markov process. We define 
\begin{equation}
\label{eq: makov chain}
\displaystyle \bar x_1(t) =   \bar  x_1(0) + \int_0^t   \bar  v_1(s) \, ds
\end{equation}
as an additive functional of the 
Markov chain $\{ \bar  v_1 (s) \}_{s \geq 0}$ with generator $\alpha \mathcal{L}$ introduced in~\eqref{linear-boltzmann}.
Initially $( \bar  x_1(0), \bar  v_1(0) )$  is distributed according to $\rho^0 (x) M_\beta (v)$. The  expectation associated to this Markov chain is denoted by $\mathbb{E}_{M_\beta}$.

\begin{figure}[h] %  figure placement: here, top, bottom, or page
   \centering
  \includegraphics[width=3in]{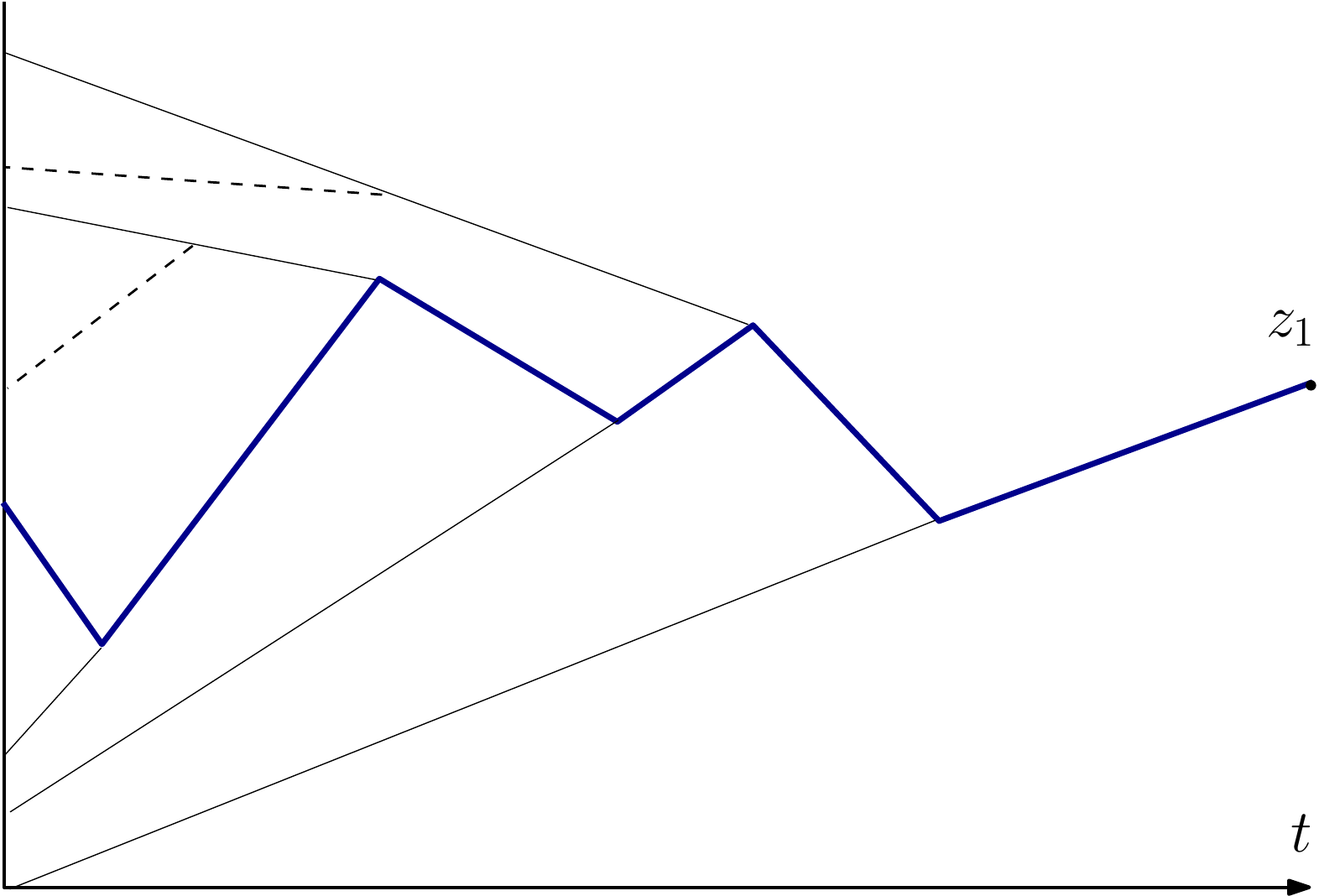} 
\caption{\small A collision tree for the Boltzmann hierarchy is depicted. The path of $z_1$ is the backbone of the tree with branchings at each new collision. 
There cannot be further branches as any collision with a new particle would lead to a cancellation in the collision operator. Thus the trees involving the branches in dashed line do not contribute to the Duhamel series in the Boltzmann hierarchy.
}
\label{fig: arbre markov}
 \end{figure}

Let $t_1 < \dots < t_\ell$ be an increasing collection of times and~$H_\ell = \{ h_1,\dots,h_\ell \}$ a collection of $\ell$ smooth functions. As in \eqref{eq: tilted measure}, we define~$g_{\alpha,H_\ell} (t)$ as the biased  distribution  of the Markov chain $\bar z_1(t)= \big( \bar x_1(t), \bar v_1(t)  \big)$  
\begin{equation*}
\forall t \in [t_k, t_{k+1}[, \qquad 
\int_{\T^{d} \times \R^d }   g_{\alpha,H_\ell} (t,  z) \, \Phi ( z) \, dz
=
\mathbb{E}_{M_\beta}  \Big(  h_1 \big( \bar  x_1 (t_1) \big) \dots  h_\ell \big( \bar  x_1(t_\ell) \big)
\, \Phi \big( \bar z_1 (t) \big)  \Big) 
\, ,
\end{equation*}
with $t_{\ell+1} = \infty$. 
One can consider a measure (cf. \eqref{solutionBoltzmannlinear}) including as well the background density of an ideal gas.
The marginals of this measure are
\begin{equation}
   g_{\alpha,H_\ell}^{(s)} (t,Z_s) = g_{\alpha,H_\ell} (t,z_1) \displaystyle \prod_{i = 2}^s M_\beta(v_i)  \, .
\end{equation}
As in \eqref{eq: fHk 2}, the distribution can be rewritten in terms of a Duhamel series
\begin{equation}
\label{eq: fHk 2 0}
\begin{aligned}
g_{\alpha,H_\ell} (t) & =  \sum_{m_1 + \dots+ m_\ell =0}^{\infty} Q^0_{1,1+ m_1} (t - t_\ell) 
\Big( h_\ell Q^0_{1+ m_1,1+m_2} (t_\ell - t_{\ell-1} ) \Big( h_{\ell -1} \dots    \\
& \qquad \qquad 
Q^0_{1+ m_1+\dots + m_{\ell-1} ,1+ m_1+ \dots + m_\ell} (t_1) \Big) g_{\alpha,H_\ell}^{(1 +m_1+ \dots + m_\ell)} (0) \, .
\end{aligned}
\end{equation}
This representation allows us to rephrase the Markov chain expectations in terms of the Boltzmann hierarchy. 
In this series, a lot of cancelations occur (see figure \ref{fig: arbre markov}). 
Indeed, the only relevant collision trees are made of a single  backbone (the trajectory of $z^0_1$) with branches representing the collisions of $z^0_1$ with the ideal gas, but no further ramification.
In step 3, we shall not use these cancellations and simply compare the series \eqref{eq: fHk 2} and \eqref{eq: fHk 2 0} term by term in order to show that in the Boltzmann Grad limit when $\alpha \ll \sqrt{\log \log N}$
\begin{equation}
\label{eq: finite dim esperance}
\lim_{\alpha \to \infty} 
\mathbb{E}_{M_\beta}  \Big(  h_1 \big( \bar x_1( \alpha  \tau_1) \big) \dots  
h_\ell \big(  \bar x_1 (\alpha \tau_\ell) \big)
 \Big) 
 - 
 \mathbb{E}_N  \Big(  h_1 \big(  x_1( \alpha  \tau_1) \big) \dots  
h_\ell \big(  x_1 (\alpha \tau_\ell) \big)
 \Big) 
= 0 .
\end{equation}

\noindent As~$\mathcal{L}$ has a spectral gap, the invariance principle holds for the position of the Markov process $\bar x_1$   
(see \cite{KLO}  Theorem 2.32 page 74).
This implies the convergence of the rescaled finite dimensional marginals
towards the ones of the brownian motion $B$ with variance $\kappa_\beta$ (see \eqref{kappa-def}), i.e.
that for any smooth functions $\{h_1,\dots , h_\ell\} $ defined in~$\T^d$, 
\begin{equation}
\label{eq: finite dim x01}
\lim_{\alpha \to \infty} 
\mathbb{E}_{M_\beta}   \Big(  h_1 \big( \bar x_1( \alpha  \tau_1) \big) \dots  
h_\ell \big(  \bar x_1 (\alpha \tau_\ell) \big)
 \Big) 
=
\mathbb{E}  \Big(  h_1 \big(  B( \tau_1) \big) \dots  h_\ell \big(  B ( \tau_\ell) \big)  \Big) \, .
\end{equation}
The diffusion coefficient $\kappa_\beta$ defined in \eqref{kappa-def} can be interpreted in terms of  
the variance of the position properly rescaled in time (see \cite{KLO} page 47).
%: the Maxwellian distribution $M_\beta$ is the invariant measure of this process and 
%the diffusion coefficient $\kappa_\beta$  can be recovered as the variance of the position for any coordinate
%\begin{equation*}
%\kappa_\beta =
%% \lim_{\alpha \to \infty} \mathbb{E}_{M_\beta}  
%%\left[ 
%%\left( \frac{1}{ \sqrt{\alpha^2 \tau}} \int_0^{\alpha \tau} \bar v_{1, k}(s) \, ds
%%\right)^2
%%\right] =  
%\frac{1}{d} \mathbb{E}_{M_\beta}  
%\left[ \bar v_{1} \, {\mathcal L}^{-1}  \bar v_{1}  \right]
%\, .
%\end{equation*}
%Note that the normalization of the variance comes from the fact that the generator of the process is $\alpha \cL$, i.e. that the time scale is speeded-up by $\alpha$.

\medskip

\noindent
{\bf Step 3. Approximation of the finite dimensional marginals.} 

\noindent We turn now to the proof of \eqref{eq: finite dim esperance} which combined with \eqref{eq: finite dim x01} will 
show the convergence  \eqref{eq: finite dim x1} of the marginals of the tagged particle sampled at different times.

\medskip

\noindent Suppose now that the collection $H_\ell$ satisfies the uniform bounds on $\T^d $
\begin{equation}
\label{eq: bornes H}
\forall i \leq \ell, \qquad 
0 \leq h_i(x_1) \leq m .
% \, , \qquad \big| \nabla_{x_1} h_i(x_1) \big| \leq  \delta^\prime\, .
\end{equation}
Thus the $f^{(s)}_{N,H_\ell}$ %defined in \eqref{eq: modified distribution} 
satisfy the maximum principle \eqref{upper-bound} with an extra factor $m^{\ell}$.  The pruning procedure  on the collision trees therefore applies also in this case and enables us to restrict to trees with at most $A^{K+1}$ collisions during the time interval~$[0,t]$.
Furthermore, the comparison of the trajectories for  $f^{(1)}_{N,H_\ell}$ and $g_{\alpha,H_\ell}$
can be achieved in the same way as before on a tree with less than $A^{K+1}$ collisions  and no recollisions. 
Analogous bounds as in Proposition~\ref{lastapprox250000} can be obtained, but one has to take into account that the trees are now weighted by $ h_1 \big(  x^0_1(\alpha   \tau_1) \big) \dots  h_\ell \big( x^0_1 ( \alpha   \tau_\ell) \big)$.
We recall that the pseudo-trajectories of $x_1$ and $x^0_1$ coincide at any time,
thus  bounds similar to \eqref{eq: approx temps gd} hold  
 \begin{equation}
\label{eq: comparaison x01}
\big\| f_{N, H_\ell}^{(1)} (\alpha \tau,   y, v) - 
g_{\alpha, H_\ell} (\alpha \tau,   y, v) \big\|_{L^\infty(  \T^d \times \R^d)} 
\leq  C   m^\ell \, \| \rho^0 \|_\infty \, 
\left[ \frac{ \tau \alpha^2}{ ({\log\log N})^{\frac{A-1}{A} } } \right]^{\frac{A^2}{A-1}} \, \cdotp
\end{equation}
This implies \eqref{eq: finite dim esperance}, hence \eqref{eq: finite dim x1} thanks to 
\eqref{eq: finite dim x01}.

%Combining  \eqref{eq: comparaison x01} and \eqref{eq: finite dim x01}, we deduce directly the convergence of the finite dimensional marginals~\eqref{eq: finite dim x1}.

\medskip

\noindent
{\bf Step 4. Tightness.}

\noindent In order to evaluate \eqref{eq: tightness criterion}, it is enough to sample the trajectory of the tagged particle at the times $\tau_i = \{ \tau + \frac{u_N}{\alpha} \,  i  \}_{i \leq \ell_N}$ for  $u_N = \frac{1}{\log N}$ and with $\ell_N :=  \alpha \eta / u_N $. 
Indeed, we can decompose the path deviations into two terms
\begin{align}
\label{eq: decomposition temps}
\mathbb P_N \left( \sup_{ \tau< \sigma< \tau+\eta}  \big| \brown(\sigma) - \brown(\tau)  \big| \geq 2 \petit 
\right) \leq  & 
1 - \mathbb P_N \left( \bigcap_{i = 1}^{\ell_N} \;  \big\{   \big| \brown( \tau_i) - \brown(\tau)  \big| < \petit \big\}  \right)\\
& \qquad + \sum_{i = 1}^{\ell_N} 
\mathbb P_N \left( \sup_{ \tau_i < \sigma < \tau_{i+1}}  \big| \brown(\sigma) - \brown( \tau_i)  \big| \geq  \petit 
\right) \, . \nonumber
\end{align}
	
\medskip

\noindent We shall first evaluate the last term in the right-hand side which involves only events occurring in a microscopic time scale of length $u_N$. Given $i \leq \ell_N$, let $t_i := i u_N + \alpha\tau$ and 
$t_{i+1} := u_N + t_i$ then
\begin{equation}
\mathbb P_N \left( \sup_{ \tau_i < \sigma < \tau_{i+1}}  \big| \brown(\sigma) - \brown( \tau_i)  \big| \geq  \petit \right) 
= 
\mathbb P_N \left( \sup_{ t_i < s < t_{i+1}}  \big| x_1 (s) - x_1 (t_i) \big| 
\geq   \petit \right) 
%= \mathbb P \left( \sup_{ t_i < s < t_{i+1}}  \big| \int_{t_i}^{t_{i+1}} v_1 (s) ds  \big| 
%\geq   \petit \right)
. \nonumber
\end{equation}
%Let $v^* = \max \{ |v_i|, \quad i \leq N\}$ be the modulus of the maximal velocity in the system.
In order to control the tagged particle fluctuations, it is enough to bound its velocity  in the time interval $[t_i, t_{i+1}]$
\begin{align*}
%\mathbb P \left( \sup_{ t_i < s < t_{i+1}}  \big| x_1 (s) - x_1 (t_i) \big| 
%\geq   \petit \right) = 
\mathbb P_N \left( \sup_{ t_i < s < t_{i+1}}  \big| \int_{t_i}^s v_1 (s') ds'  \big| 
\geq   \petit \right)
& \leq 
\mathbb P_N \left( \int_{t_i}^{ t_{i+1} }  \big| v_1 (s')  \big|  ds' \geq   \petit \right)\\
& \leq \| \rho^0 \|_\infty \, 
\widehat{\mathbb P}_N \left( \int_0^{u_N } \big| v_1 (s')  \big|    ds' 
\geq   \petit \right)
, 
\end{align*}
where we used the  maximum principle in the last inequality and $\widehat {\mathbb P}_N$ denotes the dynamics starting from the invariant measure $M_{N, \beta}$.
Following the strategy in \cite{alexander} to bound this probability, we write
\begin{align*}
\widehat{\mathbb P}_N \left( \int_0^{u_N }  \big| v_1 (s')  \big|    ds' 
\geq   \petit \right)
\leq \exp \left( - \frac{\petit}{u_N } \right)
\widehat{\mathbb E}_N \left( \exp \left( \frac{1}{u_N }  \int_0^{u_N } \big| v_1 (s')  \big|   ds'  \right) \right).
\end{align*}
Using Jensen's inequality and the invariant measure, we get
\begin{align*}
\widehat{\mathbb E}_N \left( \exp \left( \frac{1}{u_N }  \int_0^{u_N }  \big| v_1 (s')  \big|    ds'  \right) \right)
\leq 
\frac{1}{u_N }  \int_0^{u_N }  ds' \,
\widehat{\mathbb E}_N \left( \exp \left(   \big| v_1 (s')  \big|     \right) \right)
= \mathbb E_{M_{N,\beta}} \left( \exp \left(   \big| v_1  \big|  \right) \right)  \leq c_\beta,
\end{align*}
where $c_\beta$ is a constant depending only on $\beta$.
Since $u_N = \frac{1}{\log N}$, we have shown that for any~$\petit>0$, the probability of a deviation in a very short time vanishes when $N$ goes to infinity
\begin{align*}
\sum_{i =1}^{\ell_N}
\mathbb P_N \left( \sup_{ \tau_i < \sigma < \tau_{i+1}}  \big| \brown(\sigma) - \brown( \tau_i)  \big| \geq  \petit \right)
\leq  c_\beta \, \ell_N \, \exp \left( - \petit \, \log N \right)
\xrightarrow[N \to \infty ]{} 0.
\end{align*}

\bigskip

\noindent The tightness for the process $\bar x_1$ derived in \cite{KLO}  (Theorem 2.32 page 74) implies that for any~$\petit >0$ and $\ell_N =  \alpha \eta / u_N $
\begin{align*}
\lim_{\eta \to 0} 
\lim_{N \to \infty} 
\mathbb P_{M_\beta} \left( \bigcap_{i = 1}^{\ell_N} \;  \big\{   \big| \bar  x_1(\alpha  \tau_i) - \bar x_1(\alpha  \tau)  \big| < \petit  / 2 \; \big\}  \right)
= 1 \, . 
\end{align*}
By comparison, we are going to show that the same result holds also for the tagged particle $x_1$.
Using \eqref{eq: decomposition temps}, this will complete the proof of \eqref{eq: tightness criterion}.

Let $h_{ \petit} (w)= 1_{ \{ |w| \leq \petit /2 \}}$, then it is enough to show that 
\begin{equation}
\label{eq: comparaison x02}
\begin{aligned}
& \Big | \mathbb{E}_N  \left(  \prod_{i = 1}^{\ell_N} h_{ \petit}  \big( x_1 (\alpha\tau_i) - 
x_1 (\alpha \tau) \big)   \right)
-
\mathbb{E}_{M_\beta}  \left(  \prod_{i = 1}^{\ell_N} h_{  \petit}  \big(\bar x_1(\alpha\tau_i) - 
\bar  x_1(\alpha \tau) \big)   \right) \Big| \\
& \qquad \qquad \qquad  \qquad \qquad \leq
C \| \rho^0 \|_\infty \;
\left[ \frac{ \tau \alpha^2}{ ({\log\log N})^{\frac{A-1}{A} } } \right]^{\frac{A^2}{A-1}} \cdotp
\end{aligned}
\end{equation}
At this stage, it is enough to use the fact that probabilities of the form \eqref{eq: temps differents} can also be evaluated in terms of weighted trees as in Step 1.
Since $h_{ \petit}$ is bounded by 1, the maximum principle applies uniformly in~${\ell_N}$.
The tree decomposition and the reduction to non pathological trajectories hold as in the previous proof. 
For good pseudo-trajectories, the paths of  $x_1$ and $x^0_1$ coincide, therefore modifying the Duhamel series by $\prod_{i = 1}^{\ell_N} h_{ \petit}  \big( x_1 (\alpha\tau_i) - 
x_1 (\alpha \tau) \big)  $ does not alter the comparison established in 
Proposition \ref{lastapprox250000}.
This concludes the proof of tightness.

%%%%%%%%%%%%%%%%%%%%%%%%%%%%%%%%%%%%%%%%%%%%%%%%
%%%%%%%%%%%%%%%%%%%%%%%%%%%%%%%%%%%%%%%%%%%%%%%%

\appendix

\bigskip\bigskip

\section{Asymptotic control of the exclusion}\label{proofpropappendix}

\noindent For the sake of completeness, we recall here the proof  of Proposition \ref{exclusion-prop}. We omit all subscripts~$\beta$ to simplify the presentation.

	\medskip
	\noindent
	$\bullet$ {\it First step: asymptotic behaviour of the partition function.}
	
	\smallskip
	\noindent
We first prove that in the scaling $N \e^{d-1}\ \equiv \alpha,$ with~$\alpha \ll 1/\varepsilon$,
\begin{equation}
\label{Z-est}
  1 \leq   \cZ_N^{-1}  \cZ_{N-s} \leq\big( 1 - \e \alpha \k_d \big)^{-s},
\end{equation}
 where  $\k_d$ denotes the volume of the unit ball in $\R^d.$ The first inequality is due to the immediate upper bound 
 $$
   \cZ_N \leq    \cZ_{N-s} \, .
 $$
Let us prove the second inequality. We have  by definition
$$
 \begin{aligned} 
    {\mathcal Z}_{s+1}  =  \int_{ \T^{d(s+1)} }     \Big(\prod_{1 \leq i \neq j \leq s+1} \indc_{|x_i-x_{j} |>\eps}\Big)   \, dX_{s+1} \, .
\end{aligned}
$$ 
By Fubini's equality, we deduce
 $$ \begin{aligned}   {\mathcal Z}_{s+1}  = \int_{ \T^{ds}  }  \left( \int_{ \T^{d}  } \Big(\prod_{1 \leq i \leq s} \indc_{|x_i-x_{s+1} |>\eps}\Big)  \, dx_{s+1}\right) \Big(\prod_{1 \leq i\neq j \leq s}   \indc_{|x_i-x_{j} |>\eps}\Big)  \, dX_{s} \, .\end{aligned}$$
Since
 $$ \int_{ \T^{d}  } \Big(\prod_{1 \leq i \leq s} \indc_{|x_i-x_{s+1} |>\eps}\Big)\, dx_{s+1}\geq 
1-  \k_d s \eps^d    \, ,$$
   we deduce   the lower bound
$$   {\mathcal Z}_{s+1} \geq
  {\mathcal Z}_{s}   (
1-  \k_d s \eps^d   ) \geq
  {\mathcal Z}_{s}   
(1-\k_d \eps  \alpha ) \, ,
$$ 
where we used $s \leq N$ and the scaling $N \e^{d-1} \equiv \alpha.$ 
This implies by induction
$$
  {\mathcal Z}_N  \geq   {\mathcal Z}_{N-s}   \big( 1 - \e \alpha \k_d  \big)^s \, .
$$
 That proves~(\ref{Z-est}).
%We have  the lower bound
%$$
%\begin{aligned}
% \bar \cZ_{s+1} &=   \int  \left( \frac\beta 2\right)^{(s+1)d/2}\exp (-\beta  H_{s+1} (Z_{s+1})) dZ_{s+1}= \int \exp\Big(-\beta \sum _{1\leq i<j\leq s+1} \Phi_\eps(x_i-x_j) \Big) dX_{s+1} \\
%& \geq  \int_{\T^{d(s+1)}}  \exp\Big(-\beta \sum _{1\leq i<j\leq s} \Phi_\eps(x_i-x_j) \Big)  \Big( \prod_{i=1}^s \indc_{|x_i-x_{s+1}| >\eps} \Big)    \, dX_{s+1} \, .
% \end{aligned}
%$$ 
%By Fubini, we have
% $$ \begin{aligned}  \int_{\T^{d(s+1)}}  &\exp\Big(-\beta \sum _{1\leq i<j\leq s} \Phi_\eps(x_i-x_j) \Big)  \Big( \prod_{i=1}^s \indc_{|x_i-x_{s+1}| >\eps} \Big)   \, dX_{s+1}  \\ & = \int_{\T^{ds}} \left( \int_{\T^{d}} \Big(\prod_{1 \leq i \leq s} \indc_{|x_i-x_{s+1} |>\eps}\Big)   dx_{s+1}\right)\exp\Big(-\beta \sum _{1\leq i<j\leq s} \Phi_\eps(x_i-x_j) \Big) dZ_s\\
% & \geq \bar \cZ_s (1-\k s \eps^d)\, ,\end{aligned}$$
%implying by induction
%$$\bar \cZ_N  \geq \bar \cZ_{N-s}  \prod_{j = N-s}^{N-1} (1 - j \eps^d \kappa_d )  \geq \bar \cZ_{N-s} \big( 1 - \eps \lambda^d \kappa_d \big)^s \, ,$$
%where we used $s \leq N$ and the scaling $N \eps^{d-1}\lambda^{-d} \equiv 1.$  

	\medskip

	\noindent
	$\bullet$ {\it Second step: convergence of the marginals.}

	\medskip

\noindent Let us introduce the short-hand notation
$$
dZ_{(s+1,N)}  :=dz_{s+1} \dots dz_N \, .
$$
We compute for $s \leq N$ 
$$  
 \begin{aligned} 
 M_N^{(s)} (Z_s)  &=   {\mathcal Z}_N^{-1} 
 \indc_{Z_s \in {\mathcal D}_\eps^s}\left(\frac{\beta }{2\pi}\right)^{ \frac{sd}2}\exp \left(-\frac\beta2 |V_s|^2\right)  
\\
  & \times \int_{  \R^{d(N-s)}} 
\left(\frac{\beta }{2\pi}\right)^{ \frac{(N-s)d}2} \exp \left(-\frac\beta2 \sum_{i=s+1}^N|v_i|^2\right)
  \, dV_{(s+1,N)}\\
 & \quad  \times  \int_{ \T^{d(N-s)} } \left(  \prod_{s+1 \leq i \neq j \leq N} \indc_{|x_i - x_j| > \eps} \right) \left(  \prod_{i '\leq s <  j'}  \indc_{|x_{i'} - x_{j'}| > \eps} \right)
dX_{(s+1,N)} \, .
  \end{aligned}$$
  We deduce, by   symmetry, 
\begin{equation} \label{dec:tensor}   
 M_N^{(s)} = {\mathcal Z}_N^{-1} \indc_{Z_s \in {\mathcal D}_\eps^s}M^{\otimes s}   \Big(  {\mathcal Z}_{N-s} -  {\mathcal Z}^\flat_{(s+1,N)} \Big)
\end{equation} 
 with the notation
$$
 \begin{aligned}  {\mathcal Z}^\flat_{(s+1,N)} & :=  \int_{\T^{d(N-s) } }
  \Big(1 - \prod_{i \leq s < j}\indc_{|x_i - x_j| > \eps}\Big)   \prod_{s+1 \leq k \neq \ell \leq N} \indc_{|x_k- x_\ell| > \eps}\, dX_{(s+1,N)} \, .\end{aligned}
 $$
\noindent From there, the difference $ \indc_{Z_s \in {\mathcal D}_\eps^s}M^{\otimes s} 
-   M_{N}^{(s)}$ decomposes as a sum
\begin{equation}
\label{marginal-dec1}
\begin{aligned} 
 \indc_{Z_s \in {\mathcal D}_\eps^s}M^{\otimes s} 
-   M_{N}^{(s)} = \Big(1 -   {\mathcal Z}_{N}^{-1}  {\mathcal Z}_{N-s}\Big) \indc_{Z_s \in {\mathcal D}_\eps^s}M^{\otimes s} \\ {}+     {\mathcal Z}_N^{-1}  {\mathcal Z}^\flat_{(s+1,N)} \indc_{Z_s \in {\mathcal D}_\eps^s}M^{\otimes s}
 \, .
\end{aligned}
\end{equation} 
By~(\ref{Z-est}), there holds $1-   {\mathcal Z}_{N}^{-1}  {\mathcal Z}_{N-s} \to 0$ as $N \to \infty,$ for fixed $s.$ Since $M^{\otimes s}$ is uniformly bounded,   this implies that the first term in the right-hand side of \eqref{marginal-dec1} tends to 0 as~$N $ goes to~$ \infty$. 
Besides, by
 $$\dsp{0 \leq  1 - \prod_{i \leq s < j}\indc_{|x_i - x_j| > \eps} \leq   \sum_{i \leq s < j} \indc_{|x_i-x_j|<\eps}\,,}$$
 we bound 
  $$ 
  \begin{aligned}
   {\mathcal Z}^\flat_{(s+1,N)}   \leq \sum_{1 \leq i \leq s}\int_{\T^{d(N-s)}} \Big(\sum_{s+1 \leq j \leq N} \indc_{|x_i-x_j|<\eps}\Big) \prod_{s+1 \leq  k \neq \ell\leq N} \indc_{|x_k - x_ \ell | > \eps}  \,  dX_{(s+1,N)}
\,   .
  \end{aligned}
$$
 Given $1 \leq i\leq s,$ there holds by symmetry  and Fubini's equality,
 $$ \begin{aligned}
 & \int_{\T^{d(N-s)}  }\Big(\sum_{s+1 \leq j \leq N} \indc_{|x_i-x_j|<\eps}\Big) \prod_{s+1 \leq  k \neq l \leq N} \indc_{|x_k - x_l| > \eps}\   dX_{(s+1,N)}\\   &\leq (N-s) \int_{\T^{d }  } \indc_{|x_i - x_{s+1}| < \e}  \,   dx_{s+1}   \, \int_{\T^{d(N-s-1)}  } \prod_{s+2\leq  k \neq l \leq N} \indc_{|x_k - x_l| > \eps}   \, dX_{(s+2,N)} \\ &  = (N-s) \int_{\T^{d }  } \indc_{|x_i - x_{s+1}| < \e} \,  dx_{s+1}  \, \times \,  {\mathcal Z}_{N-s-1} \, ,
 \end{aligned}$$
 so that
 \begin{equation} \label{bd:flat}  {\mathcal Z}^\flat_{(s+1,N)} \leq s (N-s) \e^d \k_d  {\mathcal Z}_{N-s-1} \, .
 \end{equation}
  By~(\ref{Z-est}), we obtain
$$ % \begin{equation} \label{bd:z-n-s-bis}
    {\mathcal Z}_N^{-1}  {\mathcal Z}^\flat_{(s+1,N)}  \leq \e  \alpha s \k_d \big(1 - \e \alpha\k_d \big)^{-(s+1)} \, ,
$$%  \end{equation}
  and the upper bound tends to 0 as $N \to \infty,$ for fixed $s.$ This implies convergence to 0 of the second term in the right-hand side of \eqref{marginal-dec1}.
This completes  the proof  of Proposition \ref{exclusion-prop}.
\qed

%%%%%%%%%%%%%%%%%%%%%%%%%%%%%%%%%%%%%%%%%%%%%%%%
%%%%%%%%%%%%%%%%%%%%%%%%%%%%%%%%%%%%%%%%%%%%%%%%
\bigskip
\bigskip
\section{Recollisions in the torus}

We show here how to adapt the arguments of \cite{GSRT} to prove Lemma \ref{geometric-lem1}.

\medskip\noindent
    $ \bullet $ To build the set of ``bad velocities", we use the correspondence between the torus and the whole space with periodic structure. Asking that there exists $u \in [0,t]$ such that
$$ d\big ((x_1-v_1 u ),(  x _2-v_2 u )\big) \leq \eps\,,$$
boils down to having
$$ (x_1-v_1 u )-(  x _2-v_2 u ) \in \bigcup_{k\in \Z^d} B_{\eps}(  k)\,.$$
Then, by the triangular inequality and provided that $\eps<\bar a$,
$$( x^0_1-v_1 u )-(   x^0_2-v_2 u ) \in \bigcup_{k\in \Z^d} B_{3\bar a}(k )  \,.$$
Now, since $|v_1-v_2|\leq 2E$ and $u  \in [0,t]$, this implies that
$$ s (v_1-v_2) \in \left( \bigcup_{k\in \Z^d} B_{3 \bar a}( x^0_1- x^0_2 +  k)\right) \cap B_{0}(2Et)\,.$$
In other words, $v_1-v_2$ has to belong to a finite union of cones of vertex 0
\begin{itemize}
\item at most one of which is of solid angle $ ( \bar a /\eps_0)^{d-1}$;
\item the other ones (at most $(4E t)^d$) are of solid angle~$   c \, \bar a^{d-1}$.
\end{itemize}
The intersection $K(\bar x_1-\bar x_2, \eps_0, \bar a) $ of these  cones and of the sphere of radius $2E$  is of size 
 $$|K(\bar x_1-\bar x_2, \eps_0, \bar a) | 
 \leq  
 CE^d \left(  \Big( \frac{\bar a}{\eps_0}\Big)^{d-1} + (E t )^d {\bar a}^{d-1}  \right) \,.$$

\medskip\noindent
 $ \bullet $  In order to prove the second estimate, we need to refine a little bit the previous argument. Asking that there exists $u \in [\delta,t]$ such that
$$ d((x_1-v_1 u ),(  x _2-v_2 u) )\leq \eps_0\,,$$
boils down to having
\begin{equation}
\label{k-cond}
 u  (v_1-v_2) \in  B_{3\eps_0}( x^0_1-  x^0_2+   k)\,,
\end{equation}
for some $k \in \Z^d \cap B_{2Et }( x^0_2 -  x^0_1)$.

\begin{itemize}
\item If $| x^0_1- x^0_2 +  k| \geq 1/4$,  condition (\ref{k-cond}) implies that $v_1-v_2$ belongs to the intersection of $B_{2E}(0)$ and some cone of vertex 0 and solid angle $\eps_0^{d-1}$.

\item
If $| x^0_1- x^0_2 +   k| \leq 1/4$ (which can happen only for one value of $k$),  denoting by $n$ any unit vector normal to $\bar x_1-\bar x_2 + k$, we deduce from  (\ref{k-cond}) that
$$u  |(v_1-v_2) \cdot n | \leq 3\eps_0$$
from which we deduce that $v_1-v_2$ belongs to the intersection of $B_{2E}(0)$ and some cylinder of radius $\eps_0/\delta$.
\end{itemize}

The union $K_\delta( x^0_1- x^0_2, \eps_0,  \bar a)$ of these ``bad" sets is therefore of size
$$|K_\delta( x^0_1- x^0_2, \eps_0,  \bar a) | \leq  CE \left(  \left( \frac{\eps_0}{\delta} \right)^{d-1} +  E^{d-1}  
\big(E t \big)^d {\eps_0}^{d-1}\right) \,.$$
 The lemma is proved.
\qed

\bigskip
\noindent
{\bf Acknowledgments.} The authors are grateful to Laurent Desvillettes, Fran\c cois Golse, Herbert Spohn and Balint Toth for fruitful discussions and for suggesting interesting bibliographical references. Finally we extend our thanks to the anonymous referees for suggesting many improvements in the manuscript.
The work of I.G. has been supported by the grant ANR-12-BS01-0013-01.

\end{document}